\documentclass[12pt, a4paper]{article}
\usepackage[utf8]{inputenc} % Encoding
\usepackage[T1]{fontenc} % Font encoding
\usepackage{amsmath, amssymb, amsthm} % Math symbols and environments
\usepackage{mathtools} % Enhanced math functionality
\usepackage{bbm} % For indicator functions
\usepackage{bm} % Bold math symbols
\usepackage{graphicx} % Include graphics
\usepackage{caption} % Custom captions
\usepackage{subcaption} % Subfigures
\usepackage{float} % Control float positions
\usepackage{booktabs} % Better tables
\usepackage{array} % Table formatting
\usepackage{xcolor} % Color support
\usepackage[hidelinks]{hyperref} % Hyperlinks
\usepackage{enumitem} % Better lists
\usepackage{pgfplots} % Plots
\pgfplotsset{compat=1.18} % Ensure compatibility
\usepackage{multicol} % Multi-column text
\usepackage{algorithm,algpseudocode} % Algorithms
\algblock{Input}{EndInput}
\algnotext{EndInput}
\algblock{Output}{EndOutput}
\algnotext{EndOutput}
\algrenewcommand\algorithmicrequire{\textbf{Input:}}
\algrenewcommand\algorithmicensure{\textbf{Output:}}
\usepackage{cleveref} % Smart referencing
\usepackage{siunitx} % SI units
\usepackage[round]{natbib} % For literature references
\usepackage{tikz}
\usepackage{setspace}
\definecolor{forestgreen}{RGB}{34,139,34}

\usepackage[margin=1in]{geometry}

\usepackage{setspace}
\newtheorem{theorem}{Theorem}
\newtheorem{lemma}{Lemma}
\newtheorem{corollary}{Corollary}
\newtheorem{proposition}{Proposition}
\theoremstyle{definition}
\newtheorem{definition}{Definition}
\newtheorem{example}{Example}
\newtheorem{remark}{Remark}
\newtheorem{observation}{Observation}

\newcommand{\ind}[1]{\mathbbm{1}_{#1}} % Indicator function
\newcommand{\defeq}{\vcentcolon=}

\title{Incomplete U-Statistics of Equireplicate Designs:\\Berry-Esseen Bound and Efficient Construction}

  \author{Cesare Miglioli and Jordan Awan\\
  Department of Statistics, University of Pittsburgh}

\date{}
\begin{document}
\maketitle
\begin{abstract}
\noindent U-statistics are a fundamental class of estimators that generalize the sample mean and underpin much of nonparametric statistics. Although extensively studied in both statistics and probability, key challenges remain: their high computational cost—addressed partly through incomplete U-statistics—and their non-standard asymptotic behavior in the degenerate case, which typically requires resampling methods for hypothesis testing. This paper presents a novel perspective on U-statistics, grounded in hypergraph theory and combinatorial designs. We identify the dependency graph of an incomplete U-statistic with the line graph of its design hypergraph and use Stein's method to derive Berry-Esseen bounds for deterministic designs under full degeneracy. We further extend these results to studentized statistics using a variance estimator tailored to the fully degenerate setting. The resulting bounds establish Gaussian limits for superlinear design sizes with fixed kernel order and accommodate diverging order under separate conditions. We also develop efficient algorithms for constructing equireplicate designs, a subclass of deterministic designs, which yield minimum-variance incomplete U-statistics in the second-order case. Beyond its theoretical contributions, our framework provides permutation-free alternatives to tests based on fully degenerate U-statistics, as demonstrated in experiments with the two-sample Maximum Mean Discrepancy test and a one-sample mean test based on a higher-order U-statistic.
\end{abstract}

\section{Introduction}

U-statistics are a broad class of statistical estimators that extend the concept of the sample mean. Instead of simply averaging the observations, they average a symmetric, measurable function (kernel) of \(k > 1\) arguments over all \(\binom{n}{k}\) possible subsets of a sample of size \(n\). To mitigate the computational burden, \emph{Incomplete U-statistics} \citep{blom1976some} consider only a subset, referred to as a \emph{design}, of these \(\binom{n}{k}\) elements. While in principle the design can be strategically chosen to minimize the variance of the estimator, for ease of implementation, a random selection—either with or without replacement—is commonly used, despite not being the most efficient approach (see \cite{lee1990u}, ch. 4.3).

The theoretical properties of U-statistics have been extensively studied in both Statistics \citep{lee1990u} and Probability \citep{korolyuk2013theory}. In particular, \emph{non-degenerate U-statistics} have a Gaussian limiting  distribution \citep{hoeffding1948u}, while \textit{degenerate U-statistics} converge to non-standard asymptotic distributions (\cite{lee1990u}, ch. 3.2.3). For degenerate second-order U-statistics, i.e., when $k=2$, \cite{gregory1977large} showed that the limiting distribution is an infinite weighted sum of centered $\chi^2$ random variables. \cite{weber1981incomplete} extends the analysis to second-order degenerate incomplete U-statistics, demonstrating that both normal and weighted chi-square limiting behaviors can occur depending on the choice of the design. The latter non-standard distribution is often considered intractable for practical purposes; this is because the weights depend on the eigenvalues of a kernel-specific integral equation and, except for a few ``nice'' choices of the kernel function (\cite{lee1990u}, ch. 6.2), this equation cannot be solved analytically. 
 
U-statistics have found widespread applications across various fields, from classical statistics: e.g., nonparametric and high-dimensional hypothesis testing \citep{bergsma2014consistent, yao2018testing, he2021asymptotically}, robust statistics \citep{joly2016robust, MinskerWei2020}, bootstrap and resampling methods \citep{leucht2013dependent, bastian2024testing}, to modern machine learning applications: e.g., in empirical risk minimization \citep{clemenccon2008ranking, chen2023distributed}, supervised learning ensembles \citep{mentch2016quantifying, peng2022rates} and kernel methods \citep{gretton2005measuring,gretton2012kernel, liu2016kernelized}. 
Second-order U-statistics are the most common and widely used in practice. This class includes estimators of  variance, covariance, Kendall’s Tau, Spearman’s Rho, Wilcoxon statistics, Cramer-Von Mises statistics. They also include estimators of the Maximum Mean Discrepancy (MMD) and the Kernel Stein Discrepancy, which are kernel methods designed for two-sample testing \citep{gretton2012kernel} and goodness-of-fit testing \citep{liu2016kernelized}, respectively. Higher-order U-statistics are also relevant in applications, as they include estimators of dependence measures such as distance covariance \citep{szekely2007measuring}, sign covariance \citep{bergsma2014consistent}, and the Hilbert–Schmidt Independence Criterion (HSIC) \citep{gretton2005measuring}.

Given the extensive applicability of U-statistics, there has been growing interest in studying incomplete U-statistics, both from a statistical perspective as in \cite{chen2019randomized,sturma2024testing, leung2026berry} and also from an applied machine learning perspective as in \cite{papa2015sgd,clemenccon2016scaling,schrab2022efficient}, but there also remain significant challenges.  
 On the one hand, selecting a design that minimizes variance is highly desirable, but it is a difficult combinatorial problem (\cite{lee1990u}, ch. 4.3.2). Thus,  researchers generally resort to random designs \citep{papa2015sgd,clemenccon2016scaling, chen2019randomized,leung2026berry}, which are easier to implement and analyze, but come at the cost of a suboptimal variance. On the other hand, the problem of determining the limiting distribution remains open. For instance, in the case of second-order degenerate incomplete U-statistics, no study has yet characterized, with explicit finite-sample convergence rates, how the distribution transitions between normal and weighted chi-square laws depending on the growth rate of the size of a minimum variance design. Addressing this gap is crucial for gaining a deeper understanding of the statistical properties of incomplete U-statistics and ensuring the validity of inferential procedures. Moreover, these theoretical insights would be essential for guiding practical applications, providing clear criteria for when a Gaussian approximation is sufficiently accurate in finite samples, eliminating the need for expensive resampling methods. 
 
In this paper, we present an efficient design construction and its asymptotic framework, to address the computational and statistical challenges posed by incomplete U-statistics. Our main contributions can be structured into four distinct points: 

\begin{enumerate}[label=(\alph*)]
    \item We show that the \emph{dependency graph} \citep{baldi1989normal} of an incomplete U-statistic can be represented by the line graph of the design hypergraph. Building on this characterization and leveraging Stein's method \citep{chen2004normal} and recent refinements \citep{cai2026refined}, we derive Berry-Esseen bounds for incomplete U-statistics based on deterministic designs. Our main focus is the \emph{fully degenerate} setting, in which all lower-order terms of the \emph{Hoeffding decomposition} \citep{Hoeffding1961} vanish. For fully degenerate incomplete U-statistics based on equireplicate designs (\cite{lee1990u}, ch.~4.3.2), a subclass of deterministic designs that provides exact control over the maximum degree of the design hypergraph, we establish asymptotic normality for fixed $k$ and design sizes up to $o(n^{4/3})$, under a bounded fourth-moment condition. Our framework also allows $k$ to grow with $n$. As a complementary result, for equireplicate and linear designs in the non-degenerate case, we recover the classical $O(n^{-1/2})$ Berry-Esseen rate for design sizes compatible with the combinatorial existence constraints of this class.

     \item We establish a general \emph{studentization theorem} that converts a given Berry-Esseen bound involving the true variance into a bound for the corresponding studentized statistic, provided that the relative error of the variance estimator can be controlled. We apply this result to fully degenerate incomplete U-statistics based on equireplicate designs, using a simple variance estimator constructed from disjoint kernel evaluations. This yields an explicit \emph{studentized Berry-Esseen bound} and preserves the conditions for asymptotic normality obtained for the non-studentized statistic.

    \item We develop efficient algorithms for constructing incomplete U-statistics based on equireplicate designs. The algorithms run in linear time in the design size and achieve minimum variance when $k=2$. Our construction for $k>2$ may be of independent interest to the combinatorial design and hypergraph communities.
    \item We validate our theoretical results through numerical experiments involving fully degenerate U-statistics, including both the MMD and a higher-order U-statistic associated with an $\ell_4$ norm functional (\cite{he2021asymptotically}, Example 1). In a real data application, we showcase a permutation-free variant of the MMD test that achieves substantial computational gains while matching the power and type I error control of its permutation-based counterpart, built on the same incomplete U-statistic.
\end{enumerate}

\noindent{\bf Organization.} 
Section 2 reviews background material and sets the notation for the paper. Section 3 derives Berry-Esseen bounds for incomplete U-statistics of deterministic designs, establishes asymptotic normality under full degeneracy and potentially diverging order, and develops studentized Gaussian approximations. Section 4 presents efficient algorithms for constructing equireplicate designs. Section 5 validates our theoretical results through simulations on fully degenerate U-statistics and a real-data application. Section 6 concludes with a summary and directions for future research. 
Proofs and technical details are provided in the Supplemental Materials (sections prefixed with S). The R code required to reproduce all numerical experiments in this paper is openly available at \url{https://github.com/CaesarXVII/IUstat_equireplicate_designs}.

\smallskip

\noindent{\bf Related Work.} Existing works addressing convergence rates to limiting distributions of incomplete U-statistics that are most relevant to our contributions include: \cite{chen2019randomized}, which consider randomized incomplete U-statistics in high-dimensional settings; \cite{sturma2024testing}, which extend \cite{chen2019randomized} to a \emph{mixed degenerate} setting with applications to testing a null hypothesis defined by equality and inequality constraints; \cite{kim2024dimension}, which propose a new class of incomplete U-statistics that enable valid inference regardless of how the dimension scales with the sample size; 
\cite{shao2023u,shao2025u}, that develop higher-order approximations for the sampling distributions of studentized incomplete U-statistics under both deterministic and randomized designs, covering non-degenerate and degenerate settings; and \cite{leung2026berry}, which analyzes incomplete U-statistics of random designs—generated via Bernoulli sampling—across different regimes relating the sample size $n$ and the expected design size. Additional related work is discussed in \ref{supp:sec1}.

\section{Background and Notation}\label{sec:back_not}

We introduce the necessary background and notation for the paper. Our notation and terminology for U-statistics follows \cite{lee1990u}, for combinatorial designs we follow \cite{ColbournDinitz2006} and \cite{wallis2016introduction}, and for hypergraphs we follow \cite{bretto2013hypergraph}.

Throughout this paper, $|A|$ denotes the cardinality of a finite set $A$. For a real-valued random variable $X$ and $p\ge 1$, we let $\|X\|_p = \left(\mathbb E |X|^p\right)^{1/p}$ be the $\ell_p$ norm of $X$. The symbol $C>0$ denotes a constant whose value may change from line to line.

\subsection{Incomplete U-statistics}\label{subsec:ustat}

Let $X_1, \ldots, X_n$ be i.i.d. random variables taking values in a measurable space $(T, \mathcal{F})$ with common distribution $P$. Consider the couple $(V,\mathcal{B}_k)$, where $V = \{1, \ldots, n \}$ is the index set of the random variables and $\mathcal{B}_k = \{ S \subseteq V, k \in V \mid \vert S \vert = k  \}$ the collection of all subsets of $V$ of size $k$. Then, an incomplete U-statistic of order $k$ can be written as:

\begin{equation}
U_{n,D}^{(k)} = \frac{1}{|D|} \sum_{ \{i_1, \ldots, i_k \} \in D} h( X_{i_1}, \ldots, X_{i_k}),    \label{eq:IUk}
\end{equation}
where $h: T^k \rightarrow \mathbb{R}$ is a fixed measurable kernel function that is symmetric in its arguments, i.e., $h\left(x_1, \ldots, x_k\right)=h\left(x_{i_1}, \ldots, x_{i_k}\right)$ for every permutation $i_1, \ldots, i_k$ of $\{1, \ldots, k\}$ and $D\subseteq \mathcal{B}_k$ represents the \emph{design} of the incomplete U-statistics of size $|D|$. If $D = \mathcal{B}_k$, we obtain the complete $k$th-order U-statistic with $|D| = \binom{n}{k}$. Likewise, in the combinatorial designs literature \citep{ColbournDinitz2006}, the previously defined couple $(V,\mathcal{B}_k)$ identifies the complete (or trivial) $k$-design that considers only blocks of size $k$. 

For $c\in \{0,1,2,\ldots, k\}$, let $\mathcal{J}_c = \{ S_1, S_2 \in D \mid |S_1 \cap S_2 | = c \}$ be the set of all pairs of elements in the design which have $c$ indices in common and denote $f_c = |\mathcal{J}_c|$ its cardinality. Then,  we can write the variance of an incomplete U-statistic of order $k$ as:      

\begin{equation}
\operatorname{Var} U_{n,D}^{(k)} =|D|^{-2} \sum_{c=0}^k f_c \; \sigma_c^2,\label{eq:var_ustat} 
\end{equation}
where $\sigma_c^2=\operatorname{Cov}\left(h\left(S_1\right), h\left(S_2\right)\right)$ with $S_1, S_2 \in \mathcal{J}_c$ and $h(S)$ is shorthand for $h\left(X_{i_1}, \ldots, X_{i_k}\right)$, where $S=\left\{i_1, \ldots, i_k\right\} \in D$. Moreover, by Theorem 4 in \cite{lee1990u} (page 15), we have: 
\begin{equation}
\frac{\sigma_{b}^2}{b} \leq \frac{\sigma_{c}^2}{c},   \label{eq:ineq} 
\end{equation}
for $0 < b  \leq c \leq k$ and $b \in \{0, 1, \ldots, k\}$. We assume that $\sigma_{k}^2 < \infty$. Combined with inequality~\eqref{eq:ineq} and the identity $\sum_{c=0}^k f_c = |D|^{2}$, this assumption guarantees that $\operatorname{Var} U_{n,D}^{(k)}$ is finite. Moreover, by inequality~\eqref{eq:ineq}, if $\sigma_c^2 = 0$, then $\sigma_1^2 = \cdots = \sigma_{b}^2 = 0$ and the U-statistic is called \emph{degenerate of order $c$}. We assume $\sigma_k^2 > 0$, which ensures that $U_{n,D}^{(k)}$ can be at most degenerate of order $k-1$ (\emph{full degeneracy}). In general, the order of degeneracy determines the asymptotic distribution of the complete U-statistic. This is because $\sigma_c^2 = 0$, implies that the first $c$ terms vanish in the \emph{Hoeffding decomposition} \citep{Hoeffding1961}, which is a representation of a U-statistic of order $k$ with a sum of uncorrelated U-statistics of order $1, \ldots, k$. For example, when $k=2$ and $\sigma_1^2 = 0$, we have a \emph{first order degeneracy} and \cite{gregory1977large} proved that $n ( U_{n,\mathcal{B}_2}^{(2)} - E[U_{n,\mathcal{B}_2}^{(2)}])$ converges in distribution to an infinite weighted sum of centered $\chi^2$ random variables. However, if $D \subset \mathcal{B}_k$, the limiting behavior of $U_{n,D}^{(2)}$ can vary and depends on the choice of the design \citep{weber1981incomplete}.

\subsection{Equireplicate designs}\label{subsec:equirep}

The class of \emph{equireplicate designs} is the main focus of our paper. Within this class, every index occurs in the same number of elements of the design, usually called blocks. Any $r$-equireplicate design $D$ based on $V$, where $r$ denotes the replication parameter, satisfies:
\begin{equation}
|D| \; k = n \; r.     \label{eq:equirep} 
\end{equation}

When $k = 2$, incomplete U-statistics based on equireplicate designs achieve minimum variance, as demonstrated by Theorem 1 on page 195 of \cite{lee1990u}. When $k > 2$, additional conditions are needed to ensure minimum variance, such as \emph{balanced incomplete block designs} (BIBDs) and \emph{cyclic designs} (see \cite{lee1982incomplete}, Examples 2 and 7, respectively). Further details are provided in \ref{app:min_var}. 

The existence of equireplicate designs requires specific divisibility conditions. Given a nonzero integer $a$ and an integer $b$, we write $a \mid b$ to indicate that $a$ divides $b$, and $a \nmid b$ otherwise. 
Our design constructions use modular arithmetic and some basic group theory (see \cite{gallian2021contemporary}, ch.~2 for an overview). We write $b \pmod n$ to denote the remainder upon dividing $b$ by $n$, but set $b\pmod n\defeq n$ if the remainder is zero; we also write $\text{gcd}(b,n)$ as the greatest common divisor of $b$ and $n$.
In Section~\ref{sec:constuct}, we present efficient algorithmic constructions of equireplicate designs that use these algebraic concepts. Since we relabeled 0 as $n$ when calculating values $\pmod n$, we consider $\mathbb{Z}_n=\{1,2,\ldots, n\}$ as the group of integers modulo $n$ with additive identity element $n$. This notation was chosen to ensure that the elements of our designs are labeled from the index set $V=\{1,\ldots, n\}$. 

\subsection{Hypergraphs and deterministic designs}\label{subsec:hyper}

A hypergraph $H$ is a couple $(V, E)$, where $V$ is the vertex set, and $E \subseteq \mathcal{P}(V)\setminus\{\varnothing\}$ is the hyperedge set, with $\mathcal{P}(V)$ denoting the power set of $V$. Each hyperedge $e \in E$ is a non-empty subset of $V$ and the singleton $e_{v} = \{v\}$ represents a loop for all $v \in V$. We call $H(v)$ the set of hyperedges containing the vertex $v$ and define the degree as $d(v)=|H(v)|$ for all $v \in V$. The average degree of a hypergraph $H$ is denoted by $\Bar{d}(H)$ and its maximum degree by $\Delta(H)$. If each vertex has the same degree, we say that the hypergraph $H$ is \emph{$r$-regular}, which implies that $\Bar{d}(H) = \Delta(H) = r$. 

When every hyperedge contains exactly $k$ vertices, i.e., $|e| = k$ for all $e \in E$, the hypergraph is called \emph{$k$-uniform}. We can extend the classical \emph{handshaking lemma}—that relates the number of edges in a graph with the sum of the degrees—to $k$-uniform hypergraphs. Indeed, for any $k$-uniform hypergraph $H = (V, E)$, it holds that:

\begin{equation}
|E| \; k = \sum_{v \in V} d(v)\;,     \label{eq:handshaking_kuni}
\end{equation}

\noindent because each hyperedge contributes exactly $k$ incidences, one for each vertex it contains.

To derive the Berry-Esseen bound and the asymptotic properties in Section \ref{sec:norm_approx}, we rely on a particular line graph construction. For a given hypergraph $H = (V, E)$, its line graph $L(H)$ is the couple $(E,E^{\prime})$ where $E^{\prime}
\;=\;\bigl\{\{e_i,e_j\}
\;\bigm|\; e_i, e_j \in E \;\text{and}\; |e_i \cap e_j| \;\neq\;0 
\bigr\}$. Note that our definition automatically includes loops, i.e., self-edges. In contrast, alternative definitions of the line graph explicitly exclude this possibility by imposing the additional condition $e_i \neq e_j$ for any element of $E^{\prime}$. In any case, $L(H)$ is a graph, with vertex set that coincides with the hyperedge set of $H$ and where an edge connects two vertices if and only if the corresponding hyperedges in $H$ have at least one vertex in common.

A correspondence exists between $k$-uniform hypergraphs and deterministic designs\footnote{In \ref{app:equi_hyper}, we discuss an analog of Observation \ref{obs:hyper_design} that applies specifically to equireplicate designs and $k$-uniform, $r$-regular hypergraphs.}:
\begin{observation}\label{obs:hyper_design}
The hyperedge set of any $k$-uniform hypergraph identifies a deterministic design on the vertex set $V$, and conversely, any deterministic design with building blocks of size $k$ corresponds to the hyperedge set of a $k$-uniform hypergraph.
\end{observation}

Because of this equivalence, structure in the hypergraph setting translates into properties of the corresponding designs and vice versa. For example, the couple $(V, \mathcal{B}_k)$ can be interpreted either as the complete $k$-uniform hypergraph—commonly denoted by $K_n^{(k)}$—or as the complete $k$-design that underlies the complete $k$th-order U-statistic. More generally, given a deterministic design $D\subseteq \mathcal{B}_k$, the couple $(V,D)$ defines the hypergraph of the design $D$, which we denote by $\mathcal{D} = (V,D)$. When $D\subset \mathcal{B}_k$, $\mathcal{D}$ is a sub-hypergraph of $K_{n}^{(k)}$ and its hyperedge set $D$ is the design of an incomplete U-statistic of order $k$. By equation~\eqref{eq:handshaking_kuni}, any deterministic design $D$ defined over $V$ must satisfy:

\begin{equation}
|D| \; k = n \; \Bar{d}(\mathcal{D})\;,     \label{eq:deter} 
\end{equation}

\noindent where $\Bar{d}(\mathcal{D})$ is the average degree of the hypergraph of the design $D$. We rely on this result to express the size of a deterministic design as $|D|=n\;\Bar{d}(\mathcal{D})\;k^{-1}$. Clearly, for an $r$-equireplicate design, $\Bar{d}(\mathcal{D}) = r$ and equation (\ref{eq:deter}) boils down to equation (\ref{eq:equirep}).

\section{Normal Approximations of Incomplete U-statistics}\label{sec:norm_approx}

An incomplete U-statistic can be seen as a classical mean computed in the space induced by the kernel function $h$. Therefore, equation~\eqref{eq:IUk} can be expressed more concisely as: 

\begin{equation}
  U_{n,D}^{(k)} = \frac{1}{|D|} \sum^{}_{S \in D} h(S),  \label{eq:IUk_red}
\end{equation}
where $h(S)$ represents the random variable  $h\left(X_{i_1}, \ldots, X_{i_k}\right)$ with $S=\left\{i_1, \ldots, i_k\right\} \in D$.

The random variables in the set $\left\{ h(S), \;S \in D \right\}$ are identically distributed since the data, $X_1, \ldots, X_n$, are i.i.d., and $h$ is a fixed, symmetric and measurable kernel function. We denote $E\left[h(S)\right] = \mu_{k}$ for all $S \in D$ and, following the notation in Section \ref{subsec:ustat}, we conclude that $\operatorname{Var}[h(S)] = \sigma^{2}_{k}$ for all $S \in D$. However, unlike in a standard i.i.d. mean estimation setting, the random variables in $\left\{h(S),\;S\in D\right\}$ may be dependent. The  insight that underpins the results of this section is that we can tame this dependence by controlling $\Delta(\mathcal{D})$, the maximum degree of the hypergraph of the design, where for a given design $D$, $\Delta(\mathcal{D})$ represents the index that appears most frequently in the design.

\subsection{The dependency graph of incomplete U-statistics}

The set $\left\{ h(S), \;S \in D \right\}$ contains the random variables whose average defines $U_{n,D}^{(k)}$.
To characterize the dependence structure of this set, we define its dependency graph, a concept originally introduced in \cite{baldi1989normal} to derive normal approximations:

\begin{definition}[Dependency graph] \label{def:dep_graph}
    For a set of random variables $\left\{ h(S), \;S \in D \right\}$ indexed by the vertices of a graph $\mathcal{G}=(D, E)$, $\mathcal{G}$ is said to be a dependency graph if, for any pair of disjoint sets $\Gamma_1$ and $\Gamma_2$ in $D$ such that no edge in $E$ has one endpoint in $\Gamma_1$ and the other in $\Gamma_2$, the sets of random variables $\left\{h(S), S \in \Gamma_1\right\}$ and $\left\{h(S), S \in \Gamma_2\right\}$ are independent.
\end{definition}

In a dependency graph, there is an edge between two random variables only if they are dependent (absence of an edge implies independence). In the case of U-statistics, we include an edge between two random variables $h(S_1)$ and $h(S_2)$ if and only if $S_1\cap S_2\neq\emptyset$.  

Note that this corresponds with the definition of $L(\mathcal{D})$, given in Section~\ref{subsec:hyper}:

\begin{proposition}\label{prop:dep_graph}
$L(\mathcal{D})$ is the dependency graph of $\left\{ h(S), \;S \in D \right\}$.  
\end{proposition}

The maximum degree of the dependency graph $\Delta(L(\mathcal D))$, bounds how many random variables each $h(S)$ can be dependent with. Controlling this quantity is key for establishing a normal limiting distribution, as we demonstrate in the following section.  

\subsection{Berry-Esseen bounds and asymptotic implications}\label{subsec:BE_bound}

We begin by deriving an interpretable and tight upper bound for the maximum degree of the dependency graph $\Delta(L(\mathcal{D}))$, expressed solely in terms of $k$ and $\Delta(\mathcal{D})$. This result is then combined with key findings from the literature on dependency graphs to establish a Berry–Esseen bound for incomplete U-statistics based on deterministic designs.

Intuitively, the maximum value of $\Delta(L(\mathcal{D}))$ is attained when the block $S$, which identifies the variable $h(S)$, consists of indices that each appear with the highest possible frequency in the design; i.e., each has degree $\Delta(\mathcal{D})$ in the hypergraph $\mathcal{D}$. In this case, each of the $k$ indices in $S$ can contribute up to $(\Delta(\mathcal{D}) - 1)$ edges in the line graph (excluding the self-edge), yielding a total of at most $k \;(\Delta(\mathcal{D}) - 1)$ dependencies for $h(S)$. In contrast, the minimum value of $\Delta(L(\mathcal{D}))$ is reached when only one index in the block $S$ has degree $\Delta(\mathcal{D})$, while the remaining $k - 1$ indices do not appear in any other blocks. In this setting, $h(S)$ is connected to exactly $(\Delta(\mathcal{D}) - 1)$ other variables in $L(\mathcal{D})$. We formalize the upper and lower bounds—both of which are tight—in the following lemma.

\begin{lemma}\label{lemma:maxdeg}
    Let $\mathcal{D} = (V, D)$ be the hypergraph of a deterministic design $D$ and $L(\mathcal{D})$ its line graph. Then, $\Delta(\mathcal{D}) \leq \Delta(L(\mathcal{D})) \leq k\;(\Delta(\mathcal{D})-1) + 1.$
\end{lemma} 

With this foundation, we can leverage powerful existing results from the literature on dependency graphs. In particular, we rely on \cite{cai2026refined} that couples \emph{Stein’s method} with a concentration inequality to derive normal approximations under local dependence. Combining their findings with Lemma \ref{lemma:maxdeg}, we obtain a Berry-Esseen bound valid for all incomplete U-statistics based on deterministic designs.

\begin{theorem}[Berry-Esseen for Deterministic Designs]\label{theo:BE_bound}

Let $\left\{ h(S), \;S \in D \right\}$ be random variables indexed by the vertices of their dependency graph $L(\mathcal{D})$, with $D$ being a non-empty deterministic design. Assume that $\sigma^{2}_{k} > 0$ and that $\|h(S)-\mu_k\|_4 < \infty$. Then,

\begin{equation}\label{eq:berry_esseen}
\sup _z\left|P\left(\frac{U_{n,D}^{(k)} - \mu_{k}}{\sqrt{\operatorname{Var}U_{n,D}^{(k)}}} \leq z\right)-\Phi(z)\right| \leq C\,
\frac{\{k(\Delta(\mathcal D)-1)+1\}^{2}}{|D|^{1/2}}
\left(
    \frac{\|h(S)-\mu_k\|_4}{\sigma_k}
\right)^{3}. 
\end{equation}

\end{theorem}
\normalsize

 As either $k$ or $\Delta(\mathcal{D})$ increase, the bound deteriorates. A similar effect occurs when the gap between the average degree $\bar{d}(\mathcal{D})$ and the maximum degree $\Delta(\mathcal{D})$ widens, since $|D|= n\; \bar{d}(\mathcal{D})\;k^{-1}$ by \eqref{eq:deter}. This corresponds to deterministic designs whose associated hypergraph exhibits an increasingly skewed degree distribution. On the contrary, maximum and average degrees coincide in an equireplicate design.
 The next corollary, gives a specialized result in the case of incomplete U-statistics based on equireplicate designs:

\begin{corollary}\label{cor:BE_equi}\textnormal{(Berry-Esseen for Equireplicate Designs).}
Consider $D^{\dagger}$ to be a non-empty $r$-equireplicate design and assume that the conditions of Theorem \ref{theo:BE_bound} are met. Then,

\small{
\begin{equation}\label{eq:be_equi}
\sup _z\left|P\left(\frac{U_{n,D^{\dagger}}^{(k)} - \mu_{k}}{\sqrt{\operatorname{Var}U_{n,D^{\dagger}}^{(k)}}} \leq z\right)-\Phi(z)\right| \leq C \;
\{k(r-1)+1\}^{2} \left( \frac{k}{n\;r} \right)^{1/2}
\left(
    \frac{\|h(S)-\mu_k\|_4}{\sigma_k}
\right)^{3}\;\;.  
\end{equation}}
\end{corollary}

The result follows by substituting $\Delta(\mathcal{D}) = r$ and $|D^{\dagger}| = nrk^{-1}$ in \eqref{eq:berry_esseen}, since also $\bar{d}(\mathcal{D})=r$. Moreover, the bound in \eqref{eq:be_equi} is tighter for incomplete U-statistics based on $r$-equireplicate designs than the bound in \eqref{eq:berry_esseen} for U-statistics based on non-equireplicate deterministic designs of the same size (see Section \ref{subsec:equi_vs_nonequi} for a detailed explanation). 

\begin{remark}[Additional Berry-Esseen Bounds]\label{remark:be_add}
We can relax the moment assumptions of Theorem \ref{theo:BE_bound} to a bounded $\|h(S)-\mu_k\|_p$ with $2< p \leq 3$, and still obtain a Berry-Esseen bound by \cite{chen2004normal}. The advantages of this new bound—presented in Theorem \ref{theo:BE_bound_2004}—are not only the weaker moment conditions, but also that the constant value is specified ($C=75$). However, this comes at the cost of an increase in the exponent of $\Delta(L(\mathcal{D}))$, which is $2$ in Theorem \ref{theo:BE_bound} and becomes $5(p-1)$ in Theorem \ref{theo:BE_bound_2004}. On the other hand for bounded kernels—by applying results of \citet{baldi1989normal}, also used by \citet{cai2024asymptotic} in high-dimensional independence testing—we get both an exponent of one for $\Delta(L(\mathcal{D}))$ and an explicit constant value ($C=32(1+ \sqrt6))$. However, in this Berry-Esseen bound—presented in Theorem \ref{theo:BE_bounded_kern}—the design size appears as $|D|^{1/4}$.
\end{remark}

Our Berry-Esseen bounds are primarily relevant under full degeneracy. This is because we rely on a lower bound for $\operatorname{Var}U_{n,D}^{(k)}$ which is tight\footnote{A comprehensive discussion on the variance of incomplete U-statistics of deterministic designs—including tight upper and lower bounds that involve $\Bar{d}(\mathcal{D})$ and $\Delta(\mathcal{D})$—can be found in~\ref{app:var_Ustat}.} when $\sigma^2_{k-1} = 0$. However, if there exists a $c \in \{1, \ldots, k-1\}$ such that $\sigma_c^2 > 0$, the bounds can be made sharper for specific classes of deterministic designs, as shown in Remark \ref{remark:be_equi_lin} for equireplicate and linear designs.

\begin{remark}[Berry-Esseen bound for equireplicate and linear designs]\label{remark:be_equi_lin}
    By assuming that $\sigma_1^2 > 0$ (non-degenerate case) and $\|h(S)-\mu_k\|_4 < \infty$, we derive a Berry-Esseen bound valid for the class of equireplicate and linear design. The term ``linear'' refers to the associated hypergraph $\mathcal{D}^{\diamond} = (V, D^{\diamond})$, in which for all $S_1, S_2 \in D^{\diamond}$ with $S_1 \neq S_2$, we have that $\left|S_1 \cap S_2\right| \leq 1$. That is, any pair of random variables in the set $\left\{ h(S), \;S \in D^{\diamond} \right\}$ shares at most one index. Compared to Corollary \ref{cor:BE_equi}, this new Berry-Esseen bound—presented in Theorem~\ref{theo:BE_equi_lin_nondeg}—improves the exponent of $\{k(r-1)+1\}$ from $2$ to $1/2$, while keeping the design size exponent unchanged (see Remark \ref{remark:rate_equi_lin} for a discussion on the rate).
\end{remark}

Theorem \ref{theo:BE_bound} and Corollary \ref{cor:BE_equi}, as well as Theorems \ref{theo:BE_bound_2004}, \ref{theo:BE_bounded_kern}, and \ref{theo:BE_equi_lin_nondeg}, identify converge rates and thus conditions under which the Gaussian approximation becomes asymptotically exact. We next examine the implications of these bounds in the \emph{finite-order regime}, where $k$ is fixed, and in the \emph{infinite-order regime}, where $k$ is allowed to grow with $n$.

\vspace{0.25cm}
 
\noindent \textbf{Finite-order regime.} When $k$ is fixed, we refer to $D_n$ as a sequence of deterministic designs of growing size that identifies the sequence of hypergraphs $\mathcal{D}_n = (V,D_n)$. Thus, the bound of Theorem \ref{theo:BE_bound} is $O(\Delta(\mathcal{D}_n)^2\;|D_n|^{-1/2} )$ and, since $|D_n|= O(n\; \bar{d}(\mathcal{D}_n))$ by \eqref{eq:deter}, converges to zero as $n \rightarrow \infty$, when $\Delta(\mathcal{D}_n)^2 \;n^{-1/2}\; \bar{d}(\mathcal{D}_n)^{-1/2} \rightarrow 0$. If we assume\footnote{This assumption implies that every observation $X_i$, with $i\in V$, appears in at least one block of the design, thereby avoiding unnecessary loss of information.} that $\min\limits_{v\in V} d(v) \geq 1$, then a sufficient condition to obtain normality under full degeneracy for any deterministic design is $\Delta(\mathcal{D}_n) = o(n^{1/4})$, since $\bar{d}(\mathcal{D}_n) \geq 1$. Both conditions are harder to satisfy when $\mathcal{D}$ has an unbalanced degree distribution, as shown in the following example. 

\begin{example}[Designs with unbalanced degree distribution]
     Consider a $k$-uniform star design (see \emph{Definition 1.10} in \cite{keevash2014spectral}, with $t=1$). When $k=2$, the hypergraph $\mathcal{D}_n^{\star}$ is a star graph with $n$ vertices. This design has $\Delta(\mathcal{D}_n^{\star}) = n-1$ and $\bar{d}(\mathcal{D}_n^{\star}) = 2 - 2/n$. When $n$ diverges, an incomplete U-statistic based on the star design violates both $\Delta(\mathcal{D}_n^{\star})^2 \;n^{-1/2}\; \bar{d}(\mathcal{D}_n^{\star})^{-1/2} \rightarrow 0$ and $\Delta(\mathcal{D}^{\star}_n) = o(n^{1/4})$.  
\end{example}

On the other hand, for an equireplicate design, the parameter $r$ determines the order of growth of the entire degree distribution of its associated hypergraph. The bound of Corollary \ref{cor:BE_equi} is $O(r^{3/2}\;n^{-1/2})$. Thus, a sufficient condition to obtain normality under
full degeneracy for any equireplicate design is $r=o(n^{1/3})$, which implies $|D^{\dagger}_n| = o(n^{4/3})$ by \eqref{eq:equirep}. To our knowledge, these are the first converge rate results on incomplete U-statistics under full degeneracy that cover the whole class of deterministic designs. \cite{shao2023u}—Theorem 3.1—state an asymptotic normality result for a specific deterministic construction under degeneracy, but an explicit convergence rate is not provided. However, their construction allows $|D_n| \asymp n^{\alpha}$ for $\alpha \in (1,2)$. Our growth condition for equireplicate designs is therefore more conservative, but provides an explicit finite-sample guarantee for this class of designs. In our simulation study on the MMD statistic (see Figure \ref{fig:MMD_simstudy}), we have evidence of a deviation from normality already at $r=n^{0.75}$—so $|D^{\dagger}_n| = O(n^{7/4})$. Moreover, since \cite{shao2023u} consider studentized incomplete $U$-statistics, we address this distinction in Section~\ref{subsec:stud}, comparing our studentized version of Corollary~\ref{cor:BE_equi} with their convergence rates for random designs.

\begin{remark}[Canonical rate for equireplicate and linear designs]\label{remark:rate_equi_lin} 
In the non-degenerate case i.e., when $\sigma_1^2 >0$, the bound of Theorem \ref{theo:BE_equi_lin_nondeg} is $O(n^{-1/2})$. Thus, for any admissible choice of $|D^{\diamond}_n|$ or $r$, our rate obtained for equireplicate and linear designs matches the canonical Berry-Esseen rate for complete U-statistics in the finite-order non-degenerate case (see \cite{lee1990u}, ch. 3.3 and references therein). Since the design is linear, we know that $|D^{\diamond}_n| \leq \binom{n}{2}/\binom{k}{2}$ (\cite{ColbournDinitz2006}, ch. 40), and thus the admissible design size is $O(n^2)$, which implies $r=O(n)$. Note that the $O(n^{-1/2})$ rate is also valid for any equireplicate design when $k=2$ because they are also linear, since no pair of distinct blocks $S_1$ and $S_2$ in $\mathcal{B}_2$ can share more than one element. In the non-degenerate case—for a large class of designs of size $|D_n| \asymp n^{\alpha}$ with $\alpha \in (1,2)$—\cite{shao2025u} obtain an approximation error $O(n^{-\alpha/2}\log^{1/2}n)$ by a higher-order Edgeworth expansion for studentized incomplete U-statistics. On the other hand, our Berry–Esseen result directly controls the Kolmogorov distance to the Gaussian distribution and can be studentized following the steps outlined in Section \ref{subsec:stud} for the fully degenerate case.
\end{remark}

\noindent \textbf{Infinite-order (IO) regime.} In this setting, $k$ is a sequence of natural numbers indexed by $n$, while $D^{(k)}_n$ is a sequence of deterministic designs of growing cardinality and block size that identifies the sequence of hypergraphs $\mathcal{D}^{(k)}_n = (V,D^{(k)}_n)$. Throughout, we assume that there exists a constant $M<\infty$ such that $\|h(S)-\mu_k\|_4 \;\sigma_k^{-1} \leq M$ uniformly over $k$. Then, the bound of Theorem \ref{theo:BE_bound} is $O(k^{5/2}\;\Delta(\mathcal{D}^{(k)}_n)^2\;n^{-1/2}\; \bar{d}(\mathcal{D}^{(k)}_n)^{-1/2})$, since $|D^{(k)}_n|= n\; \bar{d}(\mathcal{D}^{(k)}_n)\;k^{-1}$ by \eqref{eq:deter}. Thus, if $\max\{k,\;\Delta(\mathcal{D}^{(k)}_n)\}=O(\log^q(n))$ with $q >0$ and $\min\limits_{v\in V} d(v) \geq 1$, we recover for incomplete U-statistics of deterministic designs the canonical Berry-Esseen $O(n^{-1/2})$ convergence rate to the standard normal distribution—up to logarithmic factors—even under full degeneracy. Moreover, since $\Bar{d}(\mathcal{D}^{(k)}_n) \leq \Delta(\mathcal{D}^{(k)}_n)$ by definition of the maximum degree, we retain linear computational complexity in $n$ for any deterministic design $D^{(k)}_n$, again up to logarithmic terms. On the other hand for equireplicate designs, the bound of Corollary \ref{cor:BE_equi} is $O(k^{5/2}\;r^{3/2}\;n^{-1/2})$. Therefore, a sufficient condition to obtain normality under full degeneracy is $k^5r^3 = o(n)$, which implies $|D^{\dagger\;(k)}_n| = o(n^{4/3}\;k^{-8/3})$.

Existing works in the IO regime for incomplete U-statistics either assume a non degenerate first-order Hoeffding component \citep{SongChenKato2019}, or consider degeneracy under stronger tail conditions: \citet{sturma2024testing}, for example, allow diverging kernel order and mixed degeneracy, but assume sub-Weibull kernels and additional moment conditions. Moreover, both works study Bernoulli-randomized designs. To the best of our knowledge, our results are the first to provide Gaussian approximation rates in the IO regime for deterministic designs, including equireplicate designs, while allowing full degeneracy under weaker moment assumptions. These results illustrate the flexibility of our approach. Under full degeneracy, only the highest-order Hoeffding component remains, making the analysis increasingly delicate as $k$ grows \citep{SongChenKato2019}. In contrast, our dependency graph approach captures all dependence induced by overlapping blocks through the same line graph construction, regardless of its order. Thus, although degeneracy affects the variance, the same framework applies to both finite-order and IO regimes. In Section~\ref{subsect:ious}, we provide a more detailed analysis of the IO regime, including a discussion of variance estimation under full degeneracy with diverging $k$ in Remark~\ref{remark:var_inford}.

\subsection{Studentization and inference under full degeneracy }\label{subsec:stud}

In practice, the variance of an incomplete U-statistic is unknown and has to be estimated. To verify whether this additional step degrades the rates established in Theorem \ref{theo:BE_bound} and Corollary \ref{cor:BE_equi}, we derive a studentized Berry-Esseen bound that is valid for fixed $k$ under full degeneracy. This setting is relevant for hypothesis testing problems, including kernel-based tests, that we tackle in our simulation studies of Section~\ref{sec:simul}. We start by introducing a general way to obtain studentized Berry-Esseen bounds, which is independent of U-statistics, and may be useful beyond the present context.

\begin{theorem}[General studentized Berry-Esseen bound]
\label{theo:gen_stud}
Let $T$ be a real-valued statistic, let $\theta\in\mathbb R$, and let
$\tau>0$ be a deterministic variance parameter with an almost surely
positive estimator $\widehat{\tau}$. Suppose that, for some $B_T\geq 0$,
$$
\sup_{z\in\mathbb R}
\left|
P\left(
\frac{T-\theta}{\sqrt{\tau}}
\leq z
\right)-\Phi(z)
\right|
\leq B_T.
$$
Then, for every $0<\eta<1/2$,
$$
\sup_{z\in\mathbb R}
\left|
P\left(
\frac{T-\theta}{\sqrt{\widehat{\tau}}}
\leq z
\right)-\Phi(z)
\right|
\leq
B_T
+
\frac{2}{\sqrt{2\pi e}}\;\eta
+
P\left(
\left|
\frac{\widehat{\tau}}{\tau}-1
\right|>\eta
\right)\;.
$$
\end{theorem}

The bound clearly separates the Gaussian approximation error under normalization by the true variance—i.e., $B_T$—from the additional error introduced by replacing the true variance with its estimator. Thus, after having established a Berry-Esseen result such as Corollary \ref{cor:BE_equi}, one only needs a tail bound for the variance estimator of choice to obtain a studentized result.
Since we focus on the case of full degeneracy, only $\sigma^2_k$ is of interest.

\begin{remark}[Adaptive variance estimation under unknown degeneracy]\label{remark:adapt_var}
In the setting considered by this work, the degeneracy structure is known \emph{a priori}. This is because $\sigma_1^2=\cdots=\sigma_{k-1}^2=0$, therefore only $\sigma_k^2$ needs to be estimated. If the order of degeneracy is unknown, variance estimation becomes more delicate, as estimation errors associated with truly zero lower-order Hoeffding components must remain negligible relative to the variance contribution of the first non-degenerate component. For instance, \cite{shao2026higher} develop such a unified variance estimator—that adapts automatically to an unknown degeneracy order—for network two-sample tests. Incorporating a similar adaptive procedure would allow our methodology to accommodate unknown or \emph{mixed degenerate} settings \citep{sturma2024testing}, but we leave this extension for future work.
\end{remark}

Proposition~\ref{prop:sig_k_con} shows that $\sigma^2_k$ can be consistently estimated with the  sample variance $s^2_k$ calculated over the set $\left\{ h(S), \;S \in D^{\perp}_n \right\}$, which contains i.i.d. random variables that do not have indices in common. More specifically, if $k \mid n$, $D^{\perp}_n$ is an equireplicate design with $r=1$ and thus of size $n/k$ by equation~\eqref{eq:equirep}\footnote{If $k \nmid n$, $D^{\perp}_n$ would be of size $\lfloor n\rfloor/k$ where $\lfloor n\rfloor$ is the largest integer, smaller than $n$, such that $k \mid \lfloor n\rfloor$. To avoid unnecessary complications, we restrict ourselves to the case $k \mid n$ in the main text.}. For instance, fixing $n=9$ and $k=3$, we can consider the set $\left\{ h\left(X_1, X_2, X_3\right),h\left(X_4, X_5, X_6\right), h\left(X_7, X_8, X_9\right)  \right\}$ to calculate $s^2_k$. The choice of this estimator for $\sigma^2_k$ makes the derivation of the tail bound in Theorem \ref{theo:gen_stud} straightforward, by Chebyshev’s inequality. Thus, we are able to obtain a studentized version of Corollary \ref{cor:BE_equi} under full degeneracy.

\begin{corollary}\label{cor:BE_equi_stud}
\textnormal{(Studentized Berry-Esseen for Equireplicate Designs Full Degeneracy).}
Consider $D^{\dagger}$ to be a non-empty $r$-equireplicate design and assume that the conditions of Corollary~\ref{cor:BE_equi} are met. Suppose further that $\sigma_{k-1}^2 = 0$,
and let $D^{\perp}$ be a $1$-equireplicate design such that 
$|D^{\perp}|=\frac{n}{k}\geq 3$. Let $s_k^2$ be the unbiased sample variance calculated from
$\{h(S),\,S\in D^{\perp}\}$, assume that $s_k^2>0$ a.s., and define
\[
\widehat{\operatorname{Var}}U_{n,D^{\dagger}}^{(k)}
=
\frac{s_k^2}{|D^{\dagger}|}
=
\frac{k s_k^2}{nr}.
\]
Then,
\begin{align}
\sup_z
\left|
P\left(
\frac{U_{n,D^{\dagger}}^{(k)}-\mu_k}
{\sqrt{\widehat{\operatorname{Var}}U_{n,D^{\dagger}}^{(k)}}}
\leq z
\right)
-\Phi(z)
\right|
&\leq
C\{k(r-1)+1\}^{2}
\left(\frac{k}{nr}\right)^{1/2}
\left(
\frac{\|h(S)-\mu_k\|_4}{\sigma_k}
\right)^3
\nonumber\\
&\quad+
2\left(\frac{k}{n}\right)^{1/3}
\left(
\frac{\|h(S)-\mu_k\|_4}{\sigma_k}
\right)^{4/3},
\label{eq:be_equi_stud}
\end{align}
where $C>0$ is the universal constant appearing in
Corollary~\ref{cor:BE_equi}.
\end{corollary}

The bound of Corollary \ref{cor:BE_equi_stud} is $O(r^{3/2} \; n^{-1/2} + n^{-1/3})$, since $k$ is fixed. Therefore, the variance estimation error dominates whenever $r=o(n^{1/9})$, which implies $|D^{\dagger}_n| = o(n^{10/9})$. The sufficient condition to obtain normality under full degeneracy for the class of equireplicate designs remains $r=o(n^{1/3})$, or $|D^{\dagger}_n| = o(n^{4/3})$, and is thus unaffected by studentization. For generic deterministic designs, we can easily obtain a studentized version of Theorem \ref{theo:BE_bound} under full degeneracy, by adding the variance estimation error of eq. \eqref{eq:be_equi_stud}—which is $O(n^{-1/3})$—to the original bound.

In the literature of incomplete U-statistics, the majority of the works study random designs, which—for fixed $k$—can allow both larger computational budgets under degeneracy and sharper distributional approximation rates (see e.g., \cite{chen2019randomized,shao2023u}). In particular, under full degeneracy, \cite{shao2023u} obtain asymptotic normality for studentized incomplete U-statistics of random designs of size $|D_n| \asymp n^\alpha$ for $\alpha \in (1,k)$. However, from a computational perspective, the regime $\alpha \in (1,2)$ remains arguably the most relevant even for higher-order U-statistics, as it keeps the number of kernel evaluations subquadratic in $n$. Then, the difference in convergence rates reflects the distinct assumptions, approximation targets, and analytical techniques underlying the two approaches. \cite{shao2023u} exploit the additional randomness of the design through Edgeworth expansions, obtaining under full degeneracy an approximation error $\widetilde O_p\left(n^{-\alpha}\log^{1/2}n+n^{(\alpha - k)/2}\log^{k/2}n\right)$, which improves on \cite{chen2019randomized} whenever $\alpha <k$. In contrast, our dependency graph approach applies to deterministic designs through their combinatorial dependence structure, at the cost of more conservative rates. This greater robustness also allowed us to derive in Section \ref{subsec:BE_bound} explicit Gaussian approximation rates under full degeneracy, even when $k$ grows with $n$.

\begin{remark}[Inference under full degeneracy]
    By letting $n\to\infty$ in \eqref{eq:be_equi_stud} while keeping $k$ fixed, we obtain a studentized central limit theorem for incomplete U-statistics based on equireplicate designs under full degeneracy (see Theorem \ref{theo:CLT_equi_deg}). The resulting Gaussian approximation is directly usable for inference and is particularly relevant for hypothesis-testing problems whose null hypothesis induces fully degenerate U-statistics, such as the two-sample test based on the MMD and the finite-order $\ell_p$-type U-statistics considered by \cite{he2021asymptotically}. Both instances are examined in the simulation studies of Section~\ref{sec:simul}.
\end{remark}

\section{Efficient Construction of Equireplicate Designs} \label{sec:constuct}

In this section, we highlight another major advantage of equireplicate designs: they can be constructed in linear time with respect to the design size.

For second-order incomplete U-statistics, equireplicate designs of any given size can be efficiently constructed and have minimum variance. Therefore, equireplicate designs should be preferred when $k=2$, the most common and widely used case in practice.

When $k > 2$, constructing equireplicate designs is challenging, as it relates to open problems in discrete mathematics. In Section \ref{subsec:equi_k>2} we present a construction which has a linear computational complexity in the design size—when $k$ is fixed—and with $|D| = O(n^2)$. This construction, based on cyclic permutations (see \cite{lee1982incomplete}, Example 7), may be of independent interest for the combinatorial design and hypergraph community.

\subsection{Construction of r-equireplicate designs when k = 2}\label{subsec:equi_k2}

For second-order incomplete U-statistics, an $r$-equireplicate design is a subset of $\mathcal{B}_2$ of size $|D|=nr/2$, such that each index appears in exactly $r$ pairs. Our approach relies on a partition of $\mathcal{B}_2$ into disjoint $1$-equireplicate designs when $n$ is even ($2 \mid n$), and into disjoint $2$-equireplicate designs when $n$ is odd ($2 \nmid n$ and thus $2 \mid r$ by equation~\eqref{eq:equirep} if the design exists). In both cases the whole partition can be generated sequentially, allowing us to obtain an $r$-equireplicate design by taking the union of a prescribed number of subsets from it. Algorithms~\ref{alg:evenDesign} and~\ref{alg:oddDesign} implement this procedure for even and odd $n$, respectively. Theorem~\ref{theo:algo_clt_equi} establishes that these algorithms have linear computational complexity in the design size and that the resulting incomplete U-statistics achieve minimum variance. Additional details are provided in~\ref{supp:sec4_k2}. That section also discusses the connections between our algorithms and factorizations of the complete graph $K^{(2)}_n$.

\begin{algorithm}[t]
    \caption{$r$-Equireplicate Design for $n$ even and arbitrary $r$}\label{alg:evenDesign}
    \begin{algorithmic}
        \Require  $n$ even and $r\in \{1,2,\ldots, n-1\}$.
        \State Set $D=\emptyset$
        \For{$g \in \{1,2,\ldots, r\}$}
        
        Set $D=D\cup\left\{\big(g,n\big)\right\}$
\For{$i \in \{1,2,\ldots, n/2-1\}$}
\State Set $D=D\cup\left\{\big(g+i\pmod{n-1},g-i\pmod{n-1}\big)\right\}$
\EndFor
        \EndFor
        \Ensure $D$
    \end{algorithmic}
\end{algorithm}

\begin{algorithm}[t]
    \caption{$r$-Equireplicate Design for $n$ odd and $r$ even}\label{alg:oddDesign}
    \begin{algorithmic}
        \Require $n$ odd and $r\in \{2,4,6,\ldots, n-1\}$.
        \State Set $D=\emptyset$
        \For{$g \in \{1,2,\ldots, r/2\}$}
\For{$i \in \{1,2,\ldots, n\}$}
\State Set $D=D\cup\left\{\big(i,i+g\pmod{n}\big)\right\}$
\EndFor
        \EndFor
        \Ensure $D$
    \end{algorithmic}
\end{algorithm}

\begin{theorem}[Equireplicate Designs with Minimum Variance when $k=2$]\label{theo:algo_clt_equi}
Let $n$ be an integer and consider the designs produced by Algorithms~\ref{alg:evenDesign} and~\ref{alg:oddDesign}.  
If $n$ is even and $r \in \{1,2,\ldots,n-1\}$, the output of Algorithm~\ref{alg:evenDesign} is an $r$-equireplicate design, $D$.  
If $n$ is odd and $r \in \{2,4,6,\ldots,n-1\}$, the output of Algorithm~\ref{alg:oddDesign} is an $r$-equireplicate design, $D$.  
In both cases, the algorithm runs in $O(nr)=O(|D|)$ time. Moreover, the variance of the corresponding incomplete U-statistic satisfies $
\operatorname{Var} U_{n,D}^{(2)} = |D|^{-1} \{ 2(r-1)\sigma_1^2 + \sigma_2^2 \}$,
which is minimal among all incomplete U-statistics with the same design size $|D|$.
\end{theorem}

 Thus, given the findings of Section~\ref{subsec:stud}, choosing $r=o(n^{1/3})$ in Algorithm~\ref{alg:evenDesign} or~\ref{alg:oddDesign} ensures asymptotic normality under full degeneracy ($\sigma_1^2=0$) of the resulting incomplete U-statistic. Moreover, when $k=2$, every equireplicate design is also linear. Hence, in the non-degenerate case ($\sigma_1^2>0$), Theorem \ref{theo:BE_equi_lin_nondeg} guarantees asymptotic normality over the full feasible range of $r$ allowed by these algorithms (see Remark~\ref{remark:rate_equi_lin} for further details).
 
\begin{remark}\label{remark:compcost_k2}
The accompanying code of this paper implements Algorithms~\ref{alg:evenDesign} and~\ref{alg:oddDesign} in a fully vectorized manner in \textsf{R}, resulting in highly efficient execution. For example, with $n = 10^6$ and $r = 10^2$, Algorithm~\ref{alg:evenDesign} completes in approximately 62 seconds, while Algorithm~\ref{alg:oddDesign} requires about 72 seconds for $n = 10^6 + 1$ with the same $r = 10^2$. The experiment was conducted on a laptop equipped with an Intel i7-6700HQ CPU @ 2.60\,GHz and 16\,GB of RAM. Additional speedups could be achieved through parallelization over either of the for-loops, although this has not yet been implemented.
\end{remark}

\subsection{Construction of r-equireplicate designs when k > 2}\label{subsec:equi_k>2}

For incomplete U-statistics of order $k$, an $r$-equireplicate design is a subset of $\mathcal{B}_k$ of size $|D|=nr/k$, such that each index appears in exactly $r$ blocks. Our approach relies on a partial partition of $\mathcal{B}_k$ into disjoint $k$-equireplicate designs (hence $k \mid r$). This partition can be generated sequentially, allowing us to obtain an $r$-equireplicate design by taking the union of a prescribed number of subsets from it. Algorithm~\ref{alg:k>2Design} implements this procedure and Theorem~\ref{theo:algo_clt_equi_k>2} shows that, for any strictly increasing sequence of natural numbers $\eta(\cdot)$, if $n > 3\;\eta(k-1)\;\{\eta(k-1) - \eta(0)\}$ and $r\in \{k,2k,\ldots,\phi(n)k \}$, then Algorithm~\ref{alg:k>2Design} constructs $r$-equireplicate designs with linear computational complexity in the design size. In both the algorithm and theorem, we denote by $\mathcal{C}_n = \{ a \in \mathbb{Z}_n| \; \text{gcd}(a,n) = 1 \}$ the set of coprimes of $n$ and with $\phi(n) = \mid \mathcal{C}_n \mid $ its cardinality, which is known as \emph{Euler's totient function}. Additional details
are provided in \ref{supp:sec4_k>2}, which also explores connections between our algorithm and factorizations of the complete $k$-uniform hypergraph $K_n^{(k)}$.

\begin{algorithm}
    \caption{$r$-Equireplicate Design for $k>2$ and $r$ multiple of $k$}\label{alg:k>2Design}
    \begin{algorithmic}
        \Require $k > 2$, $\eta(\cdot)$, $n > 3\;\eta(k-1)\;\{\eta(k-1) - \eta(0)\}$ and $r\in \{k,2k,\ldots,\phi(n)k \}$.
        \State Set $D=\emptyset$, $b = \emptyset$ and $ \mathcal{C}_{n,r} = \{a \in \{1,\ldots, r/k \}| \; \text{gcd}(a,n) = 1 \}$
        \For{$g \in \mathcal{C}_{n,r}$}
\For{$i \in \{0, 1,\ldots, n-1\}$}
\For{$j \in \{0, 1,\ldots, k-1\}$}

\State Set $b = b \cup \left\{i+g\bigl[\eta(j)- \eta(0)\bigr]\pmod{n}\right\}$

\EndFor
\State Set $D=D \cup b$ and $b = \emptyset$
\EndFor
        \EndFor
        \Ensure $D$
    \end{algorithmic}
\end{algorithm}

\begin{theorem}
    [Equireplicate Designs when $k>2$]\label{theo:algo_clt_equi_k>2}
  Let $k > 2$, $\eta: \{0, \ldots, k-1\} \to \mathbb{N}_0$ be any strictly increasing natural number valued sequence, $n$ be a positive integer such that $n > 3\;\eta(k-1)\;\{\eta(k-1) - \eta(0)\}$ and $r\in \{k,2k,\ldots,\phi(n)k \}$. Then, the output of Algorithm \ref{alg:k>2Design} is an $r$-equireplicate design $D$ and the runtime of Algorithm \ref{alg:k>2Design} is $O(nr)=O(|D|)$.
\end{theorem}  

\begin{remark}[Order of $|D|$ and choice of $\eta(\cdot)$]\label{remark:euler} 

When $\eta(j) = 2^j$ in Algorithm~\ref{alg:k>2Design}, our construction is related to the one proposed in \citet{shao2025u}. However, the design size in finite-samples is not explicitly derived in that work. In contrast, Theorem~\ref{theo:algo_clt_equi_k>2} considers $r\in \{k,2k,\ldots,\phi(n)k \}$, which implies $|D| \in \{n,2n,\ldots,\phi(n)n\}$ by equation~\eqref{eq:equirep}. Following \citet{hardy2008introduction}, $\phi(n)$ is asymptotically of order $n$ and, when $n$ is prime, $\phi(n)$ attains its maximum value of $(n-1)$.
Thus, Algorithm~\ref{alg:k>2Design} constructs an $r$-equireplicate design of size $|D| = O(n^2)$, and fixing an admissible $r=o(n^{1/3})$ ensures asymptotic normality under full degeneracy of the resulting incomplete U-statistic (see Section~\ref{subsec:stud}) regardless of the choice of $\eta(\cdot)$. However, choosing the right $\eta(\cdot)$ remains important, as it determines the degree of overlap between the blocks forming the design and thereby influences the variance of the resulting incomplete U-statistic. A theoretical analysis of this choice, quantifying its impact on the variance, is left for future work.
\end{remark}

\section{Numerical Experiments}\label{sec:simul}

In this section, we empirically validate the theoretical results established in Section \ref{subsec:stud} regarding the asymptotic distribution of incomplete U-statistics of equireplicate designs in the finite-order regime under full degeneracy. We also investigate the variance-minimization property (for $k=2$) and linear computational complexity of our novel equireplicate design construction algorithms introduced in Section~\ref{sec:constuct}.

To this end, we conduct two sets of simulation studies: one on a two-sample test based on the unbiased MMD (uMMD) statistic, which is a second-order U-statistic (i.e., covering the $k=2$ case), and the other on a one-sample mean test based on an unbiased estimator of the fourth power of the $\ell_4$ norm \citep{he2021asymptotically}, which is a fourth-order U-statistic (i.e., covering the $k>2$ case). In addition, we illustrate our methodology on the widely used \emph{CIFAR-10} image classification dataset \citep{Krizhevsky09learningmultiple}, which has frequently served as a benchmark for evaluating alternatives to the standard MMD two-sample test (see e.g., \cite{liu2020learning}). In this context, our approach provides a permutation-free version of the MMD test, that offers substantial computational advantages.

\subsection{Experiments on MMD}\label{subsec:mmd}

In the first study, we perform a series of two-sample tests using incomplete versions of the uMMD test statistic—constructed via equireplicate designs—to assess departures from normality in their asymptotic distributions. Under the null $H_0$—i.e., when the two samples are drawn from the same distribution—the complete uMMD statistic (see eq.~(6) in \cite{gretton2006kernel}) is a degenerate U-statistic and has a non-Gaussian limiting distribution. Under the alternative $H_1$, it is non-degenerate and converges to a normal (see Theorem~8 in \cite{gretton2006kernel} for details). 

We simulate both samples $(X,Y)$ i.i.d. from a $\mathcal{N}(0,1)$. We vary the common sample size $n \in \{100, 200, 400,800,1600\}$. Each incomplete uMMD statistic is computed using a linear kernel and an $r$-equireplicate design generated by Algorithm~\ref{alg:evenDesign}. We select $r \in \{1,n^{0.33},n^{0.5},n^{0.75}, n/2, n-1\}$. We obtain the empirical distribution of the standardized uMMD statistic by repeating the experiment $500$ times, standardizing the statistic in each replicate as shown in Theorem \ref{theo:CLT_equi_deg}, with $k=2$. To quantify departures from normality, we compute the Kolmogorov-Smirnov (KS) distance between this empirical distribution and $\mathcal{N}(0,1)$. We then repeat the entire procedure $100$ times to produce a sampling distribution of KS distances. The left panel of Figure~\ref{fig:MMD_simstudy} reports the mean KS distance together with $95\%$ Monte Carlo confidence intervals (CI) and, for reference, shows the mean KS distance under exact $\mathcal{N}(0,1)$ sampling.  

We observe that, for $r \in \{1,n^{0.33},n^{0.5}\}$, the mean KS distance quickly stabilizes around the normal-reference mean KS. This is not the case for $r=n^{0.75}$: although the KS distance decreases as $n$ grows, its finite-sample values still provide evidence of a departure from normality. For larger values of $r$, such as $r=n/2$, the distance from normality no longer decreases as $n$ increases\footnote{As expected, when $r=n-1$, the KS CI indicate a clear departure from normality. In this case, we recover the complete uMMD statistic, whose limiting distribution under the null is non-Gaussian.}. These findings are consistent with Corollary \ref{cor:BE_equi_stud}, which guarantees that, for $r=o(n^{1/3})$, the normal approximation improves as $n$ increases. At the same time, the behavior observed for $r=n^{0.5}$ suggests that extending the theoretical result to a regime approaching $r=o(n^{1/2})$ may be possible and represents an interesting direction for future work. Moreover, our variance estimator is specifically tailored to the fully degenerate setting (see Remark \ref{remark:adapt_var}). In Section~\ref{subsec:hsic}, we show that for the HSIC statistic, which under the null is not fully degenerate, comparable Gaussian behavior is observed only when standardizing by the true variance (see Figure~\ref{fig:HSIC_simstudy}, Supplement).

\begin{figure}[t]
    \centering
    \includegraphics[width=\linewidth]{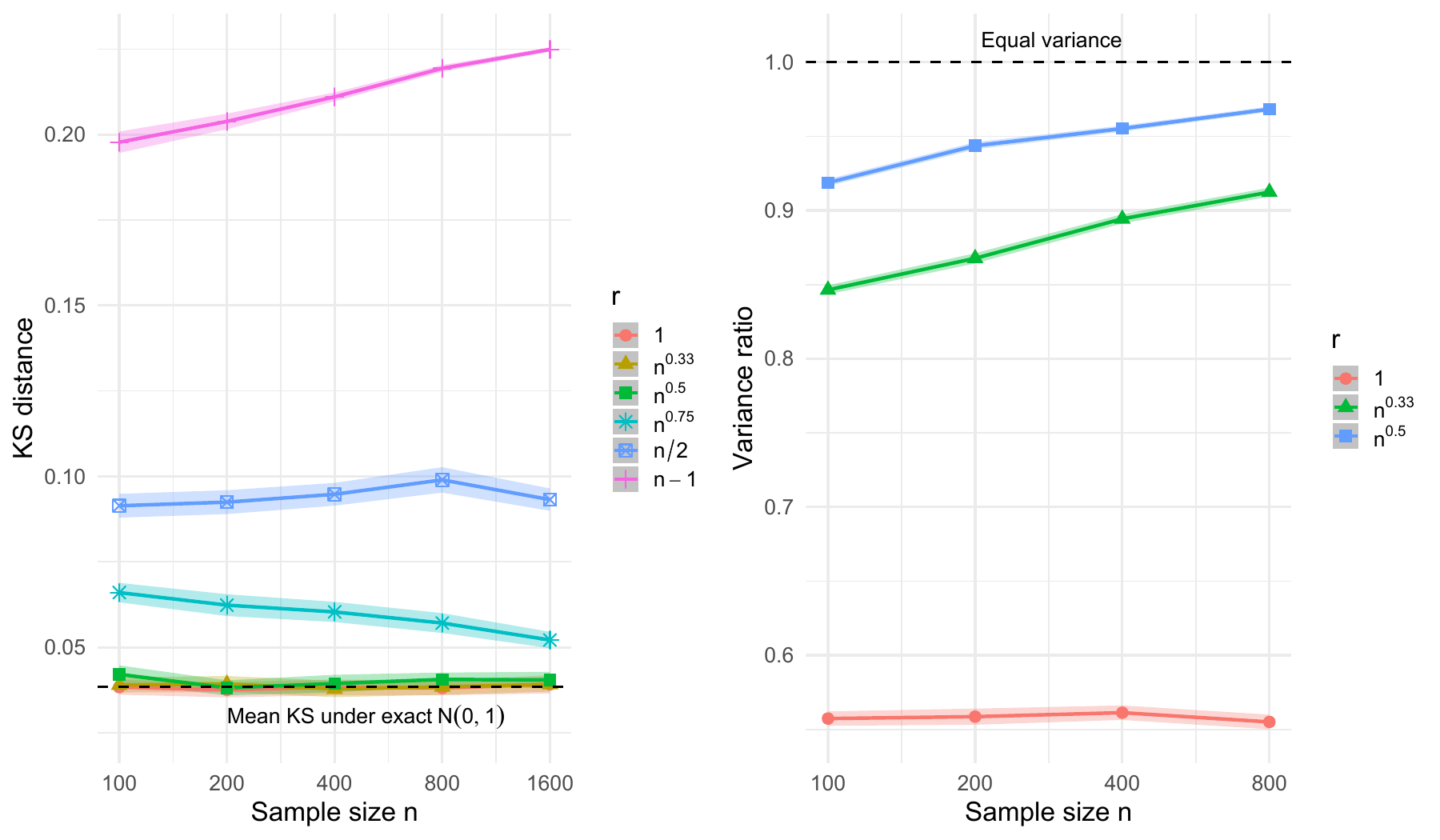}
   
    \caption{ (Left) KS distance between the empirical distribution of the standardized incomplete uMMD statistics under $H_0$ and the $\mathcal{N}(0,1)$ distribution.
(Right) Variance ratio under $H_1$ between the incomplete uMMD statistic based on equireplicate designs and that based on random designs. $95\%$ Monte Carlo CI are included for both experiments. Horizontal axes are displayed on a $\log_2$ scale, with tick labels corresponding to the original sample sizes.}
    \label{fig:MMD_simstudy}
\end{figure}

In the second study, we compare the variance of the uMMD statistic computed under equireplicate designs with that under random designs, to verify that with the equireplicate designs produced by Algorithm~\ref{alg:evenDesign} we achieve minimum variance. To this end, we perform another series of two-sample tests, this time drawing the two samples from different distributions so that the uMMD statistic is non-degenerate.

We simulate $X \overset{iid}{\sim} \mathcal{N}(0,1)$ and $Y \overset{iid}{\sim} \mathcal{N}(2,1)$. We vary $n \in \{100, 200, 400,800\}$ and select $r \in \{1,n^{0.33},n^{0.5}\}$. Then, we obtain Monte Carlo estimates of the true variance of the uMMD statistic under both the equireplicate design and under $100$ independently generated random designs. For the latter, we average the $100$ corresponding Monte Carlo variance estimates. We then take the ratio between the variance under the equireplicate design and this average variance under the random designs (equireplicate/random) as a measure of relative efficiency, and repeat the experiment $100$ times to obtain a sampling distribution of variance ratios. The right panel of Figure~\ref{fig:MMD_simstudy} reports $95\%$ Monte Carlo CI for these ratios, together with a reference line at $1$, which corresponds to equal variance. 
We observe that the upper bounds of all CI lie below $1$. As $r$ increases, the efficiency gains diminish but remain positive as $n$ grows. These findings validate Theorem~\ref{theo:algo_clt_equi}, confirming that the designs produced by Algorithm~\ref{alg:evenDesign} achieve the minimum-variance property.

\subsection{Experiments on the $\ell_4$ norm}\label{subsec:l4}

In the second set of studies, we consider a one-sample mean test based on the unbiased estimator of the fourth power of the $\ell_4$ norm proposed by \citet{he2021asymptotically}, which is a fourth-order U-statistic. Under the null $H_0$—i.e., when the population mean vector is zero—the corresponding kernel is fully degenerate. In the classical fixed-dimensional setting, the complete U-statistic therefore does not generally have a Gaussian limiting distribution, whereas \citet{he2021asymptotically} establish asymptotic normality under a high-dimensional regime. Under the alternative $H_1$, when the population mean vector is nonzero, the kernel is non-degenerate. We use incomplete versions of this statistic constructed via equireplicate designs to assess whether asymptotic normality can be recovered under $H_0$ through the specific design choice, even in a low-dimensional setting.

We simulate a single sample $X_1,\ldots,X_n$ i.i.d. from a $p$-dimensional $\mathcal{N}(0,I_p)$, with $p=5$, so that the null hypothesis of a zero mean vector holds. We choose $n \in \{200, 400,800,1600\}$. The incomplete fourth-order U-statistic is constructed from the unbiased estimator of the fourth power of the $\ell_4$ norm in Example~1 of \citet{he2021asymptotically}, taking $a=4$ and $p=5$, and is computed using an $r$-equireplicate design constructed by Algorithm~\ref{alg:k>2Design}, with $k=4$ and $\eta(j)=j$, inducing a high degree of overlap to test the regime covered by our theory (see Remark \ref{remark:euler}). We select $r \in \{1,n^{0.33},n^{0.5},4\phi(n)\}$\footnote{Note that $r$ must be a multiple of $k=4$. If it is not the case, we select the nearest multiple. Moreover, when $r=1$, we use the design $D_n^{\perp}$, described in Section \ref{subsec:stud}.}. As in the first MMD study, we obtain the sampling distribution of the KS distance between the empirical distribution of the standardized incomplete $\ell_4$-based $U$-statistic and $\mathcal{N}(0,1)$. The left panel in Figure \ref{fig:l4_simstudy} reports the mean KS distance together with $95\%$ Monte Carlo CI and, for reference, shows the mean KS distance under exact $\mathcal{N}(0,1)$ sampling.

We observe that, for $r \in \{1,n^{0.33},n^{0.5}\}$, the mean KS distance stabilizes around the normal-reference mean KS as $n$ increases. This is not the case for $r=4\phi(n)$, the largest replication parameter allowed by Algorithm~\ref{alg:k>2Design} (see Remark \ref{remark:euler}, e.g., when $n$ is prime then $\phi(n)=n-1$). Thus, we have evidence of a substantial departure from normality in finite-samples already at design sizes $|D| = O(n^2)$, although the KS distance decreases as $n$ grows. These findings are again consistent with Corollary~\ref{cor:BE_equi_stud} and Theorem~\ref{theo:CLT_equi_deg} for $k=4$.

\begin{figure}[tt]
    \centering
    \includegraphics[width=\linewidth]{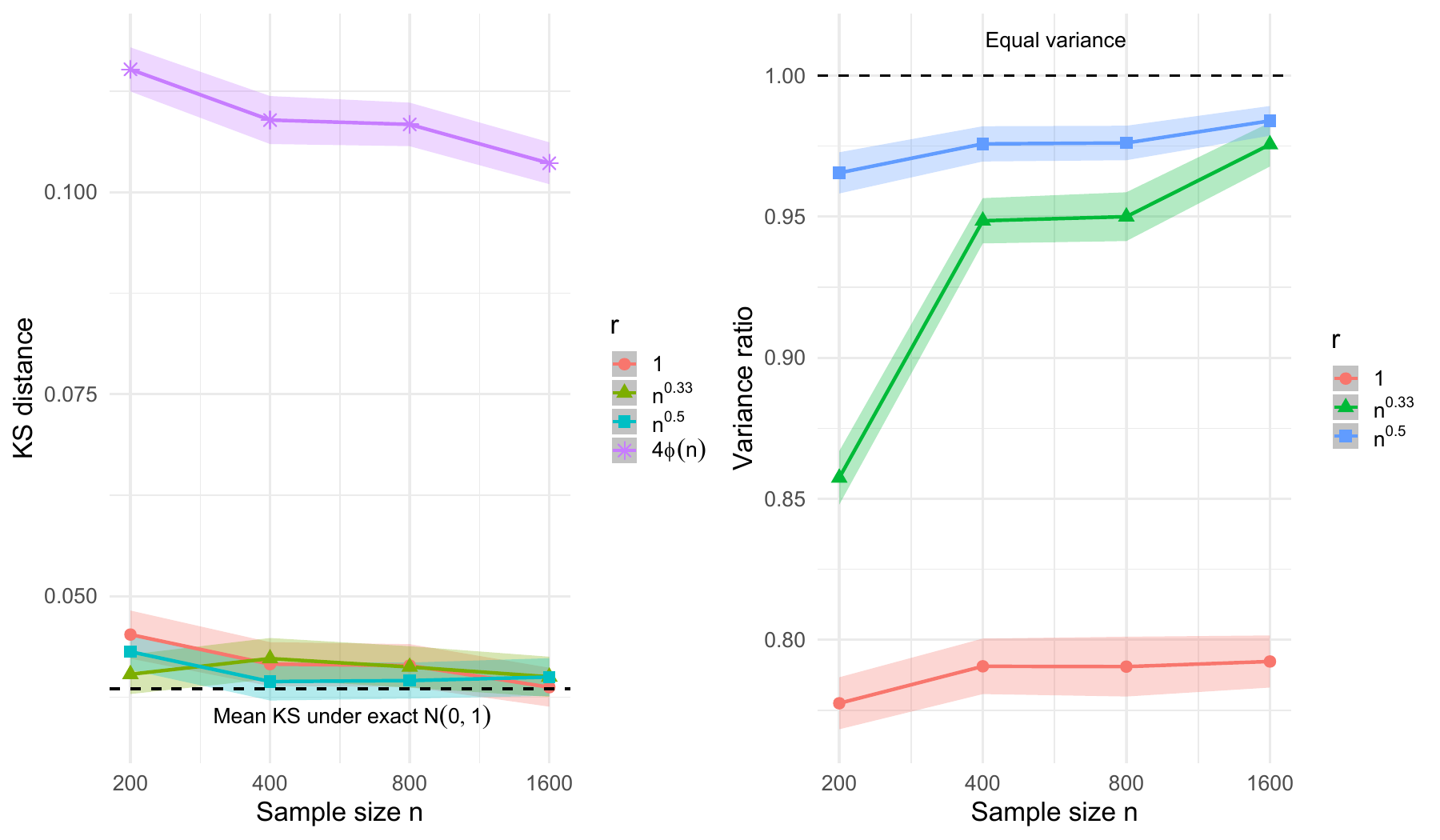}
    \caption{(Left) KS distance between the empirical distribution of the standardized incomplete $\ell_4$-based U-statistic under $H_0$ and the $\mathcal{N}(0,1)$ distribution.
(Right) Variance ratio under $H_1$ between the incomplete $\ell_4$-based U-statistic of equireplicate designs and that based on random designs. $95\%$ Monte Carlo CI are included for both experiments. Horizontal axes are displayed on a $\log_2$ scale, with tick labels corresponding to the original sample sizes.}
    \label{fig:l4_simstudy}
\end{figure}

Following the same approach of the second MMD study, we compare the variance of the $\ell_4$-based incomplete U-statistic computed under equireplicate designs with that under random designs, to verify that incomplete U-statistics based on the equireplicate designs produced by Algorithm~\ref{alg:k>2Design} do not lose efficiency with respect to those based on random designs. To this end, we perform another series of one-sample mean tests, this time generating samples from a distribution with a nonzero mean vector so that the corresponding fourth-order U-statistic is non-degenerate.
 
We simulate a single sample $X_1,\ldots,X_n$ i.i.d. from a $p$-dimensional $\mathcal{N}(\mathbf{1}_p,I_p)$, with $p=5$, so that the population mean vector is nonzero and the corresponding fourth-order U-statistic is non-degenerate. We vary $n \in \{200, 400,800, 1600\}$ and select $r \in \{1,n^{0.33},n^{0.5}\}$ to build equireplicate designs with Algorithm \ref{alg:k>2Design}, fixing $k=4$ and $\eta(j)=2^j$. As in the second MMD study, we obtain a sampling distribution of the ratio between the Monte Carlo variance estimate of the $\ell_4$-based incomplete U-statistic under the equireplicate design and the average of the corresponding variance estimates over $100$ random designs. The right panel of Figure~\ref{fig:l4_simstudy} reports $95\%$ Monte Carlo CI for these ratios (equireplicate/random), together with a reference line at $1$, which corresponds to equal variance.

There is no evidence of efficiency loss for equireplicate designs relative to random designs, since for every value of $r$ the CI intersect values below $1$. In contrast, we observe efficiency gains for all $r$, particularly at moderate sample sizes. Consistent with the second MMD study, these gains diminish as $r$ increases, but remain positive as $n$ grows. An important role is played by the choice of $\eta(j)$, which controls the degree of overlap in the resulting design. In the Supplement, Figure~\ref{fig:Lpnorm_simstudy_var} reports the results of the same experiment using first $\eta(j)=j$ and then $\eta(j)=j^2$, providing evidence that faster growth of $\eta(j)$ leads to larger variance gains for equireplicate designs relative to random designs.

\begin{figure}[t]
    \centering
    \includegraphics[width=\linewidth]{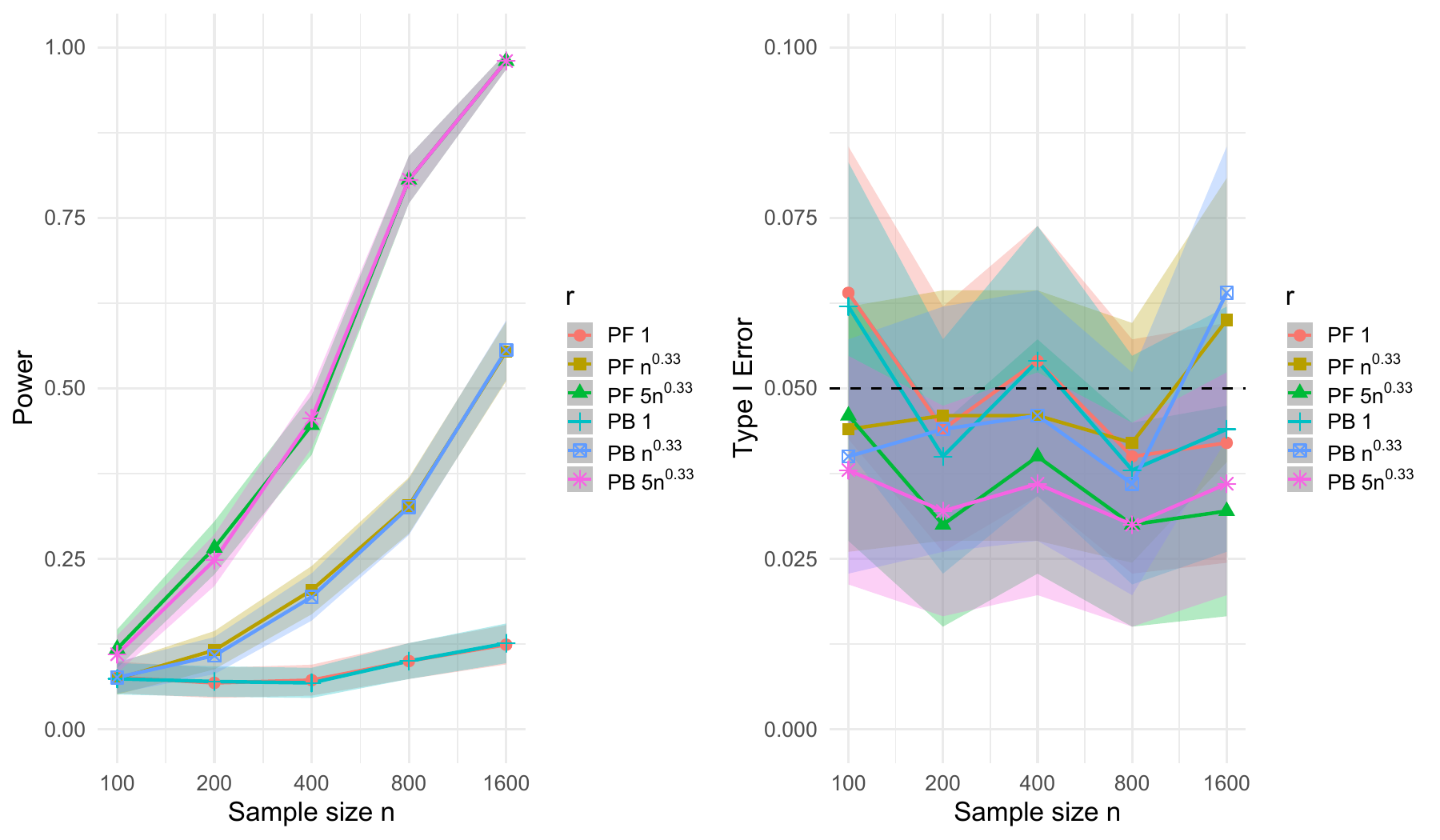}
    \caption{$95\%$ Monte Carlo CI for the power (left) and type I error (right) of the permutation-free (PF) version of the MMD test compared with its permutation-based (PB) counterpart, both evaluated on CIFAR-10 for different values of $n$ and $r$. Horizontal axes are displayed on a $\log_2$ scale, with tick labels corresponding to the original sample sizes.}
    \label{fig:realdata}
\end{figure}

\subsection{Real data example: CIFAR-10 dataset}\label{subsec:real_data}
We compare our permutation-free (PF) version of the MMD two-sample test against its standard permutation-based (PB) counterpart in terms of power, type I error, and runtime. Both methods rely on the same incomplete uMMD test statistic—built using equireplicate designs—to test for distributional differences between two balanced stratified samples of CIFAR-10 images. This dataset contains $60000$ color images of size $32 \times 32$ pixels, evenly distributed across $10$ mutually exclusive classes—airplane, automobile, bird, cat, deer, dog, frog, horse, ship, and truck—with 6000 images per class.

\begin{table}[t]
\centering
\small
\setlength{\tabcolsep}{4pt}
\begin{tabular}{c cc cc cc}
\toprule
& \multicolumn{2}{c}{$r=1$}
& \multicolumn{2}{c}{$r=n^{0.33}$}
& \multicolumn{2}{c}{$r=5\,n^{0.33}$} \\
\cmidrule(lr){2-3}
\cmidrule(lr){4-5}
\cmidrule(lr){6-7}
$n$ & PF & PB & PF & PB & PF & PB \\
\midrule

100
& $\mathbf{0.03} \pm 0.00$
& $\mathbf{10.21} \pm 0.17$
& $\mathbf{0.04} \pm 0.00$
& $\mathbf{31.43} \pm 0.08$
& $\mathbf{0.17} \pm 0.00$
& $\mathbf{168.77} \pm 0.41$ \\

200
& $\mathbf{0.03} \pm 0.00$
& $\mathbf{18.53} \pm 0.06$
& $\mathbf{0.08} \pm 0.00$
& $\mathbf{74.29} \pm 0.22$
& $\mathbf{0.40} \pm 0.00$
& $\mathbf{409.75} \pm 0.92$ \\

400
& $\mathbf{0.05} \pm 0.00$
& $\mathbf{40.77} \pm 0.09$
& $\mathbf{0.22} \pm 0.00$
& $\mathbf{213.07} \pm 0.49$
& $\mathbf{1.02} \pm 0.00$
& $\mathbf{1070.56} \pm 2.05$ \\

800
& $\mathbf{0.11} \pm 0.00$
& $\mathbf{83.18} \pm 0.19$
& $\mathbf{0.54} \pm 0.00$
& $\mathbf{549.69} \pm 0.88$
& $\mathbf{2.53} \pm 0.00$
& $\mathbf{2669.55} \pm 4.36$ \\

1600
& $\mathbf{0.29} \pm 0.00$
& $\mathbf{178.15} \pm 0.49$
& $\mathbf{1.42} \pm 0.01$
& $\mathbf{1437.10} \pm 2.22$
& $\mathbf{7.14} \pm 0.03$
& $\mathbf{7283.74} \pm 12.71$ \\

\bottomrule
\end{tabular}
\caption{Runtime in seconds (mean $\pm$ 95\% CI half-width), rounded to two decimal places, for different values of $n$ and $r \in \{1,n^{0.33},5n^{0.33}\}$ for the PF and PB MMD tests.}
\label{tab:runtime}
\end{table}

More specifically, we sample without replacement 
$X = \{X_{\text{cat}},\, X_{\text{deer}},\, X_{\text{ship}},\, X_{\text{truck}}\}$ and  
$
Y = \{Y_{\text{airplane}},\, Y_{\text{automobile}},\, Y_{\text{dog}},\, Y_{\text{horse}}\}$, where each class contributes equally, i.e. $|X_{\text{class}}| = |Y_{\text{class}}| = n/4$. In total, we use 8 classes from the CIFAR-10 dataset—4 animals (cat, deer, dog, horse) and 4 vehicles (airplane, automobile, ship, truck)—split evenly between $X$ and $Y$, which yields a final population of $N=48000$ images. We vary the common sample size $n \in \{100, 200, 400,800,1600\}$. For both methods, each incomplete uMMD statistic is computed using a Gaussian kernel with bandwidth selected via the standard median heuristic, and an $r$-equireplicate design generated by Algorithm~\ref{alg:evenDesign}. For the PF test, we use the variance estimator of Theorem~\ref{theo:CLT_equi_deg}. We select $r \in \{1,n^{0.33},5n^{0.33}\}$, where the constant $5$ was chosen based on an additional simulation study assessing which multiplicative constant in $c\,n^{0.33}$ preserves the nominal type I error (see Figure~\ref{fig:realdata_const}, Supplement). We fix $B=1000$ permutations for the PB test and choose $\alpha = 0.05$. We repeat the experiment $500$ times to obtain Monte Carlo estimates of power, type~I error (with $X$ and $Y$ drawn from the same distribution), and the empirical distribution of runtimes. The left panel of Figure~\ref{fig:realdata} shows $95\%$ Monte Carlo CI for power, while the right panel reports the corresponding CI for type~I error. Table~\ref{tab:runtime} summarizes mean runtimes (in seconds) along with the half-width of the $95\%$ CI.

We find no evidence of a loss of power for the PF MMD test relative to the PB version. For both methods, increasing $r$ yields substantial power gains without compromising type~I error control (see the right panel of Figure~\ref{fig:realdata}). For all values of $n$ and $r$, the empirical type~I error remains broadly compatible with the nominal level $\alpha=0.05$, with the corresponding Monte Carlo CI generally covering $\alpha$. The choice $r=5n^{0.33}$, tends to be somewhat conservative, particularly at moderate sample sizes. Overall, for any given $r$, the empirical type~I error of the PF and PB procedures follows a very similar pattern across sample sizes. These results provide empirical support for the asymptotic validity predicted by Theorem~\ref{theo:CLT_equi_deg}, of which the PF MMD test is a special case. In terms of computation, the speedup is substantial: the PF test runs about $1000\times$ faster than PB, while retaining both power and validity. More broadly, these findings confirm the expected computational gains are proportional to $B$, the number of permutations.

\section{Conclusion}
We introduced a novel characterization of the dependence structure of a U-statistic 
via its dependency graph. This perspective allowed us to derive a new Berry–Esseen bound that applies to all incomplete U-statistics based on deterministic designs, establishing conditions for 
Gaussian limiting distributions even in the fully degenerate case and when the order diverges. We further 
developed efficient algorithms for constructing incomplete U-statistics using equireplicate designs, 
a subclass of deterministic designs that, for second-order U-statistics, achieve minimum variance. 
All theoretical results have been validated through extensive numerical experiments. Finally, applying our framework to kernel-based testing, we proposed a permutation-free version of the MMD two-sample 
test, which—as shown by our real data example—delivers substantial computational gains while preserving both power and type I error control.
 
An important direction for future work is to extend our results to incomplete U-statistics 
with random designs, particularly for the case $k>2$. 
In addition, future research may investigate whether $\Delta(\mathcal{D}_n)$ can grow at larger polynomial rates--as suggested by our simulation study of Section \ref{subsec:mmd}--and still maintain asymptotic normality. Moreover, one could also consider more structured subclasses of equireplicate designs, such as BIBDs, together with mixed-degenerate settings. In such cases, sharper bounds may be obtainable when the variance can be characterized explicitly, although this would require an adaptive variance estimator as in \cite{shao2026higher}. This question is closely related to understanding how the choice of $\eta(j)$ controls the overlap pattern of the designs produced by Algorithm~\ref{alg:k>2Design}, and hence identifying conditions under which these designs are guaranteed to be more efficient than random designs. Further work is also needed in the infinite-order regime: both to evaluate the applicability of existing variance estimators within our framework (see Remark~\ref{remark:var_inford}) and to develop new ones tailored to the equireplicate design construction of Algorithm~\ref{alg:k>2Design}. More broadly, an interesting challenge is to extend our framework to dependent U-statistics \citep{dehling2002dependent}, which frequently arise in applications, including the analysis of time-series and network data.

\section*{Acknowledgments and Funding}
This work was supported in part by NSF grant SES-2150615. This research was also supported in part by the University of Pittsburgh Center for Research Computing through the resources provided. Specifically, this work used the HTC cluster, which is supported by NIH award number S10OD028483. We also acknowledge the support of Purdue University, where both authors were affiliated at the beginning of this project. In preparing this work, OpenAI ChatGPT (GPT-5 and GPT-6) was used to assist with language proofreading, code refinement, and answer research-related questions. The authors independently reviewed and verified all AI-assisted content and take full responsibility for the integrity of the manuscript.

\bibliographystyle{apalike}
\bibliography{biblio}

\renewcommand{\thesection}{S\arabic{section}}
\renewcommand{\thetheorem}{S\arabic{theorem}}
\renewcommand{\thelemma}{S\arabic{lemma}}
\renewcommand{\thecorollary}{S\arabic{corollary}}
\renewcommand{\theproposition}{S\arabic{proposition}}
\renewcommand{\thedefinition}{S\arabic{definition}}
\renewcommand{\theexample}{S\arabic{example}}
\renewcommand{\theremark}{S\arabic{remark}}
\renewcommand{\theobservation}{S\arabic{observation}}

\section{Introduction}\label{supp:sec1}
 
\textbf{Related Work (extended version).} In his seminal work, \cite{blom1976some} introduces incomplete U-statistics, analyzes their finite-sample and asymptotic variance, and establishes conditions for asymptotic normality in the \emph{non-degenerate} case. He is also the first to suggest methods for constructing minimum variance designs, such as Latin squares and Graeco-Latin squares. \cite{brown1978reduced} derive asymptotic properties for incomplete U-statistics of second order, under an equireplicate design structure. They are the first to incorporate graph-theoretic language in their proofs. However, they do not consider the case of an odd replication parameter when the sample size $n$ is even, and the parameter is not allowed to grow with $n$. \cite{weber1981incomplete} and \cite{oneil1993asymptotic} show that in the \emph{degenerate} case, the limiting behavior can be either standard or non-standard, depending on the choice of the design. \cite{lee1982incomplete} studies the problem of choosing minimum variance designs for incomplete U-statistics. \cite{janson1984asymptotic} provides a comprehensive treatment of the asymptotic distribution of incomplete U-statistics of order $k$, considering both random and deterministic designs. However, it does not investigate the class of equireplicate designs. In addition, \cite{janson2004large} develops a graph-theoretic approach related to ours. The author exploits the dependency structure induced by overlapping index sets to derive exponential concentration inequalities for U-statistics, including incomplete ones. While the analysis focuses on large-deviation bounds obtained through graph-colouring arguments, our approach combines the dependency-graph structure with Stein's method to establish quantitative Gaussian approximation results, including under full degeneracy and, for equireplicate designs, studentization.

In general, all previously mentioned works—being theoretical in  nature—lack a practical procedure for constructing minimum variance designs. On the other hand, \cite{rempala2003incomplete} and \cite{kong2021design} propose asymptotically efficient incomplete U-statistics constructions, but they do not provide a finite sample analysis.  Existing works addressing convergence rates to limiting distributions of incomplete U-statistics are: \cite{rinott1997coupling}, which examine a Markov-type dependence framework with applications to incomplete U-statistics; \cite{chen2019randomized}, which consider randomized incomplete U-statistics in high-dimensional settings; \cite{sturma2024testing} which extend \cite{chen2019randomized} to a \emph{mixed degenerate} setting with applications to testing a null hypothesis defined by equality and inequality constraints; \cite{kim2024dimension}, which propose a new class of incomplete U-statistics that enable valid inference regardless of how the dimension scales with the sample size; \cite{shao2023u,shao2025u}, that develop higher-order approximations for the sampling distributions of studentized incomplete U-statistics under both deterministic and randomized designs, covering non-degenerate and degenerate settings; and \cite{leung2026berry}, which analyzes incomplete U-statistics of random designs—generated via Bernoulli sampling—across different regimes relating the sample size $n$ and the expected design size. 

Among these prior works, none has provided a comprehensive framework encompassing: (i) convergence rates to normality for incomplete U-statistics of deterministic designs under full degeneracy, including studentized bounds for equireplicate designs; (ii) asymptotic normality for equireplicate designs of size up to $o(n^{4/3})$ when $k$ is fixed, alongside results allowing $k$ to diverge with $n$; and (iii) efficient algorithms for constructing minimum-variance designs when $k=2$. 

\section{Background and Notation}\label{supp:sec2}

\subsection{Minimum variance designs}\label{app:min_var}

Being equireplicate is sufficient for minimum variance designs only in the special case $k=2$ (see Theorem~1 on page~195 of \cite{lee1990u}). For $k>2$, additional structural constraints become necessary, as the intersection sizes among the subsets in the design can no longer be adequately controlled. This is why, in the combinatorial design literature, researchers impose additional conditions, beyond the equireplicate one, in order to obtain more balanced designs. For instance, any BIBD must satisfy the condition $r(k-1) = \lambda (n-1)$, where the parameter $\lambda$ represents the number of design elements in which each distinct pair of elements from $V$ appears together. This additional constraint ensures balance not only at the level of individual elements of $V$, but also among pairs of elements. However, explicit constructions of BIBDs are known only for specific cases e.g., when $k=3$, we have \emph{Steiner triple systems} (see ch. 12 in \cite{wallis2016introduction}) and certain specific families of designs described in \cite{sprott1954note}. Constructing BIBDs for general values of $k$, on the other hand, remains a notoriously challenging problem. Another approach is to construct a design such that $f_c = 0$ for all $c \in \{2, \ldots, k\}$ i.e., to require that the intersection between any two distinct elements of $D$ contains at most one element. If an equireplicate design satisfies this property, then by Theorem 2, page 196 in \cite{lee1990u}, it attains minimum variance. In the main text, we refer to these type of designs as equireplicate and \emph{linear} designs (see Remark \ref{remark:be_equi_lin} for further details). One way to practically build these designs is shown in Example 7 of \cite{lee1982incomplete}. Another particularly interesting way, was recently introduced in \cite{shao2025u}-Section 3.1-where the authors provide a novel approach to avoid unwanted overlaps among the blocks of the design.

\subsection{Adjacency matrix of the line graph of an hypergraph}\label{app:adj_mat}

In this paragraph, we introduce the adjacency matrix of $L(H)$, the line graph of the hypergraph $H=(V,E)$. In particular, we denote it by $A_{L(H)}\in \{0,1 \}^{|E| \times |E|}$ and define it as:
$$
(A_{L(H)})_{i,j} =
\begin{cases}
1, & \text{if } |e_i \cap e_j| \neq 0, \\
0, & \text{otherwise}.
\end{cases}
$$

To give a concrete example, consider $K_{n}^{(k)}$ and its line graph $L(K_{n}^{(k)})$. Then $A_{L(K_{n}^{(k)})} \in \{0,1 \}^{\binom{n}{k} \times \binom{n}{k}}$ and the matrix indicates whether any given pairwise intersection among subsets of size $k$ of $V$ is empty or not.  In the same way, we can consider the hypergraph of the design $\mathcal{D} = (V,D)$, and denote its line graph with $L(\mathcal{D})$. Then, its adjacency matrix $A_{L(\mathcal{D})} \in \{0,1 \}^{m \times m}$ is a sub-matrix of $A_{L(K_{n}^{(k)})}$ that indicates whether any given two elements of the design $D$ share at least one index or not.

\subsection{Equireplicate designs and $k$-uniform, $r$-regular hypergraphs}\label{app:equi_hyper}

\begin{observation}\label{obs:equi_hyper_rep}
The hyperedge set of any $k$-uniform, $r$-regular hypergraph defines an $r$-equireplicate design on the vertex set $V$, and conversely, any $r$-equireplicate design with blocks of size $k$ corresponds to the hyperedge set of a $k$-uniform, $r$-regular hypergraph.
\end{observation}

 The above holds true because, as discussed in \ref{subsec:hyper}, the hyperedge set of a $k$-uniform hypergraph, being a subset of $\mathcal{B}_k$ by definition, identifies uniquely a design $D$ on the vertex set $V$. If the $k$-uniform hypergraph is also $r$-regular, then it means that each vertex $v \in V$ appears exactly $r$ times in the collection of all hyperedges. Thus, the corresponding design $D$ must be $r$-equireplicate since the vertices are the indices and occur in the same number of blocks i.e., elements of the design. The converse is true for the reverse reasoning. Because of this equivalence, equireplicate designs are also known as \emph{regular designs}. Moreover, equation (\ref{eq:equirep}) can now be interpreted as an extension of the classical \emph{handshaking lemma}—that relates the number of edges in a graph with the sum of the degrees—for $k$-uniform, $r$-regular hypergraphs.

\section{Normal Approximations of Incomplete U-statistics}\label{supp:sec3}

\subsection{Proof of Proposition \ref{prop:dep_graph}}

\begin{proof}
  As explained in section \ref{sec:back_not}, the design $D$ of $U_{n,D}^{(k)}$ is the vertex set of the line graph $L(\mathcal{D}) = (D,E)$ and $S$ is a generic element of $D$ i.e., a block of the design, which corresponds to an hyperedge of the hypergraph $\mathcal{D}$. Consider any pair of disjoint sets $\Gamma_1$ and $\Gamma_2$ in $D$ such that no edge in $E$ has one endpoint in $\Gamma_1$ and the other in $\Gamma_2$. This is equivalent to impose that, given any two hyperedges $S_1$ and $S_2$ such that $S_1 \in \Gamma_1$ and $S_2 \in \Gamma_2$, no edge connects them (note that $S_1 = S_2$ is ruled out since $\Gamma_1$ and $\Gamma_2$ are disjoint). However, by definition of the line graph of an hypergraph, if there is not an edge between two distinct hyperedges $S_1, S_2 \in E$, then $|S_1 \cap S_2| = 0$. Moreover, this implies that the two distinct hyperedges $S_1, S_2$ cannot have an index $i \in V$ in common. But then, since $X_1, \ldots, X_n$ are i.i.d. random variables indexed by $V = \{1, \ldots, n \}$ by assumption, this means that $h(S_1)$ must be independent of $h(S_2)$ because dependence can happen if and only if $h(S_1)$ and $h(S_2)$ share at least an index. Thus, the sets of random variables $\left\{ h(S), \;S \in \Gamma_1 \right\}$ and $\left\{ h(S), \;S \in \Gamma_2 \right\}$ are independent. 
  Therefore, $L(\mathcal{D}) = (D,E)$ is the dependency graph of $\left\{ h(S), \;S \in D \right\}$, as all required conditions have been verified.
\end{proof}

\subsection{Proof of Lemma \ref{lemma:maxdeg}}

\begin{proof}
      We start from the upper bound. By definition, see subsection \ref{subsec:hyper}, $L(\mathcal{D})$ is a graph, with vertex set that coincides with the hyperedge set of $\mathcal{D}$ and where an edge connects two vertices if and only if the corresponding hyperedges in $\mathcal{D}$ i.e., the blocks of the design, have at least one vertex i.e., index, in common. Now, take a generic vertex $v \in V$ and consider an hyperedge $e \in D$ such that $v \in e$. Since $\Delta(\mathcal{D})$ is the maximum degree of $\mathcal{D}$, then each $v$ can appear at most in exactly $\Delta(\mathcal{D})$ hyperedges. Furthermore, the index $v \in e$ can contribute at most for $(\Delta(\mathcal{D})-1)$ edges in $L(\mathcal{D})$. This because there are $(\Delta(\mathcal{D})-1)$ hyperedges left which have the index $v$ in common with $e$. But $\mathcal{D}$ is $k$-uniform, so in each $e \in E$ there are exactly $k$ indices, each contributing at most for $(\Delta(\mathcal{D})-1)$ edges. To conclude the proof of the upper bound, we just need to consider that each vertex of $L(\mathcal{D})$ has a loop i.e., a self-edge, by definition (see subsection \ref{subsec:hyper}). Thus the maximum degree $\Delta(L(\mathcal{D}))$ must be smaller or equal to $k\;(\Delta(\mathcal{D})-1) + 1$. The upper bound is met when the hypergraph $\mathcal{D}^{\diamond} = (V, D^{\diamond})$ is both $r$-regular and \emph{linear}, meaning that for all $e_1, e_2 \in D^{\diamond}$, with $e_1 \neq e_2$, we have that $\left|e_1 \cap e_2\right| \leq 1$. We prove this result by contradiction. Suppose that $\Delta(L(\mathcal{D}^{\diamond})) \neq k\;(r-1) + 1$, then, by the previous result on the upper bound, we know that $\Delta(L(\mathcal{D}^{\diamond})) < k\;(r-1) + 1$. This implies that there exist at least two distinct indices $v_1, v_2 \in V$ that belong to two distinct hyperedges $e_1,e_2 \in D^{\diamond}$. If this was not the case, then each $v \in V$ would contribute exactly for $(r-1)$ edges in $L(\mathcal{D^{\diamond}})$—avoiding overlaps—because $\mathcal{D}^{\diamond}$ is $r$-regular and, since $\mathcal{D}^{\diamond}$ is also $k$-uniform, the line graph would be $k\;(r-1)+1$-regular, taking into account that each vertex of $L(\mathcal{D}^{\diamond})$ has a loop i.e., a self-edge, by definition (see subsection \ref{subsec:hyper}). But this means that there exist at least two distinct hyperedges $e_1, e_2 \in D^{\diamond}$ such that $\left|e_1 \cap e_2\right| > 1$ which contradicts the fact that the hypergraph $\mathcal{D}^{\diamond}$ is linear. Thus, $\Delta(L(\mathcal{D}^{\diamond})) = k\;(r-1)+1$. Moreover, since two distinct indices $v_1, v_2 \in V$ that belong to two distinct hyperedges $e_1,e_2 \in D^{\diamond}$ cannot exist, by the previous reasoning $L(\mathcal{D}^{\diamond})$ must also be $k\;(r-1)+1$-regular.
      
      For the lower bound, we need to do the opposite reasoning. Take $v \in V$ such that $d(v) = \Delta(\mathcal{D})$ and consider an hyperedge $e \in D$ such that $v \in e$. Since $\Delta(\mathcal{D})$ is the maximum degree of $\mathcal{D}$, $v$ will appear in exactly $\Delta(\mathcal{D})$ hyperedges in $D$. Furthermore, the index $v \in e$ contributes for $(\Delta(\mathcal{D})-1)$ edges in $L(\mathcal{D})$. This because there are $(\Delta(\mathcal{D})-1)$ hyperedges left which have the index $v$ in common with $e$. But $\mathcal{D}$ is $k$-uniform, so there are $k-1$ indices left which can contribute to the degree of $e$. Since we want to prove a lower bound, we just need to impose that each of the remaining $k-1$ indices does not contribute at all to the degree of $e$, seen as a vertex of $L(\mathcal{D})$. To conclude the proof of the lower bound, note that each vertex of $L(\mathcal{D})$ has a loop i.e., a self-edge, by definition (see subsection \ref{subsec:hyper}). Thus the maximum degree $\Delta(L(\mathcal{D}))$ must be greater or equal to $\Delta(\mathcal{D})$. The lower bound is met by the $k$-uniform star, which is the $t=1$ case in definition 1.10 of \cite{keevash2014spectral}. We show this for $k=2$ i.e., when the hypergraph $\mathcal{D}$ is a star graph with $n$ vertices. The center of the star is the vertex $v^{\star} \in V$ with the maximum degree $\Delta(\mathcal{D}) = n-1$. Now, consider an edge $e$ such that $v^{\star} \in e$. There are exactly $n-2$ edges left in $D$ that have the index $v^{\star}$ in common with $e$ and the remaining index in $e$ does not contribute at the degree of $e$ by construction. Thus, $\Delta(L(\mathcal{D})) = n-1$ because we add the self-edge to the previous count of $n-2$ edges.   
\end{proof}

\subsection{Proof of Theorem \ref{theo:BE_bound}}

\begin{proof}

Notice that, by assumption $0 <\sigma^2_k$ and $\sigma^2_k < \infty$ follows again by assumption, since $\|h(S)-\mu_k\|_4 < \infty$. This implies, by inequality~\eqref{eq:ineq} and identity $\sum_{c=0}^k f_c = |D|^{2}$, that $0<\sqrt{\sum_{c=0}^k f_c \;\sigma_c^2} < \infty$ for a fixed design size. Thus, we are allowed to standardize and consider the expression of the centered and standardized incomplete U-statistic of order $k$. We start by substituting equations \eqref{eq:IUk_red} and \eqref{eq:var_ustat} in that expression to obtain the following result:

\begin{equation*}
\begin{aligned}
    \frac{U_{n,D}^{(k)} - \mu_{k}}{\sqrt{\operatorname{Var}U_{n,D}^{(k)}}} &= \frac{|D|^{-1}\left(\sum_{S \in D} h(S) - |D| \;\mu_k\right)}{\sqrt{|D|^{-2} \sum_{c=0}^k f_c \;\sigma_c^2}} \\
     &= \sum_{S \in D} \frac{h(S) - \mu_k}{\sqrt{\sum_{c=0}^k f_c \;\sigma_c^2}} = \frac{\sum_{S \in D} T_S}{\sqrt{\operatorname{Var}(\sum_{S \in D} T_S)}} \;,
\end{aligned}
\end{equation*}

  where we denote by $T_S$, the centered version of $h(S)$ for all $S \in D$ and knowing that $D$ is non-empty by assumption. At this point, we conclude that $L(\mathcal{D})$ is also the dependency graph of the set of random variables $\{T_S, \;S \in D\}$. This is because we just subtracted the same constant for all the random variables in the set $\{h(S), \;S \in D\}$, without changing their dependency structure. Consequently, Proposition \ref{prop:dep_graph} holds as well for the set $\{T_S, \;S \in D\}$. Now, we have that: i) $\{T_S, \;S \in D\}$ are random variables indexed by the vertices of $L(\mathcal{D})$, and ii) by construction, $E[T_S] = 0$ for all $S \in D$. Therefore, \emph{Theorem 3.1} in \cite{cai2026refined} applies, and we can conclude that:

\begin{equation*}
\begin{aligned}
    \sup_{z \in \mathbb{R}}\left|P\left(\frac{\sum_{S \in D} T_S}{\sigma_T} \leq z\right)-\Phi(z)\right| &\leq C_1 \; \Delta(L(\mathcal{D}))^{2}\;|D| \; \left(\frac{\|T_S\|_4}{\sigma_T}\right)^{3} \\
    &+ C_2 \; \Delta(L(\mathcal{D}))^{3/2}\;\sqrt{|D|} \; \left(\frac{\|T_S\|_4}{\sigma_T}\right)^{2} \;\;\;.  
\end{aligned}
\end{equation*}

where $\sigma_T = \sqrt{\operatorname{Var}(\sum_{S \in D} T_S)}$, $C_1,C_2 >0$ are absolute constants and we used the fact that $\{T_S, \;S \in D\}$ are identically distributed (which implies e.g., $\sum_{S \in D} \|T_S\|^{3}_4 = |D| \; \|T_S\|^3_4$). At this point, we use $\sqrt{|D|}\; \sigma_k$ to lower bound $\sigma_T$ by Lemma \ref{lemma:ord_var}. After some simplifications, we obtain that: 

\small{
\begin{equation*}
\begin{aligned}
    \sup_{z \in \mathbb{R}}\left|P\left(\frac{\sum_{S \in D} T_S}{\sigma_T} \leq z\right)-\Phi(z)\right| &\leq C_1 \; \frac{\Delta(L(\mathcal{D}))^{2}}{\sqrt{|D|}} \; \left(\frac{\|T_S\|_4}{\sigma_k}\right)^{3}
    + C_2 \; \frac{\Delta(L(\mathcal{D}))^{3/2}}{\sqrt{|D|}} \; \left(\frac{\|T_S\|_4}{\sigma_k}\right)^{2}   \\
    &\leq 2\max(C_1,C_2)\; \frac{\Delta(L(\mathcal{D}))^{2}}{\sqrt{|D|}} \; \left(\frac{\|T_S\|_4}{\sigma_k}\right)^{3} \;\;, 
\end{aligned}
\end{equation*}
} \normalsize

where the last passage follows from the facts that $\Delta(L(\mathcal{D})) \geq 1$ by definition and $\|T_S\|_4 \geq \|T_S\|_2 = \sigma_k$, thus meaning that the first term upper bounds the second. To conclude, we substitute the tight upper bound of Lemma \ref{lemma:maxdeg} instead of $\Delta(L(\mathcal{D}))$, to obtain the desired expression:

\begin{equation*}
      \sup_{z \in \mathbb{R}}\left|P\left(\frac{\sum_{S \in D} T_S}{\sigma_T} \leq z\right)-\Phi(z)\right|\leq C\; \frac{\{k \;(\Delta(\mathcal{D})-1) + 1\}^{2}}{|D|^{1/2}} \; \left(\frac{\|T_S\|_4}{\sigma_k}\right)^{3} \;\;.
\end{equation*}

This ends the proof since we showed previously that $\frac{\sum_{S \in D} T_S}{\sigma_T} = \frac{U_{n,D}^{(k)} - \mu_{k}}{\sqrt{\operatorname{Var}U_{n,D}^{(k)}}}$.\\
We underline that by using $\sqrt{|D|}\; \sigma_k$ to lower bound $\sigma_T$, we are implicitly considering the extreme case scenario of a degeneracy of order $k-1$. However, if one is willing to further assume that both $f_c > 0$ and $\sigma^2_c > 0$ for at least one $c \in \{1, \ldots, k-1\}$, the previous bound can be made sharper. This is because, by Lemma \ref{lemma:ord_var}, $\sum_{c=0}^{k} f_c \;\sigma_c^2 > |D| \;\sigma_k^2$ if there exists a $c \in \{1, \ldots, k-1\}$ such that $f_c > 0$ and $\sigma^2_c > 0$. However, interpreting the previous expression for the whole class of deterministic designs is challenging. This is due to the fact that the $f_c$ terms are specific to each deterministic design construction and require analyzing their growth rate as the design size increases. In Theorem \ref{theo:BE_equi_lin_nondeg}, we provide an example of a sharper bound for the particular case of equireplicate and linear designs, where we assumed that $\sigma^2_1 > 0$.
\end{proof}

\subsection{Equireplicate vs non-equireplicate designs}\label{subsec:equi_vs_nonequi}
We can compare the maximum degree of the dependency graph for an incomplete U-statistic based on an $r$-equireplicate design $D^{\dagger}$ with that of one based on a non-equireplicate design $D$. Specifically, if $k\;(r - 1) + 1 < \Delta(\mathcal{D})$, then it must follow that $\Delta(L(\mathcal{D}^{\dagger})) < \Delta(L(\mathcal{D}))$ by Lemma \ref{lemma:maxdeg}. When $k = 2$ i.e., for second-order U-statistics, this condition is unnecessary: if both designs have the same size, the $r$-equireplicate one always yields a lower maximum degree in the dependency graph and is therefore preferable for minimizing dependence. Furthermore, when $|D^{\dagger}| = |D|$, we also have that $ r =  \Delta(\mathcal{D}^{\dagger})< \Delta(\mathcal{D})$ for any value of $k$. These results are formally stated and proved in Proposition~\ref{prop:maxdeg_linegraph} and Lemma~\ref{lemma:equi_r_vs_det} below.

\begin{proposition}\label{prop:maxdeg_linegraph}
    Let $\mathcal{D} = (V, D)$ be the hypergraph of a non-equireplicate deterministic design $D$ and let $\mathcal{D}^{\dagger} = (V, D^{\dagger})$ be the hypergraph of an $r$-equireplicate design $D^{\dagger}$. Whenever $r < \frac{\Delta(\mathcal{D}) - 1}{k} + 1$, we have that
    
    $$\Delta(L(\mathcal{D}^{\dagger})) < \Delta(L(\mathcal{D})) \;\;.$$

    Moreover, if both designs have same block size $k = 2$ and cardinality i.e., $|D| = |D^{\dagger}|$, then 
    $$\Delta(L(\mathcal{D}^{\dagger})) < \Delta(L(\mathcal{D})) \;\;.$$
\end{proposition}

\begin{proof}
    $\mathcal{D}^{\dagger} = (V, D^{\dagger})$ is an $r$-equireplicate design and thus $\Delta(\mathcal{D}^{\dagger}) = r$. Then, by Lemma \ref{lemma:maxdeg}, we can conclude that

    $$ \Delta(L(\mathcal{D}^{\dagger})) \leq k\;(r-1) + 1 \; .$$

Again by Lemma \ref{lemma:maxdeg}, we know that for a generic deterministic design $\Delta(\mathcal{D}) \leq \Delta(L(\mathcal{D})) $. Thus, if the upper bound $k\;(r-1) + 1$ is strictly smaller than $\Delta(\mathcal{D})$, then we know that $\Delta(L(\mathcal{D}^{\dagger})) < \Delta(L(\mathcal{D}))$. By rewriting the condition in terms of the replication parameter $r$, we obtain that whenever $r < \frac{\Delta(\mathcal{D}) - 1}{k} + 1$ we have that $\Delta(L(\mathcal{D}^{\dagger})) < \Delta(L(\mathcal{D}))$, therefore proving the first part of the proposition.

For the second part of the proposition, we assume that $k = 2$ and $|D| = |D^{\dagger}|$. In this specific situation, the hypergraph $\mathcal{D}^{\dagger}$ is not only $r$-regular, which is the case for every $r$-equireplicate designs, but also linear i.e., for all $e_1, e_2 \in D^{\dagger}$, with $e_1 \neq e_2$, we have that $\left|e_1 \cap e_2\right| \leq 1$. Thus, as already explained in the proof of Lemma \ref{lemma:maxdeg} when we treat the case in which the upper bound is met, we can conclude that $\Delta(L(\mathcal{D}^{\dagger})) = 2\;(r-1)+1$. Now, consider a generic hyperedge $e \in D$ i.e., a vertex of $L(\mathcal{D})$. Since $k=2$, we know that $|e| = 2$ and, without loss of generality, we can consider vertices $v_1,v_2 \in V$ such that $e=\{v_1,v_2\}$. The degree of $e$, seen as a vertex of $L(\mathcal{D})$, is equal to $d(v_1) + d(v_2) - 1$. This because $v_1$ contributes for exactly $d(v_1) - 1$ edges in $L(\mathcal{D})$ and, equivalently, $v_2$ contributes for $d(v_2) - 1$. Overlaps cannot occur when $k=2$ since the hypergraph $\mathcal{D}$ is linear. To obtain $d(e)= d(v_1) + d(v_2) - 1$, we just add the self-edge, which is always present in $L(\mathcal{D})$, to the previous count. Besides, note that since $\Delta(L(\mathcal{D}^{\dagger})) = 2\;(r-1)+1$, if we show that there exists at least a hyperedge $e = \{v_1,v_2\} \in D$ such that $d(e) > 2\;(r-1) + 1$ then we can conclude that $\Delta(L(\mathcal{D}^{\dagger})) < \Delta(L(\mathcal{D}))$. Substituting the previously obtained value of $d(e)$ and simplifying the expression, the condition becomes $d(v_1) + d(v_2) > 2\;r$. To prove that this holds, we start noticing that 

$$\sum_{\{i,j\} \in D} \{d(v_i) + d(v_j)\} = \sum^{n}_{i = 1} d(v_i)^2$$ 

because a vertex $v_i \in V$ of degree $d(v_i)$ is incident with $d(v_i)$ edges, and each of those edges contributes $d(v_i)$ once to the left sum. Now, knowing that $|D^{\dagger}| = |D|$ by assumption, we apply Cauchy–Schwarz inequality on the vector of degrees  $\boldsymbol{d(v)} =[d(v_1),d(v_2),\ldots, d(v_n)]$ and a vector of ones of length $n$, to obtain

\begin{equation*}
\sum^{n}_{i = 1} d(v_i)^2 > \frac{\left\{\sum^{n}_{i = 1} d(v_i)\right\}^2}{n}\overset{\eqref{eq:handshaking_kuni}}{=} \frac{(2 |D^{\dagger}|)^2}{n}\overset{\eqref{eq:equirep}}{=}r^2 n \;\; .    
\end{equation*}

The equality holds if and only if the two vectors are linearly dependent i.e., when all the degrees are equal. But $D$ is not an equireplicate design by assumption so this cannot happen. Moreover, if we consider the average degree of a vertex in $L(\mathcal{D})$, we can now conclude that

\begin{equation*}
    \frac{\sum_{\{i,j\} \in D} \{d(v_i) + d(v_j)\}}{|D^{\dagger}|} >  \frac{r^2 n}{|D^{\dagger}|} \overset{\eqref{eq:equirep}}{=} 2\;r \;\;.
\end{equation*} 

But this means that there exists at least a $e=\{v_i,v_j\} \in D$ such that $d(v_i) + d(v_j) > 2r$. If this was not the case, then the average degree could not be strictly greater than $2\;r$. This concludes the proof as the existence of two vertices $v_i$ and $v_j$ that meet the previous condition implies $\Delta(L(\mathcal{D}^{\dagger})) < \Delta(L(\mathcal{D}))$. 
\end{proof}

\begin{lemma}\label{lemma:equi_r_vs_det}
    Let $\mathcal{D} = (V, D)$ be the hypergraph of a non-equireplicate deterministic design $D$ and $\mathcal{D}^{\dagger} = (V, D^{\dagger})$  be the hypergraph of an $r$-equireplicate design $D^{\dagger}$. Assume that both designs have the same block size $k$ and cardinality i.e., $|D| = |D^{\dagger}|$. Then, 
    $$\Delta(\mathcal{D}) > \Delta(\mathcal{D}^{\dagger}) = r \;\;.$$
\end{lemma}

\noindent \textit{Proof sketch:} the proof, provided below, leverages the extension of the classical \emph{handshaking lemma}, discussed in Section \ref{subsec:hyper}, to $k$-uniform hypergraphs. Indeed, assuming that both designs have the same block size $k$ and cardinality, by equations \eqref{eq:equirep} and \eqref{eq:deter}, implies that $\Bar{d}(\mathcal{D}) = r$. But then, intuitively, there must be an index which appears strictly more than $r$ times in the non-equireplicate deterministic design, therefore forcing $\Delta(\mathcal{D}) > r$.

\begin{proof}
    Equation (\ref{eq:equirep}) is an extension of the classical \emph{handshaking lemma}—that relates the number of edges in a graph with the sum of the degrees—for $k$-uniform, $r$-regular hypergraphs. However, it can also be stated for a generic $k$-uniform hypergraph, as done in equation \eqref{eq:deter}:

\begin{equation*}
|D| \; k = n \; \Bar{d}\;,   
\end{equation*}

where $\Bar{d}$ is the average degree of the hypergraph of the design $\mathcal{D} = (V,D)$. This because $\mathcal{D}$ is $k$-uniform and each hyperedge contributes exactly $k$ incidences, one for each vertex it contains. Thus, $\sum_{v \in V} d(v)=|D|k$ and the result follows dividing both sides by $n$. Clearly, for $r$-regular hypergraphs, whose edge set is an $r$-equireplicate design, $\Bar{d} = r$. Thus, since we assumed that $|D| = |D^{\dagger}|$ and that the block size $k$ is the same, both equation (\ref{eq:handshaking_kuni}) and equation (\ref{eq:equirep}) hold true and we can conclude that the average degree of the hypergraph of the deterministic design $D$ must be equal to $r$. But then, there exist at least a vertex $v \in V$ whose degree is greater than $r$ because $D$ is not equireplicate. If this was not the case, and $d(v) \leq r$ for all $v \in V$ of $\mathcal{D}$ then $\Bar{d} \neq r$ violating equation (\ref{eq:handshaking_kuni}) because equation (\ref{eq:equirep}) must hold and $\mathcal{D}$ is not $r$-regular by assumption. On the other hand, since $\mathcal{D}^{\dagger}$ is an $r$-regular hypergraph, we have that $\Delta(\mathcal{D}^{\dagger}) = r$. Therefore, there exists at least a $v \in V$ of $\mathcal{D}$ such that $d(v) > r$ which implies \emph{a fortiori} that $\Delta(\mathcal{D}) > \Delta(\mathcal{D}^{\dagger}) = r$.         
\end{proof}

\subsection{Statement and proof of Theorem \ref{theo:BE_bound_2004}}

\begin{theorem}[Berry-Esseen for Deterministic Designs, with weaker moment conditions]\label{theo:BE_bound_2004}

Let $\left\{ h(S), \;S \in D \right\}$ be random variables indexed by the vertices of their dependency graph $L(\mathcal{D})$, with $D$ being a non-empty deterministic design. Assume that $\sigma^{2}_{k} >0$ and that there exists $2<p \leq 3$ such that $\|h(S)-\mu_k\|_p < \infty$. Then,

\small{
\begin{equation*}
\sup _z\left|P\left(\frac{U_{n,D}^{(k)} - \mu_{k}}{\sqrt{\operatorname{Var}U_{n,D}^{(k)}}} \leq z\right)-\Phi(z)\right| \leq 75 \; \{k \;(\Delta(\mathcal{D})-1) + 1\}^{5(p-1)}\; |D|^{1- \frac{p}{2}}\; \left(\frac{\|h(S)-\mu_k\|_p}{\sigma_k} \right)^{p} \;\;. 
\end{equation*}
}
\end{theorem}
\normalsize

\begin{proof}

We start by substituting equations \eqref{eq:IUk_red} and \eqref{eq:var_ustat} in the expression of the centered and standardized incomplete U-statistics of order $k$ to obtain the following result:

\begin{equation*}
\begin{aligned}
    \frac{U_{n,D}^{(k)} - \mu_{k}}{\sqrt{\operatorname{Var}U_{n,D}^{(k)}}} &= \frac{|D|^{-1}\left(\sum_{S \in D} h(S) - |D| \;\mu_k\right)}{\sqrt{|D|^{-2} \sum_{c=0}^k f_c \;\sigma_c^2}} \\
     &= \sum_{S \in D} \frac{h(S) - \mu_k}{\sqrt{\sum_{c=0}^k f_c \;\sigma_c^2}} = \sum_{S \in D} Y_S \;,
\end{aligned}
\end{equation*}

  where we denote by $Y_S$, the centered and rescaled version of $h(S)$ for all $S \in D$. We assumed that $\sigma^2_k > 0$ and $\sigma^2_k < \infty$ since we also assumed that $\|h(S)-\mu_k\|_p < \infty$ for $2<p \leq 3$. This implies, by inequality~\eqref{eq:ineq},  and identity $\sum_{c=0}^k f_c = |D|^{2}$, that $0<\sqrt{\sum_{c=0}^k f_c \;\sigma_c^2} < \infty$ for a fixed design size. Moreover, still by the fact that there exists $2<p \leq 3$ such that $\|h(S)-\mu_k\|_p < \infty$, we can conclude that:

    \begin{equation}\label{eq:m_bound}
     E\left[\left|Y_S\right|^p \right] = E\left[\left|\frac{h(S) - \mu_k}{\sqrt{\sum_{c=0}^k f_c \;\sigma_c^2}}\right|^p \right] = \left(\frac{\|h(S)-\mu_k\|_p}{\sqrt{\sum_{c=0}^k f_c \;\sigma_c^2}}\right)^{p} < \infty \;\; \;.
    \end{equation}

At this point, we note that $L(\mathcal{D})$ is the dependency graph of the set of random variables $\{Y_S, \;S \in D\}$. This is because we just subtracted and divided by the same constants all the random variables in the set $\{h(S), \;S \in D\}$, without changing their dependency structure. Consequently, Proposition \ref{prop:dep_graph} holds as well for the set $\{Y_S, \;S \in D\}$. Now, we have that:

\begin{enumerate}[label=\roman*)]
    \item $\{Y_S, \;S \in D\}$ are random variables indexed by the vertices of $L(\mathcal{D})$,
    \item by construction, $E[Y_S] = 0$ for all $S \in D$, 
    \item $E\left[\left(\sum_{S \in D}Y_S\right)^{2}\right] = 1$ and
    \item $E\left[\left|Y_S\right|^p \right] = \theta^{p} < \infty$ by \eqref{eq:m_bound} for all $S \in D$ because $Y_S$ are identically distributed, with $\theta = \frac{\|h(S)-\mu_k\|_p}{\sqrt{\sum_{c=0}^k f_c \;\sigma_c^2}} > 0$.
\end{enumerate}

Therefore, all the conditions of \emph{Theorem 2.7} in \cite{chen2004normal} are satisfied, and we can conclude that:

\begin{equation*}
\sup _z\left|P\left(\sum_{S \in D} Y_S \leq z\right)-\Phi(z)\right| \leq 75 \; \Delta(L(\mathcal{D}))^{5(p-1)}\; \;|D| \; \left(\frac{\|h(S)-\mu_k\|_p}{\sqrt{\sum_{c=0}^k f_c \;\sigma_c^2}}\right)^{p} \;\;\;. 
\end{equation*}
 
At this point, we make the following substitutions in the previous expression: (a) the tight upper bound of Lemma \ref{lemma:maxdeg} instead of $\Delta(L(\mathcal{D}))$, and (b) $|D|^{1/2}\; \sigma_k$ instead of $\sqrt{\sum_{c=0}^k f_c\; \sigma_c^2}$ because the former, by Lemma \ref{lemma:ord_var}, it is a tight lower bound for the latter. All these substitutions preserve the direction of the inequality. After some rearrangements, we finally obtain that:

\begin{equation*}
\sup _z\left|P\left(\sum_{S \in D} Y_S  \leq z\right)-\Phi(z)\right| \leq 75 \; [k \;(\Delta(\mathcal{D})-1) + 1]^{5(p-1)}\;  |D|^{1- \frac{p}{2}}\; \left(\frac{\|h(S)-\mu_k\|_p}{\sigma_k} \right)^{p} \;\;,
\end{equation*}

which ends the proof since we showed previously that $\sum_{S \in D} Y_S = \frac{U_{n,D}^{(k)} - \mu_{k}}{\sqrt{\operatorname{Var}U_{n,D}^{(k)}}}$.\\
As explained at the end of the proof of Theorem \ref{theo:BE_bound}, the previous bound can be made sharper if one is willing to further assume that both $f_c > 0$ and $\sigma^2_c > 0$ for at least one $c \in \{1, \ldots, k-1\}$.
\end{proof}

\subsection{Statement and proof of Theorem \ref{theo:BE_bounded_kern}}

\begin{theorem}\label{theo:BE_bounded_kern}\textnormal{(Berry-Esseen for Bounded Kernels).} Let $\left\{ h(S), \;S \in D \right\}$ be random variables indexed by the vertices of their dependency graph $L(\mathcal{D})$, with $D$ being a non-empty deterministic design. Assume that $\sigma^{2}_{k}>0$ and that $\|h(S)-\mu_k\|_{\infty} < \infty$. Then,

\small{
\begin{equation*}
\sup _z\left|P\left(\frac{U_{n,D}^{(k)} - \mu_{k}}{\sqrt{\operatorname{Var}U_{n,D}^{(k)}}} \leq z\right)-\Phi(z)\right| \leq 32 \;(1 +\sqrt{6}) \; \frac{k \;(\Delta(\mathcal{D})-1) + 1}{|D|^{1/4}}\;\; \left(\frac{\|h(S)-\mu_k\|_{\infty}}{\sigma_k} \right)^{\frac{3}{2}}\;. 
\end{equation*}}

\noindent \normalsize{Moreover, if $D$ is also an $r$-equireplicate design, we can conclude that }

\small{
\begin{equation*}
\sup _z\left|P\left(\frac{U_{n,D}^{(k)} - \mu_{k}}{\sqrt{\operatorname{Var}U_{n,D}^{(k)}}} \leq z\right)-\Phi(z)\right| \leq 32 \;(1 +\sqrt{6}) \; \{k \;(r-1) + 1\}\; \left(\frac{k}{n\;r} \right)^{\frac{1}{4}}\; \left(\frac{\|h(S)-\mu_k\|_{\infty}}{\sigma_k} \right)^{\frac{3}{2}}\;. 
\end{equation*}
}
\end{theorem}

\begin{proof}
    By Proposition \ref{prop:dep_graph}, $\{h(S), \;S \in D\}$ are random variables indexed by the vertices of $L(\mathcal{D})$. Set $W = \sum_{S \in D} h(S)$ and note that we assumed $\sigma^{2}_{k}>0$ as well as $\|h(S)-\mu_k\|_{\infty} < \infty$, which implies that there exists $\theta>0$ such that $\left|h(S)\right| \leq \theta$ a.s. for all $S \in D$. Thus, all the conditions of \emph{Corollary 2} in \cite{baldi1989normal} are satisfied and we can conclude that:

\begin{equation*}
\sup _z\left|P\left(\frac{W - E[W] }{\sqrt{\operatorname{Var}W}} \leq z\right)-\Phi(z)\right| \leq 32 \;(1 +\sqrt{6}) \; \sqrt{|D|} \; \Delta(L(\mathcal{D})) \left(\frac{\|h(S)-\mu_k\|_{\infty}}{\sqrt{\operatorname{Var}W}} \right)^{\frac{3}{2}}\;. 
\end{equation*}

At this point, we make the following substitutions in the previous expression: (a) the tight upper bound of Lemma \ref{lemma:maxdeg} instead of $\Delta(L(\mathcal{D}))$, and (b) $\sqrt{|D|}\; \sigma_k$ instead of $\sqrt{\operatorname{Var}W}$ because the former, by Lemma \ref{lemma:ord_var}, it is a tight lower bound for the latter. All these substitutions preserve the direction of the inequality. After some rearrangements, we finally obtain that:

\small{
\begin{equation*}
\sup _z\left|P\left(\frac{W - E[W] }{\sqrt{\operatorname{Var}W}} \leq z\right)-\Phi(z)\right|  \leq 32 \;(1 +\sqrt{6}) \; \frac{k \;(\Delta(\mathcal{D})-1) + 1}{|D|^{1/4}}\;\; \left(\frac{\|h(S)-\mu_k\|_{\infty}}{\sigma_k} \right)^{\frac{3}{2}}\;, 
\end{equation*}
}

\normalsize{which ends the first part of the proof since $\frac{W - E[W] }{\sqrt{\operatorname{Var}W}} = \frac{U_{n,D}^{(k)} - \mu_{k}}{\sqrt{\operatorname{Var}U_{n,D}^{(k)}}}$. As already underlined in the proof of Theorem \ref{theo:BE_bound}}, the bounds of Theorem \ref{theo:BE_bounded_kern} can be further sharpened outside the full degeneracy scenario we are considering.
For the second part of the proof, the result follows by substituting $\Delta(\mathcal{D}) = \bar{d}(\mathcal{D}) = r$ in the previous bound because the design $D$ is further assumed to be $r$-equireplicate. Thus, we finally obtain that

\small{
\begin{equation*}
\sup _z\left|P\left(\frac{U_{n,D}^{(k)} - \mu_{k}}{\sqrt{\operatorname{Var}U_{n,D}^{(k)}}} \leq z\right)-\Phi(z)\right| \leq 32 \;(1 +\sqrt{6}) \; \{k \;(r-1) + 1\}\; \left(\frac{k}{n\;r} \right)^{\frac{1}{4}}\; \left(\frac{\|h(S)-\mu_k\|_{\infty}}{\sigma_k} \right)^{\frac{3}{2}}\;. 
\end{equation*}
}

\normalsize{When the order $k$ is fixed (i.e., in the \emph{finite-order regime}), it is easy to notice that if $r = o(n^{1/3})$—or equivalently $|D| = o(n^{4/3})$—}then the bound converges to zero, ensuring asymptotic normality. Unlike Theorem \ref{theo:BE_bound} in the main text, these results provide explicit constants, at the price of assuming a bounded kernel. Compared with Theorem~\ref{theo:BE_bound_2004}, the dependence on the maximum degree improves from $\Delta(L(\mathcal D))^{5(p-1)}$ to $\Delta(L(\mathcal D))$. 
\end{proof}

\subsection{Variance of incomplete U-statistics of deterministic designs}\label{app:var_Ustat}

The dependency graph of an incomplete U-statistics i.e.,  $L(\mathcal{D})$, encodes the presence or absence of dependence relationships between pairs of random variables in $\left\{ h(S), \;S \in D \right\}$. Clearly, linear types of dependencies are also represented in $L(\mathcal{D})$. Thus, for example, if there is no edge between two random variables $h(S_1)$ and $h(S_2)$, with $S_1,S_2 \in D$, then $\operatorname{Cov}[h((S_1),h(S_2)] = 0$. Obviously, the opposite direction does not hold, as the absence of a linear dependence does not imply independence. Following this line of reasoning, it is straightforward to conclude that an upper bound on $\Delta(L(\mathcal{D}))$ is also an upper bound on the number of covariance terms in $\operatorname{Var}U_{n,D}^{(k)}$. However, a U-statistic can be degenerate, which implies that $\operatorname{Cov}[h((S_1),h(S_2)] = 0$ even if an edge connects $h(S_1)$ and $h(S_1)$ in $L(\mathcal{D})$. Thus, Lemma \ref{lemma:maxdeg} allows us to obtain only an upper bound for $\operatorname{Var}U_{n,D}^{(k)}$. The lower bound represents a scenario of \emph{full degeneracy}, where $\sigma^2_{k-1} = 0$ which, by inequality (\ref{eq:ineq}), implies that $\operatorname{Cov}[h((S_1),h(S_2)] = 0$ for all $S_1,S_2 \in D$, with $S_1 \neq S_2$. In this situation, excluding the case in which $\operatorname{Var}U_{n,D}^{(k)} = 0$, only $\operatorname{Var}[h(S)] = \sigma^{2}_{k} > 0$ and the number of these variance terms corresponds to the number of vertices of $L(\mathcal{D})$. The next Lemma formalizes these results.    

\begin{lemma}\label{lemma:ord_var}
    Let $\operatorname{Var}[h(S)] = \sigma^{2}_{k} < \infty$ and strictly positive, then $\operatorname{Var} U^{(k)}_{n,D}$, the variance of any incomplete U-statistic of order $k$ of a deterministic design $D$, is lower and upper bounded by:
    $$\frac{\sigma^{2}_{k}}{|D|} \leq \;\operatorname{Var} U^{(k)}_{n,D} \; < \frac{(k \;\Delta(\mathcal{D})-k + 1) \;\sigma^{2}_{k}}{|D|} $$
\end{lemma}

In the proof provided in the next paragraph, we begin by stacking the random variables in the set $\left\{ h(S), \;S \in D \right\}$ to form the vector $\boldsymbol{h(S)}$. We then leverage the interpretation of $A_{L(\mathcal{D})}$—the adjacency matrix of the line graph $L(\mathcal{D})$ (see Section \ref{app:adj_mat})—as an unweighted analogue of the variance-covariance matrix $\Sigma = \operatorname{Var}[ \boldsymbol{h(S)}]$. This correspondence allows us to apply the upper bound established in Lemma~\ref{lemma:maxdeg}. The upper bound is
relatively loose for deterministic designs that are far from being equireplicate. However, the bound becomes tighter as the gap between the maximum degree $\Delta(\mathcal{D})$ and the average degree $\Bar{d}(\mathcal{D})$ decreases, for a fixed design size. Moreover, as long as $k = o(\sqrt{n})$, the bound is asymptotically tight for designs $D^{\diamond}$ that are both equireplicate and linear. When $k \asymp \sqrt{n} $, the bound stays asymptotically tight for this class of designs, but now up to a constant (see the proof below for further details).

 To conclude, we underline that the results derived in Lemma \ref{lemma:ord_var} and \ref{lemma:var_lin_hyper} encompass both \emph{degenerate} and \emph{non-degenerate} cases of incomplete U-statistics based on deterministic designs. Moreover, they imply that $\operatorname{Var}U_{n,D}^{(k)}$ is $O(\frac{1}{n})$ when $k$ is fixed and $\Delta(\mathcal{D})$ and $\Bar{d}(\mathcal{D})$ grow at the same rate. This is akin to the variance of a classical mean estimator, even if the random variables in the set $\left\{ h(S), \;S \in D \right\}$ are dependent. 

\subsection{Proof of Lemma \ref{lemma:ord_var}}

\begin{proof}
We consider the set of random variables $\left\{ h(S), \;S \in D \right\}$ and stack them into a vector $\boldsymbol{h(S)}$, which is of size $|D|$. We call $\Sigma = \operatorname{Var}[ \boldsymbol{h(S)}]$ its variance-covariance matrix, which is of size $|D|\times |D|$, and such that each element $\sigma_{i,j} = \operatorname{Cov}\left(h\left(S_i\right), h\left(S_j\right)\right)$ with $S_i, S_j \in D$. For all $i,j \in \{1,.., |D|\}$, the intersection $|S_i \cap S_j | \in \{0, \ldots, k \}$ by definition (see subsection \ref{subsec:ustat}). Therefore, each $\sigma_{i,j}$ must be equal to a particular $\sigma^{2}_c$ with $c \in \{0, \ldots, k \}$. This also implies that each $\sigma_{i,j} \geq 0$ since each $\sigma^{2}_c \geq 0$. However, note that when $|S_i \cap S_j | = 0$ then $\sigma_{i,j} = \sigma^{2}_0 = 0$ because the random variables $h(S_i)$ and $h(S_j)$ do not share an index and thus are independent. Now, given that $\left\{ h(S), \;S \in D \right\}$ is the set of random variables whose average defines $U_{n,D}^{(k)}$, we can clearly express the variance as the rescaled sum of the elements of the variance-covariance matrix:

\begin{equation}
    \operatorname{Var} U_{n,D}^{(k)} = |D|^{-2}  \sum_{i=1}^{|D|} \sum_{j=1}^{|D|} \ind{\{|S_i \cap S_j | \neq 0\}} \; \sigma_{i,j} \;\;, \label{eq:mat_varUstat}
\end{equation}

where $\ind{\{|S_i \cap S_j | \neq 0\}}$ is an indicator function that is $1$ when $|S_i \cap S_j | \neq 0$ and $0$ otherwise. This expression is an alternative to the standard equation (\ref{eq:var_ustat}) as it does not explicitly consider the grouping with respect to the $f_c$ cardinalities. At this point, obtaining a lower bound on $\operatorname{Var} U_{n,D}^{(k)}$ is straightforward. This because each $\sigma_{i,j} \geq 0$ and they are all equal to zero if and only if the U-statistics is degenerate of order $k$. But we assumed that $\sigma^{2}_{k} > 0$ so at worst the U-statistics can be degenerate of order $k-1$, which implies that $0 = \sigma_1^2 = \ldots = \sigma_{k-1}^2$. In this scenario, if $\sigma_{i,j} > 0$ then $\sigma_{i,j} = \sigma^{2}_{k}$. Moreover, the indicator is always one in this case since $|S_i \cap S_j | = k$. Then, there are exactly $|D|$ of these quantities because they can only appear on the principal diagonal of $\Sigma$ since $\operatorname{Var}[h(S)] = \sigma^{2}_{k}$. Thus, we can conclude that: $$\frac{\sigma^{2}_{k}}{|D|}\;  \leq \;\operatorname{Var} U^{(k)}_{n,D}\;\;.$$
In contrast, to maximize the variance we need to consider the non-degenerate case where all $\sigma^2_c > 0$ for $c \in \{1, \ldots, k \}$. To obtain an upper bound, we start by observing that, by inequality (\ref{eq:ineq}), $\sigma^2_k > \sigma^2_c$ for $c \in \{1, \ldots, k-1 \}$ so that

\begin{equation*}
 \sum_{i=1}^{|D|} \sum_{j=1}^{|D|} \ind{\{|S_i \cap S_j | \neq 0\}} \; \sigma_{i,j} < \; \sigma^2_k\;  \sum_{i=1}^{|D|} \sum_{j=1}^{|D|}  \ind{\{|S_i \cap S_j | \neq 0\}} \;\;. 
\end{equation*}

Then, note that $ \sum_{i=1}^{|D|} \sum_{j=1}^{|D|}  \ind{\{|S_i \cap S_j | \neq 0\}}$ is just summing all the elements of a $\{0,1\}^{|D|\times |D|}$ matrix that indicates whether any given two blocks of the design $D$ share at least one index or not. But this is exactly the definition of $A_{L(\mathcal{D})}$, the adjacency matrix of $L(\mathcal{D})$ which is the line graph of the hypergraph $\mathcal{D}$ (see section \ref{subsec:hyper}). At this point, by Lemma \ref{lemma:maxdeg}, we can conclude that the sum of all the elements of any given row or column of $A_{L(\mathcal{D})}$, which represents the degree of a given vertex of $L(\mathcal{D})$, is upper bounded by $[k \;(\Delta(\mathcal{D})-1) + 1]$. Due to the fact that there are a total of $|D|$ columns or lines, this implies that:

\begin{equation*}
\sigma^2_k\;  \sum_{i=1}^{|D|} \sum_{j=1}^{|D|}  \ind{\{|S_i \cap S_j | \neq 0\}} \leq \sigma^2_k\;  \sum_{i=1}^{|D|} \{k \;(\Delta(\mathcal{D})-1) + 1\} = \sigma^2_k\; |D| \;\{k \;(\Delta(\mathcal{D})-1) + 1\} . 
\end{equation*}

Thus, we can conclude that 

 $$ \operatorname{Var} U^{(k)}_{n,D} \; < \frac{(k \;\Delta(\mathcal{D})-k + 1) \;\sigma^{2}_{k}}{|D|} \;. $$  
This ends the proof. In the case of $r$-equireplicate designs, the maximum and average degrees coincide, i.e., $\Delta(\mathcal{D}) = \Bar{d}(\mathcal{D}) = r$. If a design $D^{\diamond}$ is also linear, the difference between the general upper bound for the variance and the closed-form expression of  $\operatorname{Var} U_{n,D^{\diamond}}^{(k)}$, see Lemma \ref{lemma:var_lin_hyper}, is

$$
\frac{k^2 \;(r - 1) \;(\sigma_k^2 - \sigma_1^2)}{n \;r}.
$$

Thus, as long as $k = o(\sqrt{n})$, the previously established upper bound is asymptotically tight. When $k \asymp \sqrt{n} $, the bound is asymptotically tight up to a constant. 
\end{proof}

\subsection{Statement and proof of Lemma \ref{lemma:var_lin_hyper}}

\begin{lemma}\label{lemma:var_lin_hyper}
Let $D^{\diamond}$ be an $r$-equireplicate and linear design. Then the variance of any incomplete U-statistic of order $k$ based on $D^{\diamond}$ is
$$
\operatorname{Var}U_{n,D^{\diamond}}^{(k)} = \frac{k^2\;(r - 1)\;\sigma_1^2 + k\; \sigma_k^2}{n \;r},
$$

\noindent and this variance is minimal among all incomplete U-statistics with the same design size $|D^{\diamond}|$. 

\end{lemma}

\begin{proof}
We follow the same passages outlined in the proof of Lemma \ref{lemma:ord_var}, up to obtaining equation (\ref{eq:mat_varUstat}) which we reproduce below for $|D| = |D^{\diamond}|$:

\begin{equation*}
    \operatorname{Var} U_{n,D^{\diamond}}^{(k)} = |D^{\diamond}|^{-2}  \sum_{i=1}^{|D^{\diamond}|} \sum_{j=1}^{|D^{\diamond}|} \ind{\{|S_i \cap S_j | \neq 0\}} \; \sigma_{i,j} \;\;.
\end{equation*}

Now, because the hypergraph is linear, we can completely characterize the values of the $\sigma_{i,j}$. Either $|S_i \cap S_j | = 0$ which implies $\sigma_{i,j} = \sigma^2_{0} = 0$, $|S_i \cap S_j | = 1$ which implies $\sigma_{i,j} = \sigma^2_1$ or $S_i = S_j$ which implies $|S_i \cap S_j | = k$ and thus $\sigma_{i,j} = \sigma^2_k$. Then, as already explained in the proof of Lemma \ref{lemma:ord_var}, $\sum_{i=1}^{|D^{\diamond}|} \sum_{j=1}^{|D^{\diamond}|}  \ind{\{|S_i \cap S_j | \neq 0\}}$ is summing the elements of $A_{L(\mathcal{D}^{\diamond})}$, the adjacency matrix of $L(\mathcal{D}^{\diamond})$ which is the line graph of the hypergraph $\mathcal{D}^{\diamond}$. Now, without loss of generality, consider a particular row of $A_{L(\mathcal{D}^{\diamond})}$. A value of $1$ to a given element of that row can be given either if $|S_i \cap S_j | = 1$ which implies $\sigma_{i,j} = \sigma^2_1$ or $S_i = S_j$ which implies $\sigma_{i,j} = \sigma^2_k$ as already discussed. However, $S_i = S_j$ can happen only once for each row because it represents an element on the principal diagonal of $L(\mathcal{D}^{\diamond})$. Moreover, since the hypergraph is $r$-regular and linear by assumption, we know by Lemma \ref{lemma:maxdeg} (see the part of the proof that shows when the upper bound is attained) that $L(\mathcal{D}^{\diamond})$ must also be $k\;(r-1)+1$-regular. Therefore, the sum of all the elements of any given row or column of $A_{L(\mathcal{D}^{\diamond})}$, which represents the degree of a given vertex of $L(\mathcal{D}^{\diamond})$, must be equal to $[k \;(r-1) + 1]$. Then, we can uniquely conclude that the number of $\sigma^2_1$ in each row must be $k \;(r-1)$ and clearly, as explained before, there is only one $\sigma^2_k$. Due to the fact that $L(\mathcal{D}^{\diamond})$ has a total of $|D^{\diamond}| = \frac{nr}{k}$ rows, this implies that:

$$
\operatorname{Var}U_{n,D^{\diamond}}^{(k)} = \frac{k \left\{k\;(r - 1)\sigma_1^2 +\sigma_k^2\right\}}{n \;r}.
$$

To finish the proof, note that by definition of a linear hypergraph, the intersection size of any two distinct blocks of $D^{\diamond}$ is at most one. This means that the off-diagonal elements of the matrix $NN^{T}$, where $N$ is the incidence matrix of $D^{\diamond}$, are either zero or one. If this was not the case, then the linearity condition would be violated since an off-diagonal element of $NN^{T}$ displays how many blocks of $D^{\diamond}$ contain the pair of indices $\{i,j\}$, with $i,j \in V$. Thus, since $D^{\diamond}$ is also balanced—i.e., equireplicate—by assumption, Corollary 1 of Theorem 2 of \cite{lee1990u}, page 197, applies and $D^{\diamond}$ is a minimum variance design. Therefore, the variance of the incomplete U-statistic $U_{n,D^{\diamond}}^{(k)}$ induced by $D^{\diamond}$ is minimal among all incomplete U-statistics having the same design size.
\end{proof}

\subsection{Statement and proof of Theorem \ref{theo:BE_equi_lin_nondeg}}

\begin{theorem}
\label{theo:BE_equi_lin_nondeg}\textnormal{(Berry-Esseen for Equireplicate and Linear Designs, non-degenerate case).}
Let $\left\{ h(S), \;S \in D^{\diamond} \right\}$ be random variables indexed by the vertices of their dependency graph $L(\mathcal{D}^{\diamond})$, with $D^{\diamond}$ being a non-empty $r$-equireplicate and linear design. Moreover, assume that $\sigma^{2}_{1} >0$ and that $\|h(S)-\mu_k\|_4 < \infty$. Then,

\begin{equation*}
\sup _z\left|P\left(\frac{U_{n,D^{\diamond}}^{(k)} - \mu_{k}}{\sqrt{\operatorname{Var}U_{n,D^{\diamond}}^{(k)}}} \leq z\right)-\Phi(z)\right| \leq C\,\left(
\frac{k\;(r-1)+1}{|D^{\diamond}|}\right)^{1/2}
\left(
    \frac{\|h(S)-\mu_k\|_4}{\sigma_1}
\right)^{3}\;\;. 
\end{equation*}
\end{theorem}

\begin{proof}
Consider $D^{\diamond}$ to be a non-empty $r$-equireplicate and linear design. We denote by $T_S$, the centered version of $h(S)$ for all $S \in D^{\diamond}$. Since we just subtracted the same constant for all the random variables, the dependence structure is unchanged and thus Proposition \ref{prop:dep_graph} holds. Now, we have that: i) $\{T_S, \;S \in D^{\diamond}\}$ are random variables indexed by the vertices of $L(\mathcal{D^{\diamond}})$, and ii) by construction, $E[T_S] = 0$ for all $S \in D^{\diamond}$. Therefore, \emph{Theorem 3.1} in \cite{cai2026refined} applies, and we can conclude that:  

\begin{equation*}
\begin{aligned}
    \sup_{z \in \mathbb{R}}\left|P\left(\frac{\sum_{S \in D^{\diamond}} T_S}{\sigma_T} \leq z\right)-\Phi(z)\right| &\leq C_1 \; [k\;(r-1)+1]^{2}\;|D^{\diamond}| \; \left(\frac{\|T_S\|_4}{\sigma_T}\right)^{3} \\
    &+ C_2 \; [k\;(r-1)+1]^{3/2}\;\sqrt{|D^{\diamond}|} \; \left(\frac{\|T_S\|_4}{\sigma_T}\right)^{2} \;\;\;,  
\end{aligned}
\end{equation*}

where $\sigma_T = \sqrt{\operatorname{Var}(\sum_{S \in D^{\diamond}} T_S)}$, $C_1,C_2 >0$ are absolute constants and we used the facts that $\{T_S, \;S \in D^{\diamond}\}$ are identically distributed (which implies e.g., $\sum_{S \in D^{\diamond}} \|T_S\|^{3}_4 = |D^{\diamond}| \; \|T_S\|^3_4$) and that $\Delta(L(\mathcal{D^{\diamond}})) = k\;(r-1)+1$ by Lemma \ref{lemma:maxdeg}. Moreover, by Lemma \ref{lemma:var_lin_hyper} $\sigma^2_T = |D^{\diamond}| \; \{k\;(r-1) \sigma^2_1 + \sigma^2_k\;\}$. Thus $\sigma^2_T$ is strictly positive, since $ \sigma^2_1 > 0$ by assumption together with $\sigma^2_1 \leq \sigma^2_k$ by \eqref{eq:ineq}, and bounded, because we assumed $\|h(S)-\mu_k\|_4 < \infty$. Now, we use $\sqrt{|D^{\diamond}| \; [k\;(r-1)+1]} \;\sigma_1$ to lower bound $\sigma_T$ since $\sigma_1 \leq \sigma_k$ again by \eqref{eq:ineq}. After some simplifications, we obtain that:

\small{
\begin{equation*}
\begin{aligned}
    \sup_{z \in \mathbb{R}}\left|P\left(\frac{\sum_{S \in D^{\diamond}} T_S}{\sigma_T} \leq z\right)-\Phi(z)\right| &\leq \max(C_1,C_2) \; \left(\frac{k\;(r-1)+1}{|D^{\diamond}|}\right)^{1/2}\;\left[ \left(\frac{\|T_S\|_4}{\sigma_1}\right)^{3} + \left(\frac{\|T_S\|_4}{\sigma_1}\right)^{2} \right]   \\
    &\leq 2\max(C_1,C_2)\; \left(\frac{k\;(r-1)+1}{|D^{\diamond}|}\right)^{1/2} \; \left(\frac{\|T_S\|_4}{\sigma_1}\right)^{3} \;\;, 
\end{aligned}
\end{equation*}
} \normalsize

where the last passage follows from the fact that $\|T_S\|_4 \geq \|T_S\|_2 = \sigma_k \geq \sigma_1$, thus meaning that the first term inside the square brackets upper bounds the second. This ends the proof since we can show, by following the same passages outlined at the beginning of the proof of Theorem \ref{theo:BE_bound}, that $\frac{\sum_{S \in D^{\diamond}} T_S}{\sigma_T} = \frac{U_{n,D^{\diamond}}^{(k)} - \mu_{k}}{\sqrt{\operatorname{Var}U_{n,D^{\diamond}}^{(k)}}}$. 
\end{proof}

\subsection{Infinite-order U-statistics}\label{subsect:ious}

Recently, infinite-order U-statistics (IOUS) have attracted growing attention in the statistics and machine learning communities, particularly due to their relevance in uncertainty quantification for supervised learning ensembles (see e.g., \cite{mentch2016quantifying, peng2022rates}). However, the main analytical tool in this domain i.e., the Hoeffding decomposition, becomes problematic in the IOUS regime. Specifically, when assuming $\sigma^2_k < \infty$, the variance of the first-order term in the decomposition—known as the \emph{Hájek projection}—may vanish as $k \to \infty$, significantly complicating the analysis. This issue arises from inequality~\eqref{eq:ineq}, which implies $k\sigma^2_1 < \sigma^2_k$. Hence, if $\sigma^2_k$ is bounded and $k$ diverges, we necessarily have $\sigma^2_1 = O(k^{-1})$. As a result, the variance of the Hájek projection may shrink to zero, and controlling the moments of an increasing number of degenerate terms becomes challenging, as noted by \cite{SongChenKato2019}.

Our approach avoids these complications as it does not rely on the Hoeffding decomposition and instead operates directly in the space induced by the kernel function $h$. Consequently, it does not need to impose assumptions on the order of the $\sigma^2_c$ terms in equation \eqref{eq:var_ustat}. Indeed, our dependency-graph does not distinguish dependencies according to their Hoeffding order, it captures all dependence induced by overlapping blocks through the same line graph structure. Hence, although degeneracy still affects the variance and normalization, it does not require a different dependence analysis as $k$ grows, allowing the same framework to cover both fixed and infinite-order incomplete U-statistics.

\begin{remark}\label{remark:var_inford}
    To apply our results in the IO regime under full degeneracy, we need a consistent estimator of the variance when $k$ diverges. Prior works have proposed variance estimation methods for incomplete U-statistics in the IO regime, which can be applied \citep{wang2014variance,SongChenKato2019,xu_var_est_rf_2024}. However, determining specific details for implementation of our framework in this regime is left as future work.
\end{remark}

\subsection{Proof of Theorem \ref{theo:gen_stud}}

\begin{proof}
Fix $0<\eta<1/2$ and consider the event
$
\mathcal E_\eta
=
\left\{
\left|
\frac{\widehat{\tau}}{\tau}-1
\right|
\leq \eta
\right\}
$.
On $\mathcal E_\eta$,
\[
1-\eta
\leq
\frac{\widehat{\tau}}{\tau}
\leq
1+\eta.
\]
Therefore, for every $z\in\mathbb R$,
$$
z_-(\eta)
\leq
z\sqrt{\frac{\widehat{\tau}}{\tau}}
\leq
z_+(\eta),
$$
where
$
z_-(\eta)
=
\min\{z\sqrt{1-\eta},z\sqrt{1+\eta}\}
$
and
$
z_+(\eta)
=
\max\{z\sqrt{1-\eta},z\sqrt{1+\eta}\}.
$

Hence,
\begin{align*}
P\left(
\frac{T-\theta}{\sqrt{\widehat{\tau}}}
\leq z
\right)
&=
P\left(
\frac{T-\theta}{\sqrt{\tau}}
\leq
z\sqrt{\frac{\widehat{\tau}}{\tau}}
\right) \\
&\leq
P\left(
\frac{T-\theta}{\sqrt{\tau}}
\leq z_+(\eta)
\right)
+
P(\mathcal E_\eta^c)\;.
\end{align*}
For the lower bound, note that on $\mathcal E_\eta$ the random cutoff in the studentized event is at least $z_-(\eta)$. Thus any realization for which the statistic normalized by the deterministic variance parameter is below $z_-(\eta)$ also belongs to the studentized event. This means that the event
$
\left\{
\frac{T-\theta}{\sqrt{\tau}}
\leq z_-(\eta)
\right\}
\cap \mathcal E_\eta
$
is contained in
$
\left\{
\frac{T-\theta}{\sqrt{\widehat{\tau}}}
\leq z
\right\}.
$
Therefore, we conclude that
\begin{align*}
P\left(
\frac{T-\theta}{\sqrt{\widehat{\tau}}}
\leq z
\right)
&\geq
P\left(
\frac{T-\theta}{\sqrt{\tau}}
\leq z_-(\eta)
\right)
-
P(\mathcal E_\eta^c).
\end{align*}

Using the assumed bound for the statistic normalized by the deterministic variance parameter, and subtracting $\Phi(z)$, the upper bound gives
\begin{align*}
P\left(
\frac{T-\theta}{\sqrt{\widehat{\tau}}}
\leq z
\right)-\Phi(z)
&\leq
\Phi(z_+(\eta))-\Phi(z)
+
B_T
+
P(\mathcal E_\eta^c),
\end{align*}
whereas the lower bound gives
\begin{align*}
\Phi(z)-
P\left(
\frac{T-\theta}{\sqrt{\widehat{\tau}}}
\leq z
\right)
&\leq
\Phi(z)-\Phi(z_-(\eta))
+
B_T
+
P(\mathcal E_\eta^c).
\end{align*}
Therefore,
\footnotesize
\begin{align*}
\left|
P\left(
\frac{T-\theta}{\sqrt{\widehat{\tau}}}
\leq z
\right)-\Phi(z)
\right|
&\leq
B_T
+
P(\mathcal E_\eta^c)
+
\max\left\{
\Phi(z_+(\eta))-\Phi(z),
\Phi(z)-\Phi(z_-(\eta))
\right\}.
\end{align*}
\normalsize
Since $z_-(\eta)\leq z\leq z_+(\eta)$ and $\Phi$ is nondecreasing, this is equivalently

\footnotesize
\begin{align*}
\left|
P\left(
\frac{T-\theta}{\sqrt{\widehat{\tau}}}
\leq z
\right)-\Phi(z)
\right|
&\leq
B_T
+
P(\mathcal E_\eta^c)
+
\max\left\{
|\Phi(z_+(\eta))-\Phi(z)|,
|\Phi(z_-(\eta))-\Phi(z)|
\right\}.
\end{align*}
\normalsize

It remains to control uniformly in $z$ the normal-cdf perturbation induced by replacing $z$ with
$z_-(\eta)$ and $z_+(\eta)$. We first prove that, for every $a\in[1/2,2]$,
$$
\sup_{z\in\mathbb R}|\Phi(az)-\Phi(z)|
\leq
\frac{2}{\sqrt{2\pi e}}|a-1|.
$$

Fix $z\in\mathbb R$ and consider the function
$f_z(a)=\Phi(az)$, $a\in[1/2,2]$. By the mean value theorem applied to $f_z$ between $1$ and $a$, there exists a point $b$ between $1$ and $a$ such that
$$
\Phi(az)-\Phi(z)
=
f_z(a)-f_z(1)
=
(a-1)f_z'(b)
=
(a-1)z\phi(bz).
$$
Taking absolute values gives
$
|\Phi(az)-\Phi(z)|
\leq
|a-1|\,|z|\phi(bz).
$
We now set $u=b|z|$, and consider that $b\in[1/2,2]$, to obtain
$$
|z|\phi(bz)
=
\frac{u}{b}\phi(u)
\leq
2u\phi(u)
\leq
\frac{2}{\sqrt{2\pi e}},
$$
where the last inequality follows from
$\sup_{u\geq0}u\phi(u)=1/\sqrt{2\pi e}$. Hence, we conclude that
$$
|\Phi(az)-\Phi(z)|
\leq
\frac{2}{\sqrt{2\pi e}}|a-1|.
$$
Since this bound does not depend on $z$, taking the supremum over $z\in\mathbb R$ gives
$$
\sup_{z\in\mathbb R}|\Phi(az)-\Phi(z)|
\leq
\frac{2}{\sqrt{2\pi e}}|a-1|,
\qquad a\in[1/2,2].
$$

We now apply the bound obtained above fixing
$
a_-=\sqrt{1-\eta}
$
and
$
a_+=\sqrt{1+\eta}.
$
Since $0<\eta<1/2$, both $a_-$ and $a_+$ belong to $[1/2,2]$. Moreover, by rationalizing,
$$
|a_+-1|
=
\sqrt{1+\eta}-1
=
\frac{\eta}{1+\sqrt{1+\eta}}
\leq
\eta,
$$
and
$$
|a_--1|
=
1-\sqrt{1-\eta}
=
\frac{\eta}{1+\sqrt{1-\eta}}
\leq
\eta.
$$
Now, observing that
$\{z_-(\eta),z_+(\eta)\}=\{a_-z,a_+z\}$, we can conclude that
\begin{align*}
\sup_{z\in\mathbb R}
\max\left\{
|\Phi(z_+(\eta))-\Phi(z)|,\,
|\Phi(z_-(\eta))-\Phi(z)|
\right\}
&\leq
\frac{2}{\sqrt{2\pi e}}\;\eta .
\end{align*}
Combining this with the previous studentized bound yields the desired expression
$$
\sup_{z\in\mathbb R}
\left|
P\left(
\frac{T-\theta}{\sqrt{\widehat{\tau}}}
\leq z
\right)-\Phi(z)
\right|
\leq
B_T
+
\frac{2}{\sqrt{2\pi e}}\;\eta
+
P\left(
\left|
\frac{\widehat{\tau}}{\tau}-1
\right|>\eta
\right).
$$
\end{proof}

\subsection{Statement and proof of Proposition \ref{prop:sig_k_con}}
\begin{proposition}\label{prop:sig_k_con}
    Assume that $k \mid n$ and that $0 < \sigma^{2}_{k} < \infty$. Consider the set $\left\{ h(S), \;S \in D^{\perp}_n \right\}$, where $D^{\perp}_n$ is a sequence of $1$-equireplicate designs of size $n/k$. Let $s^2_k$ be the standard unbiased sample variance estimator calculated over $\left\{ h(S), \;S \in D^{\perp}_n \right\}$. Then $s^2_k \xrightarrow{p} \sigma^2_k$.
\end{proposition}
\begin{proof}
    First of all, note that the random variables in the set $\left\{ h(S), \;S \in D^{\perp}_n \right\}$ are identically distributed since $X_1, \ldots, X_n$ are i.i.d., and $h$ is a fixed, symmetric and measurable kernel function. Moreover, the random variables in the previously defined set are also independent. This is because $D^{\perp}_n$ is a sequence of $1$-equireplicate designs of size $n/k$ and thus-by construction-the random variables cannot have indices in common. Consequently, we can use well-known results that hold for i.i.d. random variables when $\sigma^2_k < \infty$, which is assumed to hold. Indeed, we can write that:

\begin{equation*}
    \begin{aligned}
s^2_k & = \frac{1}{|D^{\perp}_n| - 1}\left\{\sum_{S \in D^{\perp}_n} \left(h(S) - U_{n, D^{\perp}_n}^{(k)} \right)^2\right\} \\
& = \frac{|D^{\perp}_n|}{|D^{\perp}_n| - 1}  \left\{\frac{\sum_{S \in D^{\perp}_n} h(S)^2}{|D^{\perp}_n|} - \left(U_{n, D^{\perp}_n}^{(k)}\right)^2 \right\} 
\end{aligned}
\end{equation*}

Now, for the SLLN $U_{n, D^{\perp}_n}^{(k)} \xrightarrow{\text { a.s. }} \mu_k$ and $\frac{\sum_{S \in D^{\perp}_n} h(S)^2 }{|D^{\perp}_n|}  \xrightarrow{\text { a.s. }} \mathbb{E}\left[ h(S)^2\right] = \sigma^2_k + \mu^2_k $. Then, we apply the continuous mapping theorem to obtain $\left(U_{n, D^{\perp}_n}^{(k)}\right)^2 \xrightarrow{\text { a.s. }} \mu^2_k$. At this point, by making use of all previous results, we can conclude that $s^2_k \xrightarrow{\text { a.s. }} \sigma^2_k$. This ends the proof since almost sure convergence implies $s^2_k \xrightarrow{p} \sigma^2_k$.   
\end{proof}

\subsection{Proof of Corollary \ref{cor:BE_equi_stud}}

\begin{proof}
Under complete degeneracy,
\[
\operatorname{Var}U_{n,D^{\dagger}}^{(k)}
=
\frac{\sigma_k^2}{|D^{\dagger}|}
=
\frac{k\sigma_k^2}{nr},
\]
and therefore
\[
\frac{\widehat{\operatorname{Var}}U_{n,D^{\dagger}}^{(k)}}
{\operatorname{Var}U_{n,D^{\dagger}}^{(k)}}
=
\frac{s_k^2}{\sigma_k^2}.
\]
Since the blocks in $D^{\perp}$ are pairwise disjoint,
$\{h(S),\,S\in D^{\perp}\}$ are i.i.d. and the variance of their unbiased sample variance is
\[
\operatorname{Var}(s_k^2)
=
\frac{1}{|D^{\perp}|}
\left[
\|h(S)-\mu_k\|_4^4
-
\frac{|D^{\perp}|-3}{|D^{\perp}|-1}\sigma_k^4
\right].
\]
Since $|D^{\perp}|\geq 3$, the second term in brackets is non-negative,
and hence
\[
\operatorname{Var}(s_k^2)
\leq
\frac{1}{|D^{\perp}|}\|h(S)-\mu_k\|_4^4
=
\frac{k}{n}\|h(S)-\mu_k\|_4^4.
\]
Thus, by Chebyshev's inequality, for every $\eta>0$,
\begin{align}
P\left(
\left|
\frac{\widehat{\operatorname{Var}}U_{n,D^{\dagger}}^{(k)}}
{\operatorname{Var}U_{n,D^{\dagger}}^{(k)}}-1
\right|>\eta
\right)
&=
P\left(
\left|\frac{s_k^2}{\sigma_k^2}-1\right|>\eta
\right)
\nonumber\\
&\leq
\frac{k}{n\eta^2}
\left(
\frac{\|h(S)-\mu_k\|_4}{\sigma_k}
\right)^4.
\label{eq:sample_var_tail}
\end{align}

Suppose first that
\[
\eta = \left[
\frac{k}{n}
\left(
\frac{\|h(S)-\mu_k\|_4}{\sigma_k}
\right)^4
\right]^{1/3}
<\frac12.
\]

Applying Theorem~\ref{theo:gen_stud} with $T= U_{n,D^{\dagger}}^{(k)}\;$, $\theta = \mu_k\;$, $\tau = \operatorname{Var}U_{n,D^{\dagger}}^{(k)}\;$, the chosen value for $\eta$, and using both Corollary~\ref{cor:BE_equi} and
\eqref{eq:sample_var_tail}, gives the following expression
\begin{align*}
\sup_z
\left|
P\left(
\frac{U_{n,D^{\dagger}}^{(k)}-\mu_k}
{\sqrt{\widehat{\operatorname{Var}}U_{n,D^{\dagger}}^{(k)}}}
\leq z
\right)-\Phi(z)
\right|
&\leq
C\{k(r-1)+1\}^{2}
\left(\frac{k}{nr}\right)^{1/2}
\left(
\frac{\|h(S)-\mu_k\|_4}{\sigma_k}
\right)^3 \\
&\quad+
\frac{2}{\sqrt{2\pi e}}\eta
+
\frac{k}{n\eta^2}
\left(
\frac{\|h(S)-\mu_k\|_4}{\sigma_k}
\right)^4.
\end{align*}
By the chosen value of $\eta$,

\[\frac{k}{n\eta^2}
\left(
\frac{\|h(S)-\mu_k\|_4}{\sigma_k}
\right)^4
=
\eta,
\]
so that the two terms arising from variance estimation satisfy $
\frac{2}{\sqrt{2\pi e}}\eta+\eta
\leq 2\eta$. Then
\begin{align*}
\sup_z
\left|
P\left(
\frac{U_{n,D^{\dagger}}^{(k)}-\mu_k}
{\sqrt{\widehat{\operatorname{Var}}U_{n,D^{\dagger}}^{(k)}}}
\leq z
\right)-\Phi(z)
\right|
&\leq
C\{k(r-1)+1\}^{2}
\left(\frac{k}{nr}\right)^{1/2}
\left(
\frac{\|h(S)-\mu_k\|_4}{\sigma_k}
\right)^3 \\
&+
2\left(\frac{k}{n}\right)^{1/3}
\left(
\frac{\|h(S)-\mu_k\|_4}{\sigma_k}
\right)^{4/3}.
\end{align*}

If instead
\[
\left[
\frac{k}{n}
\left(
\frac{\|h(S)-\mu_k\|_4}{\sigma_k}
\right)^4
\right]^{1/3}
\geq\frac12,
\]
then the Kolmogorov distance between two distribution functions is at
most one, while
\[
2\left(\frac{k}{n}\right)^{1/3}
\left(
\frac{\|h(S)-\mu_k\|_4}{\sigma_k}
\right)^{4/3}
\geq 1.
\]
Therefore \eqref{eq:be_equi_stud} also holds in this case, which
completes the proof.
\end{proof}

\subsection{Statement and proof of Theorem \ref{theo:CLT_equi_deg}}

\begin{theorem}[CLT degenerate case of order $k-1$ for equireplicate designs]\label{theo:CLT_equi_deg}

Let $D^{\dagger}_n$ be a sequence of sequence of $r$-equireplicate designs of growing size, with the order $k$ fixed. Assume that $\sigma_{k-1}^2 = 0$, $\sigma^{2}_{k} > 0$ and that $\|h(S)-\mu_k\|_4 < \infty$. Let $D^{\perp}$ be a $1$-equireplicate design such that 
$|D^{\perp}|=\frac{n}{k}\geq 3$ and $s_k^2$ be the unbiased sample variance calculated from
$\{h(S),\,S\in D^{\perp}\}$, assuming that $s_k^2>0$ a.s. Then, if $r=o(n^{1/3})$ as $n \to \infty$, we have     

\begin{equation}\label{eq:clt_dege_equi}
\sqrt{\frac{n \; r}{k}}\;\;\frac{U_{n,D^{\dagger}_n}^{(k)} - \mu_{k}}{s_k} \xrightarrow{\;\;d\;\;} \mathcal{N}(0,1) \;. 
\end{equation}
\end{theorem}

\begin{proof}
All the conditions of Corollary \ref{cor:BE_equi_stud} are satisfied and the bound of equation \eqref{eq:be_equi_stud} applies. Thus, since the bound of Corollary \ref{cor:BE_equi_stud} is $O(r^{3/2} \; n^{-1/2} + n^{-1/3})$ with fixed $k$, the result directly follows by taking the limit as $n \to \infty$, knowing that $r=o(n^{1/3})$.     
\end{proof}

\section{Efficient Construction of Equireplicate Designs}\label{supp:sec4}

When $k$ is fixed, constructing equireplicate designs in linear time with respect to the design size implies (by equation \eqref{eq:equirep} in the main text) that the choice of the replication parameter $r$ directly determines the computational complexity of the incomplete U-statistics. This complexity can range from linear in the number of observations $n$, when $r$ is fixed, up to the polynomial complexity $O(n^k)$ of the complete U-statistic, when $r = \binom{n-1}{k-1}$, since $|\mathcal{B}_k| = \binom{n}{k}$. In Algorithm \ref{alg:k>2Design}, we show a construction that allows $r = O(n)$ when $k$ is fixed.   

\subsection{Construction of r-equireplicate designs when k = 2}\label{supp:sec4_k2}

To facilitate our construction of equireplicate designs, we first introduce the new concept of an \emph{equireplicate partition}, which partitions $\mathcal{B}_2$ into disjoint equireplicate designs.
\begin{definition}\label{def:equi_part}
[Equireplicate Partition]
Let $n$ and $r$ be given. A partition $G_1,\ldots, G_{(n-1)/r}$ of $\mathcal{B}_2$ is an $r$-\emph{equireplicate partition} if each subset $G_g$ is an $r$-equireplicate design. 
\end{definition}

Note that if $G_1,\ldots, G_{(n-1)/r}$ is an $r$-equireplicate partition, then the union of any $q$ of the subsets yields a $qr$-equireplicate design. For $n$ even, this implies that from a $1$-equireplicate partition, we can produce an $r$-equireplicate design for any $r \in \{1,2,\ldots\, n-1\}$. For $n$ odd, it is not possible to have a $1$-equireplicate design as the size would be $n/2$, which is not an integer. Indeed, the only possible $r$-equireplicate designs in this case are for $r$ even. Thus, for $n$ odd, we can consider a $2$-equireplicate partition that enables the construction of an $r$-equireplicate design for any $r \in \{2,4,6,\ldots\, n-1\}$, still by unioning the subsets of the partition.

Based on the above discussion, in Theorem \ref{thm:evenPartition} we construct a $1$-equireplicate partition for $n$ even, and in Theorem \ref{thm:oddPartition} the construction a $2$-equireplicate partition for $n$ odd. To aid understanding of our constructions, we provide a visual representation of the $1$-equireplicate partition in Figure \ref{fig:1fact_K6} for $n=6$ and of the $2$-equireplicate partition in Figure \ref{fig:2fact_K7} when $n=7$.

\begin{figure}[t]
    \centering
   \begin{tikzpicture}[scale=3, every node/.style={circle, fill=white, draw, minimum size=6mm}]
 
  \tikzset{
    classA/.style={draw=red,             thick},                 % solid
    classB/.style={draw=blue,            thick, dashed},         % dashed
    classC/.style={draw=violet,          thick, dotted},         % dotted
    classD/.style={draw=brown!60!yellow, thick, dash dot},       % dash-dot
    classE/.style={draw=teal,            thick, dash dot dot},   % dash-dot-dot
  }

  % Place 6 nodes on a circle
  \foreach \i in {1,...,6} {
    \node (N\i) at ({360/6 * (\i - 1)}:1) {\i};
  }

  % Draw light gray background edges for the full K6
  \foreach \i in {1,...,6} {
    \foreach \j in {1,...,6} {
      \ifnum\i<\j
        \draw[gray!20] (N\i) -- (N\j);
      \fi
    }
  }

  % Color Class 1 — (1,6); (2,5); (3,4)
  \draw[classA] (N1) -- (N6);
  \draw[classA] (N2) -- (N5);
  \draw[classA] (N3) -- (N4);

  % Color Class 2 — (2,6); (3,1); (4,5)
  \draw[classB] (N2) -- (N6);
  \draw[classB] (N3) -- (N1);
  \draw[classB] (N4) -- (N5);

  % Color Class 3 — (3,6); (4,2); (5,1)
  \draw[classC] (N3) -- (N6);
  \draw[classC] (N4) -- (N2);
  \draw[classC] (N5) -- (N1);

  % Color Class 4 — (4,6); (5,3); (1,2)
  \draw[classD] (N4) -- (N6);
  \draw[classD] (N5) -- (N3);
  \draw[classD] (N1) -- (N2);

  % Color Class 5 — (5,6); (1,4); (2,3)
  \draw[classE] (N5) -- (N6);
  \draw[classE] (N1) -- (N4);
  \draw[classE] (N2) -- (N3);
\end{tikzpicture}

\vspace{1cm}

\begin{tikzpicture}
  \def\boxsize{0.7}
  \def\spacing{1.0}

  % Styles to match the graph (color + line pattern)
  \tikzset{
    classA/.style={draw=red,             thick},               % solid
    classB/.style={draw=blue,            thick, dashed},       % dashed
    classC/.style={draw=violet,          thick, dotted},       % dotted
    classD/.style={draw=brown!60!yellow, thick, dash dot},     % dash-dot
    classE/.style={draw=teal,            thick, dash dot dot}, % dash-dot-dot
  }

  % Define number values for the five rectangles
  \def\rectA{{1,6, 2,5, 3,4}}
  \def\rectB{{2,6, 3,1, 4,5}}
  \def\rectC{{3,6, 4,2, 5,1}}
  \def\rectD{{4,6, 5,3, 1,2}}
  \def\rectE{{5,6, 1,4, 2,3}}

  % Rectangles with associated styles
  \foreach \rect/\values/\style in {
      0/\rectA/classA,
      1/\rectB/classB,
      2/\rectC/classC,
      3/\rectD/classD,
      4/\rectE/classE} {

    \foreach \i in {0,...,5} {
      \pgfmathtruncatemacro{\row}{int(\i / 2)}
      \pgfmathtruncatemacro{\col}{mod(\i, 2)}
      \pgfmathsetmacro{\xoffset}{\rect * (2 * \boxsize + \spacing)}
      \pgfmathsetmacro{\x}{\xoffset + \col * \boxsize}
      \pgfmathsetmacro{\y}{-\row * \boxsize}
      \pgfmathtruncatemacro{\val}{\values[\i]}

      % Draw white box with the CLASS line type (color + pattern)
      \filldraw[\style, fill=white]
        (\x, \y) rectangle ++(\boxsize, -\boxsize);

      % Numbers in black (no color)
      \node[text=black] at (\x + 0.5*\boxsize, \y - 0.5*\boxsize)
        {\bfseries \scriptsize \val};
    }
  }
\end{tikzpicture}

    \caption{$1$-Equireplicate Partition for $n=6$ and its representation as a proper edge coloring of $K^{(2)}_6$. The construction of the matchings can be understood by holding ``6'' fixed, and rotating the other numbers clock-wise, which is equivalent to the construction in Theorem \ref{thm:evenPartition}.}
    \label{fig:1fact_K6}
\end{figure}

\begin{figure}[t]
    \centering
\begin{tikzpicture}

  % Line styles shared by the graph and rectangles
  \tikzset{
    classA/.style={draw=red,         thick},          % solid
    classB/.style={draw=blue,        thick, dashed},  % dashed
    classC/.style={draw=forestgreen, thick, dotted}   % dotted
  }

  % =============================
  % Left side: Graph (2-factorization of K7)
  % =============================
  \begin{scope}[xshift=2cm, yshift=0cm, scale=3, every node/.style={circle, fill=white, draw, minimum size=6mm}]
    % Place 7 nodes on a circle
    \foreach \i in {1,...,7} {
      \node (N\i) at ({360/7 * (\i - 1)}:1) {\i};
    }

    % Light gray background edges for K7
    \foreach \i in {1,...,7} {
      \foreach \j in {1,...,7} {
        \ifnum\i<\j
          \draw[gray!20] (N\i) -- (N\j);
        \fi
      }
    }

    % Red edges: consecutive cycle (SOLID)
    \draw[classA] (N1) -- (N2);
    \draw[classA] (N2) -- (N3);
    \draw[classA] (N3) -- (N4);
    \draw[classA] (N4) -- (N5);
    \draw[classA] (N5) -- (N6);
    \draw[classA] (N6) -- (N7);
    \draw[classA] (N7) -- (N1);

    % Blue edges: skip-1 (DASHED)
    \draw[classB] (N1) -- (N3);
    \draw[classB] (N2) -- (N4);
    \draw[classB] (N3) -- (N5);
    \draw[classB] (N4) -- (N6);
    \draw[classB] (N5) -- (N7);
    \draw[classB] (N6) -- (N1);
    \draw[classB] (N7) -- (N2);

    % Forestgreen edges: skip-2 (DOTTED)
    \draw[classC] (N1) -- (N4);
    \draw[classC] (N2) -- (N5);
    \draw[classC] (N3) -- (N6);
    \draw[classC] (N4) -- (N7);
    \draw[classC] (N5) -- (N1);
    \draw[classC] (N6) -- (N2);
    \draw[classC] (N7) -- (N3);
  \end{scope}

  % =============================
  % Right side: Rectangles (borders match line types; numbers black)
  % =============================
  \begin{scope}[xshift=7.5cm, yshift=1.5cm] % Adjust to align height with graph center
    \def\boxsize{0.7}
    \def\spacing{1.0}

    % Define the values
    \def\rectA{{1,2, 2,3, 3,4, 4,5, 5,6, 6,7, 7,1}} % red / solid
    \def\rectB{{1,3, 2,4, 3,5, 4,6, 5,7, 6,1, 7,2}} % blue / dashed
    \def\rectC{{1,4, 2,5, 3,6, 4,7, 5,1, 6,2, 7,3}} % forestgreen / dotted

    % Map rectangles to styles (and colors via classA/B/C)
    \foreach \rect/\values/\style in {
        0/\rectA/classA,
        1/\rectB/classB,
        2/\rectC/classC} {

      \foreach \i in {0,...,13} {
        \pgfmathtruncatemacro{\row}{int(\i / 2)}
        \pgfmathtruncatemacro{\col}{mod(\i, 2)}
        \pgfmathsetmacro{\xoffset}{\rect * (2 * \boxsize + \spacing)}
        \pgfmathsetmacro{\x}{\xoffset + \col * \boxsize}
        \pgfmathsetmacro{\y}{-\row * \boxsize}
        \pgfmathtruncatemacro{\val}{\values[\i]}

        % Each small rectangle uses the SAME border style as the graph
        \filldraw[\style, fill=white]
          (\x, \y) rectangle ++(\boxsize, -\boxsize);

        % Numbers are black for clarity
        \node[text=black] at (\x + 0.5*\boxsize, \y - 0.5*\boxsize)
          {\bfseries \scriptsize \val};
      }
    }
  \end{scope}
\end{tikzpicture}
    \caption{$2$-Equireplicate Partition for $n=7$ and its representation as a decomposition of $K^{(2)}_7$ into disjoint cycles. The construction of the partition is done by holding the left column fixed and ``rotating'' the right column vertically, which is equivalent to the construction in Theorem \ref{thm:oddPartition}. }
    \label{fig:2fact_K7}
\end{figure}

\subsubsection{Statement and proof of Theorem \ref{thm:evenPartition}}
\begin{theorem}
    [1-Equireplicate Partition for $n$ Even]
    \label{thm:evenPartition}
    Let $n$ be a positive even integer. Then there exists a $1$-equireplicate partition  of $\mathcal{B}_2$. One such construction of $G_1,\ldots, G_{(n-1)}$ is as follows:

    The subset $G_g$, for $g=1,\ldots, (n-1)$, consists of the pairs $(i,p_g(i))$ where $i\in \{1,2,\ldots,n-1\}$ and $p_g:\{1,\ldots, n-1\}\rightarrow \{1,\ldots, n\}$ is defined as follows:
    \[p_g(i)=\begin{cases}-i+2g\ \pmod{n-1},&i\neq g\pmod{n-1},\\
    n,& i=g,\end{cases}\]
    where if $-i+2g \pmod{n-1}=0$, we set the value to $n-1$. 
    Note that each pair $(i,j) \in \mathcal{B}_2$ is listed twice within its subset, except for $(g,n)$ which appears only once.
\end{theorem}

\begin{proof}
    First, we establish some basic properties about the maps $p_g$, $g=1,\ldots, n-1$.

\begin{enumerate}[label=(\alph*)]
    \item The map $p'_g(i)=-i+2g \pmod{n-1}$  on $\{1,2,\ldots, n-1\}$ is its own inverse: 
    \[-(-i+2g)+2g\pmod{n-1}=i-2g+2g\pmod{n-1}=i\pmod{n-1}.\]
    \item $p'_g(i)$ has unique fixed point $i=g$: Suppose that $i=-i+2g\pod{n-1}$. This implies that $2i=2g\pmod{n-1}$ and since $n-1$ is odd this implies that $i=g \pmod{n-1}$.
    \item $p'_g$ maps $i$ to distinct values for different $g$: Suppose to the contrary that $-i+2g=-i+2g'\pmod{n-1}$. Then $2g=2g'\pmod{n-1}$ and since $n-1$ is odd,  this implies that $g=g'\pmod{n-1}$. 
    \end{enumerate}
    We see that each subset $G_g$ contains $n/2$ pairs, since $p'_g(i)$ has $n-1$ inputs, a unique fixed point, and is its own inverse. Furthermore, by construction each $G_g$ contains exactly one pair with the value $i$, since $p'_g$ is its own inverse. By property (c) above, we have that $G_g$ and $G_g'$ are disjoint. 

    Since there are $n-1$ disjoint subgroups, each with $n/2$ pairs, and each pair occurs at most once, it follows that all $n(n-1)/2$ pairs are accounted for. Thus, we have a $1$-equireplicate partition  of $\mathcal{B}_2$.
\end{proof}

\subsubsection{Statement and proof of Theorem \ref{thm:oddPartition}}
\begin{theorem}
    [2-Equireplicate Partition for $n$ Odd]\label{thm:oddPartition}
    Let $n$ be a positive odd integer. Then there exists a $2$-equireplicate partition of $\mathcal{B}_2$. One such construction of $G_1,\ldots, G_{(n-1)/2}$ is as follows:

    For $g=1,\ldots,(n-1)/2$, define
    \[G_g = \{\{i,j\}\in \mathcal{B}_2| (i-j)=\pm g\pmod{n}\}.\] 
\end{theorem}

\begin{proof}
First we will establish the following claims:
\begin{enumerate}[label=(\alph*)]
    \item Each $i$ is paired to two distinct values within each $G_g$: Note that for $g\neq -g\; (\mod n)$, since $g\neq 0 \pmod{n}$ and $n$ is odd. Thus, $i+g$ and $i-g \pmod{n}$ are distinct.
    \item No $i$ is ever paired with itself in any $G_g$: Suppose to the contrary that $i\pm g=i \pmod{n}$. But then $\pm g=0\pmod{n}$. However, this contradicts that $g\in \{1,2,\ldots,(n-1)/2\}$ as none of these values are $0\pmod{n}$.
    \item No pair appears in two groups: Let $g,g'\in \{1,2,\ldots, (n-1)/2\}$. Note that $g\neq -g'\pmod{n}$, since $-g'\pmod{n}$ is not a member of the set. Thus, if $i-j=\pm g\pmod{n}$ and $i-j=\pm g'\pmod{n}$, then the signs must be the same. Hence, $g=g'\pmod{n}$. 
\end{enumerate}
We see that each $G_g$ has $n$ pairs, and is a $2$-equireplicate design. Since the $G_g$'s are disjoint, we have a $2$-equireplicate partition of $\mathcal{B}_2$.
\end{proof}

\subsubsection{Proof of Theorem \ref{theo:algo_clt_equi}}
\begin{proof}
We start showing that the output of Algorithm \ref{alg:evenDesign} is an $r$-equireplicate design. Let $p_g$ be as defined in Theorem \ref{thm:evenPartition}. Evaluating at $g+i$, we have:
        \[p_g(g+i)=\begin{cases}
            g-i\pmod{n-1}, & i\neq 0\pmod{n-1}\\
            n, &i=0\pmod{n-1}.
        \end{cases}\]
To avoid duplicates, we restrict $i\in\{1,2,\ldots, n/2-1\}$. We see that Algorithm \ref{alg:evenDesign} produces the union of the first $r$ subsets of the 1-equireplicate partition constructed in Theorem \ref{thm:evenPartition}. Since the $1$-equireplicate partition consists of disjoint 1-equireplicate designs, the output of Algorithm \ref{alg:evenDesign} is an $r$-equireplicate design. 

The output of Algorithm \ref{alg:oddDesign} is the union of the first $r/2$ subsets of the $2$-equireplicate partition constructed in Theorem \ref{thm:oddPartition}. By the same reasoning as in part 1, the result is an $r$-equireplicate design.

Both algorithms can be seen to have computational complexity $O(nr)$ due to the nature of the nested for-loops.

Since the output $D$ of both algorithms is an $r$-equireplicate design for $k=2$, then by definition every $1$-subset of $\{1,2, \ldots, n\}$ is contained in the same number $r$ of pairs of $D$. Thus $D$ is a minimum variance design by Theorem 1 of \cite{lee1990u}, page 195. Moreover, the variance of the incomplete U-statistics induced by $D$ is $\operatorname{Var} U_{n,D}^{(2)} =|D|^{-1} (2(r-1)\sigma_1^2 + \sigma_2^2)$ by \emph{Example 1} in \cite{lee1982incomplete}.       
\end{proof}

\begin{remark}
    An $r$-equireplicate partition can also be interpreted as an $r$-factorization of $K^{(2)}_n$, which is the complete graph with $n$ vertices. This construction decomposes $K^{(2)}_n$ into $r$-regular spanning subgraphs, called $r$-factors, whose edge sets are $r$-equireplicate designs. In particular, when $n$ is even, a $1$-factorization always exists and is equivalent to finding an edge coloring of $K^{(2)}_n$ (\cite{bondy1976graph}, ch.6). In Figure \ref{fig:1fact_K6}, we present the output of \emph{Algorithm 1} when $n=6$ and $r=5$ together with the corresponding edge coloring of $K^{(2)}_6$. On the other hand, when $n$ is odd, $K^{(2)}_n$ cannot be decomposed into $1$-factors but a $2$-factorization always exists and is equivalent to finding disjoint cycles whose union forms $K^{(2)}_n$ \citep{alspach2008wonderful}. In Figure \ref{fig:2fact_K7}, we present the output of \emph{Algorithm 2} for $n=7$ and $r=6$ together with the corresponding disjoint cycles decomposition of $K^{(2)}_7$. In graph theory, the main interest is not only in designing algorithms to build either $r$-factorizations or $r$-factors, but in understanding under which conditions these exist and also how many non-isomorphic factorizations there are (see e.g., ch. 10 of \cite{wallis2016introduction} that discusses the number of possible $1$-factorizations of $K^{(2)}_{2n}$). Thus, our algorithms offer one efficient approach among many potential other ways to construct $r$-equireplicate designs. However, to the best of our knowledge, this is the first comprehensive and constructive treatment for $k = 2$ in the literature of incomplete U-statistics based on deterministic designs.
\end{remark}

\subsection{Construction of r-equireplicate designs when k > 2}\label{supp:sec4_k>2}

In principle, we can extend definition \ref{def:equi_part} to consider an $r$-equireplicate partition of $\mathcal{B}_k$ and aim at building $r$-equireplicate designs by unioning the elements of either a $1$-equireplicate partition if $k \mid n$ or of a $k$-equireplicate partition if $k \nmid n$. However, when $k \mid n$ and $k>2$, even if the existence of a $1$-equireplicate partition is guaranteed by \emph{Baranyai's theorem} \citep{baranyai1975factorization}, there is no known general sequential construction that would allow us to build an $r$-equireplicate design for any $n$ and $r \in \{1,\ldots,\binom{n-1}{k-1} \}$. Indeed, there exist only some efficient algorithmic constructions for $k=3$ and $k=4$ cases \citep{yan2022construction}, but they impose additional divisibility conditions on $n$. The situation becomes even more challenging when $k \nmid n$. In this case, the problem is equivalent to constructing a $k$-factorization of the complete $k$-uniform hypergraph $K_n^{(k)}$. However, even the existence of such a factorization is not guaranteed—let alone an efficient algorithm for its construction—as this remains an open problem in combinatorics (see e.g., \cite{bailey2010hamiltonian, petecki2014cyclic} that discusses cyclic decompositions of $K_n^{(k)}$).

\begin{remark}[Minimum Variance when $k > 2$]
Even if constructing an $r$-equireplicate partition were both feasible and computationally efficient, when $k > 2$ we no longer have the guarantee that the individual $r$-equireplicate designs forming the partition—or \emph{a fortiori}, any union of them—achieve minimum variance. When $k>2$, a sufficient condition for achieving minimum variance is that a design is both equireplicate and linear, as shown in Lemma~\ref{lemma:var_lin_hyper}. We have provided a novel interpretation of this specific class of designs within our equireplicate framework, enabling us to derive a closed-form expression for the variance of the corresponding incomplete U-statistic (see Lemma~\ref{lemma:var_lin_hyper}) as well as its Berry-Esseen bound, derived in Theorem \ref{theo:BE_equi_lin_nondeg}. 
\end{remark}

However, if we settle for a more attainable goal, we can still construct a $r$-equireplicate design in some settings, even when $k>2$, via a partial equireplicate partition. 

\begin{definition}\label{def:partial_equi_part}
[Partial Equireplicate Partition]
Let $n$ and $r$ be given. A collection $G_1,\ldots, G_{q}$ of subsets of $\mathcal{B}_k$ is a \emph{partial} $r$-\emph{equireplicate partition} if $\bigcup_{g=1}^{q} G_g \subset \mathcal{B}_k$, the subsets are mutually disjoint and each subset $G_g$ is an $r$-equireplicate design. 
\end{definition}

In the next theorem, we show that we can construct a partial $k$-equireplicate partition under some conditions on $n$ and $k$, and for any strictly increasing natural number valued sequence of choice. We denote with $\mathcal{C}_n = \{ a \in \mathbb{Z}_n| \; \text{gcd}(a,n) = 1 \}$ the set of coprimes of $n$ and with $\phi(n) = \mid \mathcal{C}_n \mid $ its cardinality, which is known as \emph{Euler's totient function}.

\begin{theorem}
    [Partial $k$-Equireplicate Partition]
    \label{thm:partialPartition}
    Let $k > 2$ and $\eta: \{0, \ldots, k-1\} \to \mathbb{N}_0$ be any strictly increasing natural number valued sequence. Take $n$ to be a positive integer such that $n > 3\;\eta(k-1)\;\{\eta(k-1) - \eta(0)\}$. Then there exists a partial $k$-equireplicate partition of $\mathcal{B}_k$. One such construction is a collection of $\phi(n)$ subsets $G_1,\ldots, G_{q}$, where subset $G_g$ for $g \in \mathcal{C}_n$ is defined as:

\begin{equation}
G_{g}=\bigl\{
     \,\{\,i+g\bigl[\eta(j)- \eta(0)\bigr]\pmod{n}
        \; | \;  j\in \{0, \ldots, k-1\}\}
     \;| \; i\in\mathbb Z_{n}
\bigr\}. \label{def:cyc_des}
\end{equation}

\end{theorem}

\subsubsection{Proof of Theorem \ref{thm:partialPartition}}

\begin{proof}
   The construction of the partial partition relies on the map $f: \mathbb{Z}_n \times \mathcal{C}_n \to D \subset \mathcal{B}_k$ which inputs a couple $(i,g)$ and outputs $\{i+g\bigl[\eta(0)- \eta(0)\bigr], \ldots, i+g\bigl[\eta(k-1)-\eta(0)\bigr]\}$, which is an element of $\mathcal{B}_k$. To avoid potential confusions, we will refer to a given element of $G_g$ as a block. We now establish the following claims on the previously defined map:

   \begin{enumerate}[label=(\alph*)]
       \item \emph{The map $f$ assigns to each couple $(i, g)$ an element of $\mathcal{B}_k$}.

       First of all, consider $(i,g) = (0,1)$, which for sure belongs to $\mathbb{Z}_n \times \mathcal{C}_n$. Since $n > 3\;\eta(k-1)\;[\eta(k-1) - \eta(0)] > \eta(k-1)$ by assumption, the first block of $G_{1}$, generated by $(i,g) = (0,1)$, contains $k$ distinct elements. This also implies that all blocks of $G_{1}$, generated by $(i,1)$, contain $k$ distinct elements because shifting by the same $i \in \mathbb{Z}_n \pmod{n}$ each element of a block of size $k$, preserve the original distinction within the first block of $G_1$. For a generic $G_g$, consider that the first block of each $G_g$, generated by $(0,g)$, has $k$ distinct element since $g \in \mathcal{C}_n$ just permutes the elements of the first block of $G_1$ which are distinct, under $n > \eta(k-1)$. To conclude the proof of this part, we just need to verify that for each $g \in \mathcal{C}_n$, all the blocks of $G_g$ contain $k$ distinct elements. We have already shown that the first block of each $G_g$ contains $k$ distinct elements, but then all other blocks, generated by $(i,g)$ shifting by $i \in \mathbb{Z}_n \pmod{n}$, will contain distinct elements for the same reasoning outlined at the beginning for $G_1$.

       \item \emph{The map $f$ is injective and the subsets forming the partition are mutually disjoint}.

       We prove the injectivity of $f$ by contradiction. Suppose $f$ is not injective, then there exist distinct $(i_1,g_1)$ and $(i_2,g_2)$, meaning that $i_1 = i_2$ and $g_1 = g_2$ cannot occur at the same time, such that $f(i_1,g_1) = f(i_2,g_2)$. We start by introducing a permutation $\sigma:\{0, \ldots, k-1\} \to \{0, \ldots, k-1\}$, defined as a bijection from the set $\{0, \ldots, k-1\}$ onto itself. If $f(i_1,g_1) = f(i_2,g_2)$, then there must exist at least two distinct permutations $\sigma$, since one is the identity that maps an index to itself, such that for all $j \in \{0, \ldots k-1\}$:

\begin{equation}
 i_1 + g_1 [\eta(j) - \eta(0)] \equiv i_2 + g_2 [\eta(\sigma(j)) - \eta(0)] \pmod{n} 
\;.\label{eq:set_perm}
    \end{equation}

The key argument behind the proof, is to notice that if two blocks are equal in our $D$, then all possible subsets of these two blocks of any given size must be equal as well. Thus, we consider the three distinct indices $j,i, l \in \{0,,\ldots,k-1\}$ and solve the related system of $\binom{k}{3}$ congruences:
 
\[
\left\{
\begin{aligned}
& i_1 + g_1 [\eta(j) - \eta(0)] \equiv i_2 + g_2 [\eta(\sigma(j)) - \eta(0)]  \pmod{n} \\
& i_1 + g_1 [\eta(i) - \eta(0)] \equiv i_2 + g_2 [\eta(\sigma(i)) - \eta(0)] \pmod{n}\\
& i_1 + g_1 [\eta(l) - \eta(0)] \equiv i_2 + g_2 [\eta(\sigma(l)) - \eta(0)] \pmod{n}
\end{aligned}
\right.
\]

\[
\left\{
\begin{aligned}
& i_1 - i_2 \equiv g_2 [\eta(\sigma(j)) - \eta(0)] - g_1 [\eta(j) - \eta(0)]  \pmod{n} \\
& i_1 - i_2 \equiv g_2 [\eta(\sigma(i)) - \eta(0)] - g_1 [\eta(i) - \eta(0)] \pmod{n}\\
& i_1 - i_2 \equiv g_2 [\eta(\sigma(l)) - \eta(0)] - g_1 [\eta(l) - \eta(0)] \pmod{n}
\end{aligned}
\right.
\]

Now we take the collections of congruences indexed first by $i$ and then by $l$ and subtract them from the collection indexed by $j$ obtaining

\[
\left\{
\begin{aligned}
& g_2 [\eta(\sigma(j)) - \eta(\sigma(i))] - g_1 [\eta(j) - \eta(i)]  \equiv 0  \pmod{n} \\
& g_2 [\eta(\sigma(j)) - \eta(\sigma(l))] - g_1 [\eta(j) - \eta(l)]  \equiv 0 \pmod{n}
\end{aligned}
\right.
\]

We now multiply both collections of congruences to match the $g_1$ terms and subtract them to finally obtain
$$
g_2\{ [\eta(\sigma(j)) - \eta(\sigma(i))] [\eta(j) - \eta(l)]  - [\eta(\sigma(j)) - \eta(\sigma(l))] [\eta(j) - \eta(i)] \} \equiv 0 
 \pmod{n}.
$$

Since $g_2 \in \mathcal{C}_n$, we can divide both sides of the congruence by $g_2$ and obtain
\small{
\begin{equation}
[\eta(\sigma(j)) - \eta(\sigma(i))] [\eta(j) - \eta(l)]  - [\eta(\sigma(j)) - \eta(\sigma(l))] [\eta(j) - \eta(i)]  \equiv 0  \pmod{n}.    \label{eq:inj_equiv}
\end{equation}
}
\normalsize Now we show that that the absolute value of (\ref{eq:inj_equiv}) is bounded above by the quantity $3\;\eta(k-1)\;[\eta(k-1) - \eta(0)]$. We start by noticing that the LHS of (\ref{eq:inj_equiv}) can be developed to obtain

\begin{equation*}
\text{LHS} = \eta(\sigma(i))\left[\eta(l)-\eta(j)\right]+\eta(\sigma(j))\left[\eta(i)- \eta(l)\right]+\eta(\sigma(l))\left[\eta(j)-\eta(i)\right] .
\end{equation*}

Since the sequence is strictly increasing, we know that for any $u,v \in \{0,\ldots, k-1\}$, $\left| \eta(u)-\eta(v)\right| \leq [\eta(k-1) - \eta(0)]$. But then, for the triangle inequality and Cauchy-Schwarz
\small{
$$
\begin{aligned}
& \left| \text{LHS} \right| & < & \left|\eta(\sigma(i))\right|\left|\eta(l)-\eta(j)\right|+\left|\eta(\sigma(j))\right|\left|\eta(i)-\eta(l)\right|+\left|\eta(\sigma(l))\right|\left|\eta(j)-\eta(i)\right|  \\
&  & < &  \;\eta(k-1) [\eta(k-1) - \eta(0)] + \eta(k-1) [\eta(k-1) - \eta(0)] + \eta(k-1) [\eta(k-1) - \eta(0)]\\
&  & = & \;3\;\eta(k-1)\;[\eta(k-1) - \eta(0)].
\end{aligned}
$$
}
\normalsize Since $ n > 3\;\eta(k-1)\;[\eta(k-1) - \eta(0)]$ by assumption, this implies that all the possible solutions of (\ref{eq:inj_equiv}) must verify
\begin{equation*}
[\eta(\sigma(j)) - \eta(\sigma(i))] \;[\eta(j) - \eta(l)]  = [\eta(\sigma(j)) - \eta(\sigma(l))] \;[\eta(j) - \eta(i)]  .
\end{equation*}

However, the only possible solution is when $j=\sigma(j)$, $i=\sigma(i)$ and $l=\sigma(l)$. This because both the three indices $j$,$i$ and $l$ and their permutations are distinct among themselves, and the sequence is strictly increasing. But this contradicts our original statement on the existence of at least two distinct permutations that guarantee (\ref{eq:set_perm}). Thus the map $f$ is injective. This results implies \emph{a fortiori} that $G_a \cap G_b = \varnothing$ for all distinct $a,b \in \mathcal{C}_n$ meaning that the subsets forming the partition must be mutually disjoint.

       \item \emph{The set of images $D \subset \mathcal{B}_k$ thus $f$ is not surjective}.

       This is a requirement of definition \ref{def:partial_equi_part} to obtain a partial equireplicate partition. Consider that $|\mathcal{B}_k| = \binom{n}{k}$. But since the map is injective under our assumptions, we know that $ |D| = n \; \phi(n)$, which achieves its maximum when $n$ is prime. Thus, if we have that $n \; (n-1)  < \binom{n}{k}$ this ensures that $D = \bigcup_{g=1}^{q} G_g \subset \mathcal{B}_k$. Under our assumptions that $k > 2$ and $ n > 3\;\eta(k-1)\;[\eta(k-1) - \eta(0)]$, $n \; (n-1)  < \binom{n}{k}$ thus $f$ is not surjective.

       \item \emph{Each $G_g$ for $g \in \mathcal{C}_n$ is a $k$-equireplicate design}.

       This last point is easy to show since it comes automatically when using cyclic constructions (see \cite{lee1990u}, starting from page 198, for a detailed explanation). Indeed, we have that $G_1 = \{i+\bigl[\eta(0)- \eta(0)\bigr], \ldots, i+\bigl[\eta(k-1)- \eta(0)\bigr]\}$ and since we have shown that all $n$ elements are distinct, and we know that $i \in \mathbb{Z}_n$, then each element of $\mathbb{Z}_n$ will appear exactly $k$ times. The same holds as well for a generic $G_g$ because we have shown that also all his elements are distinct at the start.
   \end{enumerate}

Since we have verified all our claims, then our proposed construction is a partial $k$-equireplicate partition.

\end{proof}

\subsubsection{Proof of Theorem \ref{theo:algo_clt_equi_k>2}}

\begin{proof}
 For a given value of $g \in \mathcal{C}_{n,r}$, Algorithm \ref{alg:k>2Design} builds the subset $G_g$, as defined in Theorem \ref{thm:partialPartition} where each $G_g$ is a $k$-equireplicate design. As already underlined for the $k=2$ case, see the proof of Theorem \ref{theo:algo_clt_equi}, if $G_1,\ldots, G_{q}$ is a partial $t$-equireplicate partition, then the union of any $s$ of the subsets yields a $st$-equireplicate design. In our case, the output of Algorithm \ref{alg:k>2Design} is the union of the first $r/k$ subsets of the partial $k$-equireplicate partition constructed in Theorem \ref{thm:partialPartition}. Thus, by the previous reasoning with $t=k$ and $s=r/k$, it follows that the output of Algorithm \ref{alg:k>2Design} is an $r$-equireplicate design. 
 Regarding the computational complexity, we start by noticing that building $\mathcal{C}_{n,r} = \{a \in \{1,\ldots, r/k \}| \; \text{gcd}(a,n) = 1 \}$ requires $O(r/k \log(n))$ using the standard Euclidean algorithm. Then, there are three nested for-loops with total computational complexity $O(nr)$, since the first loop concerns $r/k$ iterations, the second $n$ and the third $k$. Therefore, asymptotically Algorithm \ref{alg:k>2Design} runs in $O(nr)$. 
\end{proof}

\section{Numerical Experiments}\label{supp:sec5}

\subsection{Experiments on HSIC}\label{subsec:hsic}

Following the same approach of the first MMD study, we conduct a series of independence tests using incomplete versions of the unbiased HSIC (uHSIC) statistic—constructed via equireplicate designs—to assess deviations from normality in their asymptotic distributions. Under $H_0$—i.e., when the two samples are independent—the complete uHSIC statistic (see the $\mathrm{HSIC}_s(Z)$ definition in \cite{gretton2007kernel}) is degenerate of order $1$ and has a non-Gaussian limiting distribution. Under $H_1$, it is non-degenerate and converges to a normal (see Theorem~1 and 2 in \cite{gretton2007kernel} for details). 

More specifically, we simulate each time both samples $(X,Y)$ i.i.d. from a $\mathcal{N}(0,1)$. We choose $n \in \{200, 400,800,1600\}$. The incomplete uHSIC statistic is computed from eq.~(23) of \cite{schrab2025practical}, accounting for kernel symmetrization, using a linear kernel and an $r$-equireplicate design constructed by Algorithm~\ref{alg:k>2Design}, with $k=4$ and $\eta(j)=j$ to induce a high degree of overlap (see Remark \ref{remark:euler}). We select $r \in \{1,n^{0.33},n^{0.5},n^{0.75}\}$\footnote{Note that $r$ must be a multiple of $k=4$. If it is not the case, we select the nearest multiple. Moreover, when $r=1$, we use the design $D_n^{\perp}$, described in Section \ref{subsec:stud}.}. We obtain a sampling distribution of the KS distance between the empirical distribution of the standardized uHSIC statistic and $\mathcal{N}(0,1)$ as described in the first MMD study. The left panel of Figure~\ref{fig:HSIC_simstudy} reports the mean KS distance together with $95\%$ Monte Carlo confidence intervals (CI) when the uHSIC statistic is standardized in each replicate using the variance estimator from Theorem~\ref{theo:CLT_equi_deg}, with $k=4$. The right panel reports the same quantities when the uHSIC statistic is, instead, standardized using its true variance, approximated via Monte Carlo. For reference, both panels also show the mean KS distance under exact $\mathcal{N}(0,1)$ sampling.

As shown in the left panel of Figure~\ref{fig:HSIC_simstudy}, when the uHSIC statistic is standardized using the variance estimator from Theorem~\ref{theo:CLT_equi_deg}, the mean KS distance remains clearly above the normal reference level for all values of $r$, indicating persistent departures from normality. This behavior is expected, since under $H_0$ the uHSIC kernel is not fully degenerate and the variance estimator is therefore not appropriate in this setting. In contrast, when standardizing by the true variance, as in the right panel, the mean KS distances are close to the exact $\mathcal{N}(0,1)$ reference and decrease with $n$, consistently with the Gaussian approximation predicted by our theory (see Corollary \ref{cor:BE_equi}).

\begin{figure}[tt]
    \centering
    \includegraphics[width=\linewidth]{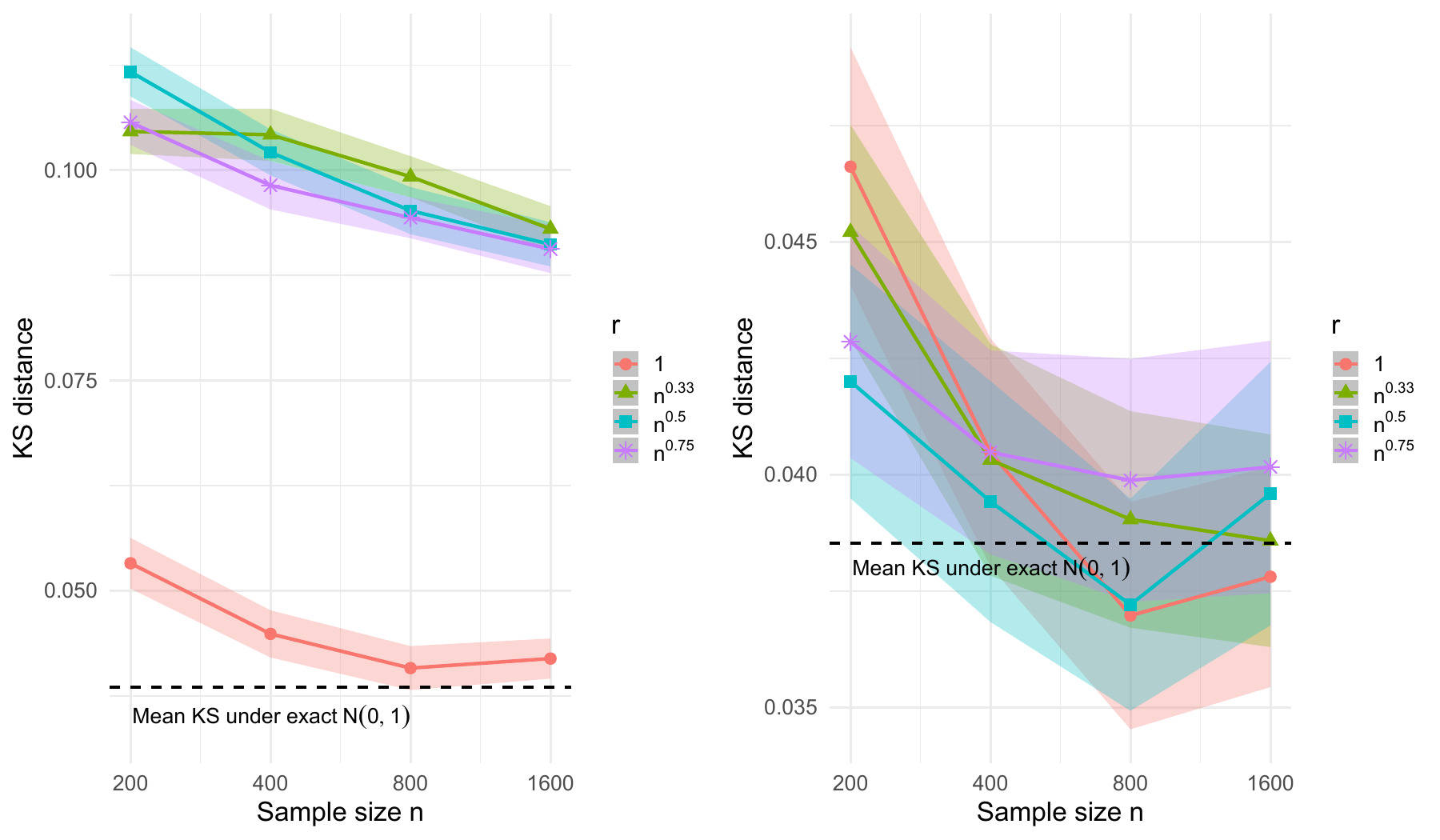}
    \caption{(Left) KS distance between the empirical distribution of the incomplete uHSIC statistic under $H_0$, standardized using the variance estimator from Theorem~\ref{theo:CLT_equi_deg} with $k=4$, and the $\mathcal{N}(0,1)$ distribution. (Right) Corresponding KS distance when the incomplete uHSIC statistic is standardized using its true variance. $95\%$ Monte Carlo CIs are included in both panels. Horizontal axes are displayed on a $\log_2$ scale, with tick labels corresponding to the original sample sizes.}
    \label{fig:HSIC_simstudy}
\end{figure}

\subsection{Experiments on $\ell_4$ norm}

The choice of $\eta(j)$ in Algorithm \ref{alg:k>2Design} controls the degree of overlap in the resulting design. In Figure~\ref{fig:Lpnorm_simstudy_var}, we repeat the variance experiment of Section \ref{subsec:l4} using first $\eta(j)=j$ (left panel) and then $\eta(j)=j^2$ (right panel)\footnote{We do not consider $n=200$ as the starting sample size of this experiment because the condition of Theorem \ref{theo:algo_clt_equi_k>2} requires $n > 243$.}. When $\eta(j)=j$, the higher degree of overlap can lead to efficiency losses, while we have efficiency gains when $\eta(j)=j^2$. Overall--since the largest efficiency gains were observed for $\eta(j)=2^j$ (see Figure \ref{fig:l4_simstudy})--we have evidence that as $\eta(j)$ grows faster, equireplicate designs exhibit larger variance gains relative to random designs.

\begin{figure}[H]
    \centering
    \includegraphics[width=\linewidth]{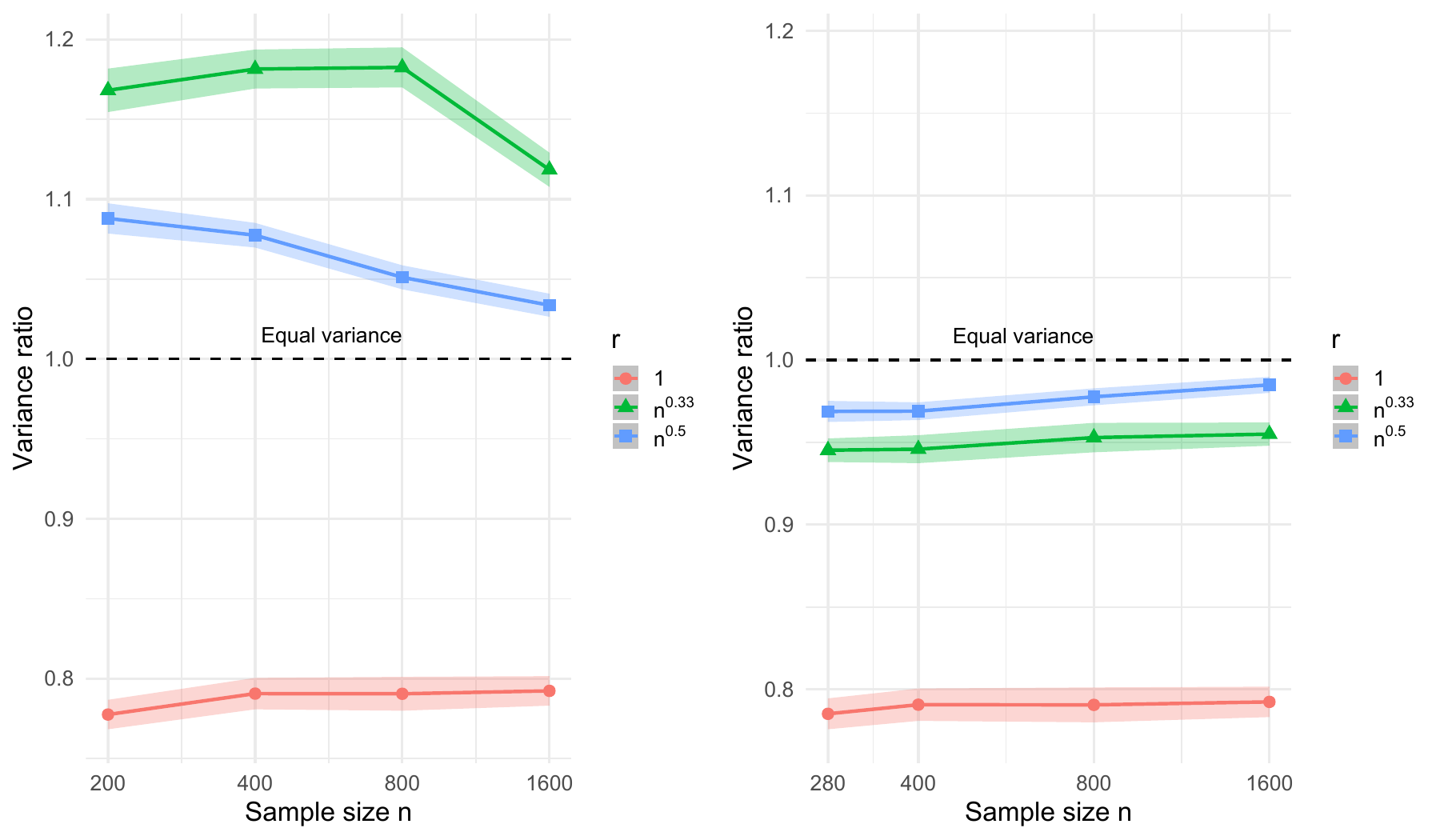} 
    \caption{Variance ratio under $H_1$ between the incomplete uMMD statistic based on equireplicate designs and that based on random designs. (Left) $\eta(j) = j$, and (Right) $\eta(j) = j^2$. $95\%$ Monte Carlo CIs are included for both experiments. Horizontal axes are displayed on a $\log_2$ scale, with tick labels corresponding to the original sample sizes.}
    \label{fig:Lpnorm_simstudy_var}
\end{figure}

\subsection{Real data example: CIFAR-10 dataset}

We perform an additional simulation study, again using the CIFAR-10 dataset as in Section~\ref{subsec:real_data}, to determine which value of the constant $c$ in $r=c\,n^{0.33}$ preserves the nominal type~I error of our permutation-free (PF) MMD two-sample test. More specifically, $X$ and $Y$ are drawn from the same distribution described in Section~\ref{subsec:real_data}, and we vary the common sample size $n \in \{50,100,200,400,800,1600\}$. The incomplete uMMD statistic is computed using a Gaussian kernel with bandwidth selected by the standard median heuristic and an $r$-equireplicate design generated by Algorithm~\ref{alg:evenDesign}. We standardize the statistic using the variance estimator from Theorem~\ref{theo:CLT_equi_deg}. We set $r=c\,n^{0.33}$ with $c \in \{2,3,4,5,6,7\}$ and assess which choices of $c$ preserve the nominal level $\alpha=0.05$. For each configuration, we repeat the experiment $500$ times to obtain Monte Carlo estimates of the type~I error. Figure~\ref{fig:realdata_const} reports the corresponding $95\%$ Monte Carlo CIs.

We observe that, for $c \in \{2,3,4,5\}$, the estimated type~I error remains broadly stable around the nominal level $\alpha=0.05$ across sample sizes. In contrast, for $c \in \{6,7\}$, the estimates exhibit a noticeable upward excursion at $n=400$, suggesting that larger multiplicative constants may lead to less stable finite-sample type~I error control. We therefore adopt the more conservative choice $c=5$ in the simulation study of Section~\ref{subsec:real_data}.

\begin{figure}[t]
    \centering
    \includegraphics[width=\linewidth]{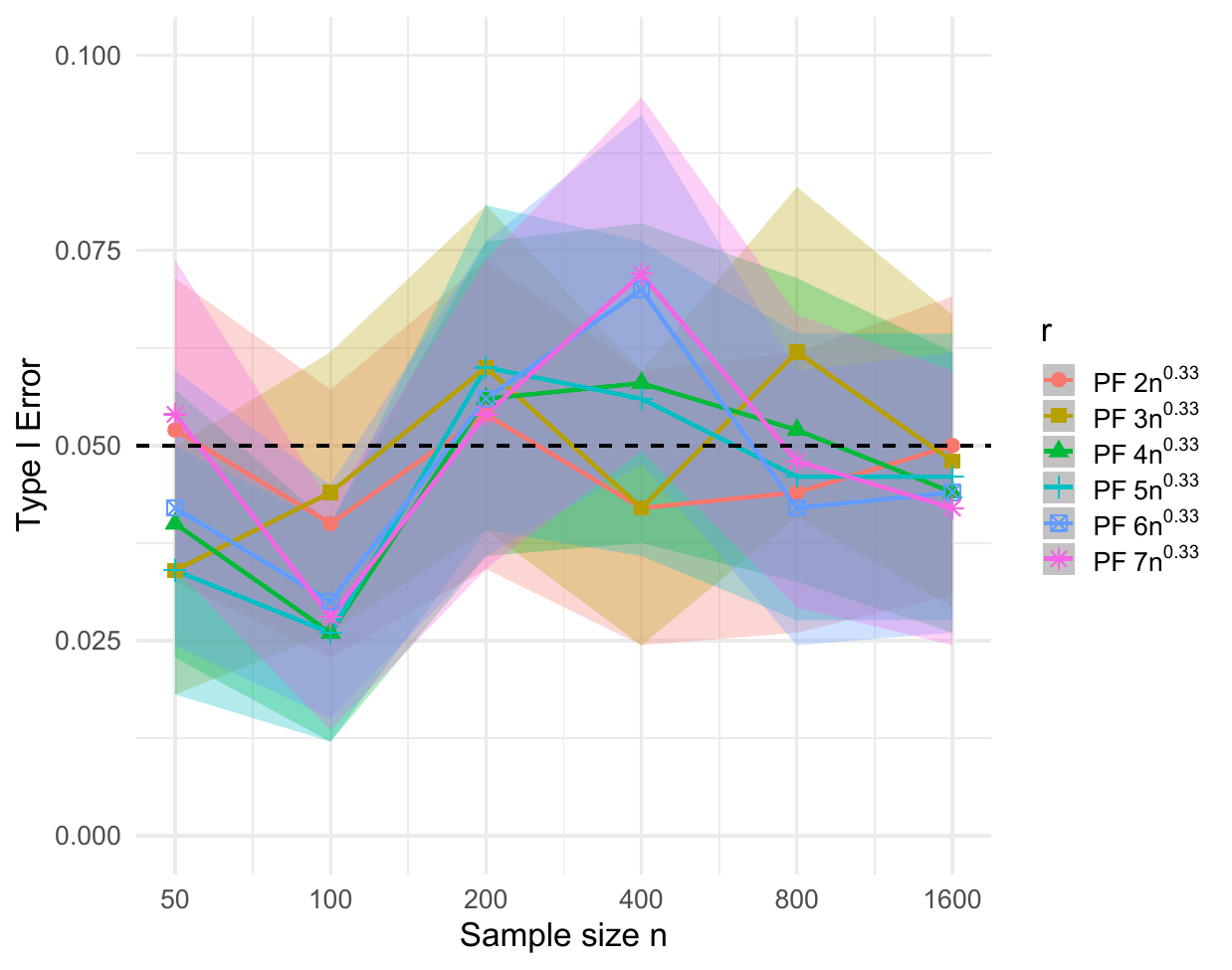}
    \caption{$95\%$ Monte Carlo CIs for the type I error of the permutation-free (PF) version of the MMD test on CIFAR-10 for different values of $n$ and $r = c\;n^{0.33}$ with $c \in \{2,3,4,5,6,7 \}$. Horizontal axes are displayed on a $\log_2$ scale, with tick labels corresponding to the original sample sizes.}
    \label{fig:realdata_const}
\end{figure}
\end{document}